\documentclass[twoside,a4paper,reqno,11pt]{amsart} 
\usepackage[top=28mm,right=28mm,bottom=28mm,left=28mm]{geometry}
\usepackage{amsfonts, amsmath, amssymb, amsbsy, mathrsfs, bm, latexsym, stmaryrd, hyperref, array, enumitem}

\renewcommand{\a}{\alpha}
\renewcommand{\b}{\beta}
\newcommand{\g}{\gamma}
\newcommand{\e}{\epsilon}
\renewcommand{\l}{\lambda} 
\newcommand{\s}{\sigma}

\renewcommand{\O}{\Omega}

\newcommand{\C}{\mathcal{C}}

\newcommand{\M}{\mathcal{M}}

\newcommand{\<}{\langle}
\renewcommand{\>}{\rangle}
\newcommand{\la}{\langle}
\newcommand{\ra}{\rangle}
\newcommand{\leqs}{\leqslant}
\newcommand{\geqs}{\geqslant}
\newcommand{\normeq}{\trianglelefteqslant}

\renewcommand{\to}{\rightarrow}

\newcommand{\fpr}{\mbox{{\rm fpr}}}

\newcommand{\Fix}{\mathrm{Fix}}

\newcommand{\what}{\widehat}

\newcommand{\vs}{\vspace{3mm}}

\newcommand{\on}{\mathbf{1}}
\newcommand{\tw}{\mathbf{2}}
\newcommand{\tr}{\mathbf{3}}
\newcommand{\fo}{\mathbf{4}}
\newcommand{\fv}{\mathbf{5}}
\newcommand{\si}{\mathbf{6}}
\newcommand{\bA}{\mathbf{A}}
\newcommand{\bB}{\mathbf{B}}
\newcommand{\bC}{\mathbf{C}}

\makeatletter
\newcommand{\imod}[1]{\allowbreak\mkern4mu({\operator@font mod}\,\,#1)}
\makeatother

\newtheoremstyle{custom}{}{}{\itshape}{0pt}{\bfseries}{.}{5pt plus 1pt minus 1pt}{}%
\theoremstyle{custom}
\newtheorem{thmstar}{Theorem}

\theoremstyle{plain}
\newtheorem{theorem}{Theorem}

\newtheorem*{conj*}{Conjecture}
 
\newtheorem{corol}[theorem]{Corollary}

\newtheorem{thm}{Theorem}[section] 
\newtheorem{lem}[thm]{Lemma}
\newtheorem{prop}[thm]{Proposition} 

\newtheorem{cor}[thm]{Corollary}

\theoremstyle{definition}
\newtheorem{rem}[thm]{Remark}
\newtheorem{remk}{Remark}
\newtheorem*{nota}{Notation}
\newtheorem{defn}[thm]{Definition}

\newtheorem*{def-non}{Definition}

\begin{document}

\author{Timothy C. Burness}
\address{T.C. Burness, School of Mathematics, University of Bristol, Bristol BS8 1UG, UK}
\email{t.burness@bristol.ac.uk}
 
\author{Scott Harper}
\address{S. Harper, School of Mathematics, University of Bristol, Bristol BS8 1UG, UK}
\email{scott.harper@bristol.ac.uk}
 
\title[Finite groups, $2$-generation and the uniform domination number]{Finite groups, $2$-generation and the uniform \\ domination number}

\date{\today}

\begin{abstract}
Let $G$ be a finite $2$-generated non-cyclic group. The spread of $G$ is the largest integer $k$ such that for any nontrivial elements $x_1, \ldots, x_k$, there exists $y \in G$ such that $G = \la x_i, y\ra$ for all $i$. The more restrictive notion of uniform spread, denoted $u(G)$, requires $y$ to be chosen from a fixed conjugacy class of $G$, and a theorem of Breuer, Guralnick and Kantor states that $u(G) \geqs 2$ for every non-abelian finite simple group $G$. For any group with $u(G) \geqs 1$, we define the uniform domination number $\gamma_u(G)$ of $G$ to be the minimal size of a subset $S$ of conjugate elements such that for each nontrivial $x \in G$ there exists $y \in S$ with $G = \la x, y \ra$ (in this situation, we say that $S$ is a uniform dominating set for $G$). We introduced the latter notion in a recent paper, where we used probabilistic methods to determine close to best possible bounds on $\gamma_u(G)$ for all simple groups $G$. 

In this paper we establish several new results on the spread, uniform spread and uniform domination number of finite groups and finite simple groups. For example, we make substantial progress towards a classification of the simple groups $G$ with $\gamma_u(G)=2$, and we study the associated probability that two randomly chosen conjugate elements form a uniform dominating set for $G$. We also establish new results concerning the $2$-generation of soluble and symmetric groups, and we present several  open problems.
\end{abstract}

\maketitle

\setcounter{tocdepth}{1}
\tableofcontents

\section{Introduction}\label{s:intro}

Let $G$ be a finite non-cyclic group that can be generated by two elements. It is natural to study the properties of generating pairs for $G$ and such problems have attracted a great deal of attention over several decades, especially in the context of finite simple groups. Here we begin by introducing the generation invariants and associated probabilities that will be the main focus of this paper. 

In \cite{BW75}, Brenner and Wiegold define the \emph{spread} of $G$, denoted $s(G)$, to be the largest integer $k$ such that for any nontrivial elements $x_1, \ldots, x_k$ in $G$, there exists $y \in G$ such that $G = \la x_i, y \ra$ for all $i$. This leads naturally to the more restrictive notion of \emph{uniform spread}, denoted $u(G)$, which was introduced more recently by Breuer, Guralnick and Kantor \cite{BGK}. This is defined to be the largest integer $k$ such that there is a conjugacy class $C$ of $G$ with the property that for any nontrivial elements $x_1, \ldots, x_k$ in $G$, there exists $y \in C$ such that  $G = \la x_i, y \ra$ for all $i$. Clearly, 
\[
s(G) \geqs u(G) \geqs 0.
\]
It is easy to see that if $s(G) \geqs 1$, then every proper quotient of $G$ is cyclic. In fact, it is conjectured that this condition on quotients is equivalent to the positive spread property; see \cite[Conjecture 1.8]{BGK}, and recent progress towards a proof of this conjecture in \cite{BG,Harper,Harper2}. 

In \cite{BH}, we introduced some new generation invariants, which can be viewed as natural extensions of spread and uniform spread. Following \cite{BH}, we say that a subset $S \subseteq G$ of nontrivial elements is a \emph{total dominating set} (TDS) for $G$ if for all nontrivial $x \in G$, there exists $y \in S$ such that $G = \la x, y \ra$. To explain the terminology, note that if $\Gamma(G)$ is the \emph{generating graph} of $G$, whose vertices are the nontrivial elements of $G$ and $x,y \in G$ are adjacent if and only if $G = \la x,y \ra$, then $S$ is a TDS for $G$ if and only if it is a total dominating set for $\Gamma(G)$ in the usual graph-theoretic sense. Consequently, if $s(G) \geqs 1$, then the \emph{total domination number} of $G$ is defined by 
\[
\gamma_t(G) = \min\{|S| \, : \, \mbox{$S$ is a TDS for $G$}\}.
\]
Similarly, if $u(G) \geqs 1$, then $G$ contains a \emph{uniform dominating set} (UDS), which is defined to be a TDS of conjugate elements, and the \emph{uniform domination number} of $G$ is 
\[
\gamma_u(G) = \min\{|S|\,:\, \mbox{$S$ is a UDS for $G$}\}.
\]
Observe that if $u(G) \geqs 1$, then 
\[
2 \leqs \gamma_t(G) \leqs \gamma_u(G) \leqs |C|
\]
for some conjugacy class $C$ of $G$ (if $u(G)=0$, then $\gamma_u(G)$ is undefined).

Probabilistic methods play an important role in the study of uniform spread and the uniform domination number. For an element $s \in G$ and a positive integer $c$, we define 
\begin{equation}\label{e:pgsc}
P(G,s,c) = \frac{|\{(x_1, \ldots, x_c) \in (s^G)^c \,:\, \mbox{$\{x_1, \ldots, x_c\}$ is a UDS for  $G$}\}|}{|s^G|^c},
\end{equation}
the probability that $c$ randomly chosen conjugates of $s$ form a UDS for $G$. In addition, we define
\begin{equation}\label{e:PG}
P_c(G) = \max\{P(G,s,c) \,: \, s \in G\},
\end{equation}
so $\gamma_u(G) \leqs c$ if and only if $P_c(G)>0$. 

\vs

There is a vast literature on the remarkable generation properties of finite simple groups. The starting point is a theorem of Steinberg \cite{St}, which states that every finite simple group of Lie type is $2$-generated. It is easy to see that every alternating group is $2$-generated and the same is true for all sporadic simple groups by a theorem of Aschbacher and Guralnick \cite{AG}. Therefore, by appealing to the Classification of Finite Simple Groups, we conclude that every finite simple group is $2$-generated (without the classification, there is no known bound on the number of generators needed for a finite simple group). This observation leads to many natural problems concerning the distribution of generating pairs across a simple group, which have been intensively studied in recent years (see the recent survey article \cite{B_sur} for more details). 

The main result on the uniform spread of simple groups is the following theorem, which combines results from \cite{BGK} and \cite{GSh}.

\begin{thmstar}\label{t:star}
\mbox{ } 
\begin{itemize}\addtolength{\itemsep}{0.2\baselineskip}
\item[{\rm (i)}] If $G$ is a finite non-abelian simple group, then $u(G) \geqs 2$, with equality if and only if 
\[
G \in \{ A_5, A_6, \O_{8}^{+}(2), {\rm Sp}_{2r}(2) \, (r \geqs 3)\}.
\]
\item[{\rm (ii)}] Let $(G_n)$ be a sequence of finite non-abelian simple groups  with $|G_n| \to \infty$. Then either $u(G_n) \to \infty$, or there is an infinite subsequence consisting of either

\vspace{1mm}

\begin{itemize}\addtolength{\itemsep}{0.2\baselineskip}
\item[{\rm (a)}] alternating groups of degree all divisible by a fixed prime; or 
\item[{\rm (b)}] odd-dimensional orthogonal groups over a field of fixed size; or
\item[{\rm (c)}] symplectic groups over a field of even characteristic and fixed size.
\end{itemize}
\end{itemize}
\end{thmstar}

Here part (i) is due to Breuer, Guralnick and Kantor \cite[Theorem 1.2]{BGK}, extending earlier work of Guralnick and Kantor \cite{GK}, and independently Stein \cite{Stein}, who established the weaker bound $u(G) \geqs 1$. The asymptotic statement in part (ii) is a theorem of Guralnick and Shalev \cite[Corollary 2.3]{GSh}. Note that the cases described in parts (a)--(c) of part (ii) are genuine exceptions. For example, if $n \geqs 6$ is composite, then \cite[Proposition 2.4]{GSh} gives 
\begin{equation}\label{e:pan}
s(A_{n}) <\binom{2p+1}{3},
\end{equation}
where $p$ is the smallest prime divisor of $n$. Similarly, \cite[Proposition 2.5]{GSh} states that $s({\rm Sp}_{2m}(q)') \leqs q$ if $q$ is even and $s(\O_{2m+1}(q)) \leqs \frac{1}{2}q(q+1)$ if $q$ is odd, for all $m,q \geqs 2$.

In view of Theorem \ref{t:star}, it is natural to study the uniform domination number of simple groups. In \cite{BH}, we developed probabilistic and computational methods to study $\gamma_u(G)$ and we used these techniques to determine close to best possible bounds for simple groups. For example, we showed that there are infinitely many simple groups $G$ with $\gamma_u(G)=2$, including the alternating groups $A_n$ when $n \geqs 13$ is prime. In contrast, we proved that the uniform domination number is in general unbounded for finite simple groups; for instance,
\[
\gamma_u(A_n) \geqs \lceil \log_2 n \rceil - 1
\]
for all even integers $n \geqs 6$ (see \cite[Theorems 1 and 2]{BH}). 

A key observation in \cite{BH} is the connection between the uniform domination number and the classical concept of \emph{bases} in permutation group theory, which allows us to apply recent work on bases for almost simple primitive groups. In order to explain this relationship, recall that if $G$ acts faithfully on a finite set $\O$, then a subset of $\O$ is a base for $G$ if its pointwise stabiliser in $G$ is trivial. We write $b(G,\O)$ for the \emph{base size} of $G$, which is the minimal size of a base. The connection in \cite{BH} arises from the easy observation that if there is an element $s \in G$ contained in a unique maximal subgroup $H$ of $G$, then $P(G,s,c)>0$ if and only if $b(G,G/H) \leqs c$, whence $\gamma_u(G) \leqs b(G,G/H)$ (see Section \ref{ss:prelims_udn} for further details).

Our initial investigations in \cite{BH} lead to a number of natural problems, which we seek to address in this paper. For example, one of our aims is to extend the study of the simple groups $G$ with the extremal property $\gamma_u(G) = 2$, with a view towards a complete classification. For such a group $G$, we will also investigate the corresponding probability $P_2(G)$ (see \eqref{e:PG}) and its asymptotic properties (with respect to a sequence of such groups). We will also revisit earlier work of Binder \cite{Binder68,Binder70I,Binder70II} from the 1960s on the spread of symmetric groups, together with results of Brenner and Wiegold \cite{BW75} from the 1970s on soluble groups and the simple linear groups ${\rm L}_{2}(q)$. In particular, we take the opportunity to bring together several important results on the $2$-generation of finite groups that are somewhat scattered through the literature.

\vs

We now present the main results of the paper. Our first theorem concerns soluble groups.

\begin{theorem}\label{t:solmain}
Let $G$ be a finite non-abelian soluble group such that every proper quotient is cyclic. 
Then $G=N{:}H$, where $N = (C_p)^f$ for some prime $p$ and integer $f \geqs 1$, and $H$ is cyclic and acts faithfully and irreducibly on $N$. Moreover, the following hold:
\begin{itemize}\addtolength{\itemsep}{0.2\baselineskip}
\item[{\rm (i)}] $s(G) = |N| - \e$ and $u(G) = |N| - 1$, where $\e = 0$ if $|H|$ is a prime and $\e=1$ otherwise.
\item[{\rm (ii)}] $\gamma_u(G) = 2$ and $P_2(G) = 1-|N|^{-1}$.
\end{itemize}
\end{theorem}

\begin{remk}\label{r:solmain}
Let us make some comments on the statement of Theorem \ref{t:solmain}:
\begin{itemize}\addtolength{\itemsep}{0.2\baselineskip}
\item[{\rm (a)}] In (i), the result on $s(G)$ is due to Brenner and Wiegold (see \cite[Theorem 2.01]{BW75}).
\item[{\rm (b)}] Let $G = N{:}H$ be a soluble group as in the theorem and let 
$s \in G$ be a nontrivial element. We will show that $P(G,s,2)>0$ if and only if $s$ generates a complement of $N$ (see Proposition \ref{p:soluble_witness}), in which case any two distinct conjugates of $s$ form a uniform dominating set and thus  $P(G,s,2) = 1 - |N|^{-1}$ since $s$ is self-centralising.
\item[{\rm (c)}] The corresponding result for a non-cyclic abelian group $G$ is transparent. Indeed, the condition on quotients implies that $G = C_p \times C_p$ for some prime $p$ and it is easy to see that $s(G) = p$ and $u(G)=0$.
\end{itemize}
\end{remk}

Next we turn to symmetric groups. 

\begin{theorem}\label{t:symmain}
Let $G = S_n$ with $n \geqs 5$. 
\begin{itemize}\addtolength{\itemsep}{0.2\baselineskip}
\item[{\rm (i)}] We have 
\[ 
s(G) = \left\{
\begin{array}{ll}
2 & \text{if $n$ is even} \\
3 & \text{if $n$ is odd}
\end{array}\right. 
\quad \text{and} \quad
u(G) = \left\{ 
\begin{array}{ll}
0 & \text{if $n=6$} \\
2 & \text{otherwise.}
\end{array}
\right.
\]
\item[{\rm (ii)}] For all $n$, 
\[
\gamma_u(G) \geqs \gamma_t(G) \geqs \lceil \log_2n \rceil \geqs 3.
\]
\item[{\rm (iii)}] Suppose $n$ is odd and $\O$ is the set of $\lfloor \frac{n}{2}\rfloor$-element subsets of $\{1, \ldots, n\}$. Then
\[
\gamma_u(G)  = b(G,\O) \leqs 2\log_2n.
\]
\item[{\rm (iv)}] If $n$ is even, then $\gamma_u(G) \leqs 3n\log_2n$. 
\end{itemize}
\end{theorem}

\begin{remk}\label{r:symmain}
Let us record some comments on the statement of Theorem \ref{t:symmain}:
\begin{itemize}\addtolength{\itemsep}{0.2\baselineskip}
\item[{\rm (a)}] The equality $u(S_n) = 2$ for $n \ne 6$ is the main feature of part (i). Indeed, the calculation of $s(S_n)$ is due to Binder \cite{Binder68,Binder70I}. In addition, Binder \cite{Binder70II} showed that $u(S_n) \geqs 1$ if $n \ne 6$ (as noted in \cite[Theorem 2]{BG}, we have $u(S_6)=0$).
\item[{\rm (b)}] For (iii), the upper bound (and its proof) is analogous to the upper bound on $\gamma_u(A_{2m})$ obtained in \cite{BH} (see Theorem~3(i) below). It is also worth noting that the base size $b(G,\O)$ is not known exactly (see \cite{Hal}).
\item[{\rm (c)}] It is easy to handle the small degree symmetric groups excluded in the theorem. By Theorem~\ref{t:solmain} we have $s(S_3) = 3$, $u(S_3)=2$, $\gamma_u(S_3) = 2$ and $P_2(S_3) = \frac{2}{3}$. Since $S_4$ has a proper non-cyclic quotient, we see that $s(S_4) = u(S_4) = 0$.    
\end{itemize}
\end{remk}

For alternating groups, our main result is the following. 

\begin{theorem}\label{t:altmain}
Let $G = A_n$ with $n \geqs 5$. 
\begin{itemize}\addtolength{\itemsep}{0.2\baselineskip}
\item[{\rm (i)}] If $n$ is even, then 
\[
s(G) = u(G) = \left\{\begin{array}{ll}
2 & \mbox{if $n=6$} \\
4 & \mbox{otherwise}
\end{array}\right.
\]
and 
\[
\log_2 n \leqs \gamma_t(G) \leqs \gamma_u(G) \leqs 2 \log_2 n.
\]
\item[{\rm (ii)}] If $n \geqs 9$ is odd, then $u(G) \geqs 4$ and
\[
\log_pn \leqs \gamma_t(G) \leqs \gamma_u(G) \leqs 77\log_2n,
\]
where $p$ is the smallest prime divisor of $n$. 
\item[{\rm (iii)}] We have $\gamma_u(G) = 2$ if and only if $n \geqs 13$ is a prime. 
\item[{\rm (iv)}] If $n>13$ is a prime, then $P_2(G)> 1 - n^{-1}$.  
\end{itemize}
\end{theorem}

\begin{remk}\label{r:altmain}
Some remarks on the statement of Theorem \ref{t:altmain}:
\begin{itemize}\addtolength{\itemsep}{0.2\baselineskip}
\item[{\rm (a)}] In part (i), the result on $s(G)$ and $u(G)$ is due to Brenner and Wiegold \cite[(3.01)--(3.05)]{BW75} and the upper bound on $\gamma_u(G)$ is  taken from \cite[Theorem 2]{BH}. The lower bound is established in Proposition \ref{p:sn_an_tdn_lower}.
\item[{\rm (b)}] Part (ii) extends the bound $u(G) \geqs 3$ established in \cite{BGK} (with essentially the same proof). The excluded cases $n \in \{5,7\}$ are genuine exceptions: it is straightforward to show that $s(A_5) = u(A_5)=2$ and $s(A_7)=u(A_7)=3$. 
The upper bound on $\gamma_u(G)$ is \cite[Proposition 3.15]{BH} and the lower bound follows from Proposition \ref{p:sn_an_tdn_lower}.
\item[{\rm (c)}] In stark contrast to the even degree case, the uniform spread of an odd-degree alternating group can be arbitrarily large. Indeed, this is a consequence of part (ii) of Theorem \ref{t:star}. More precisely, \cite[Proposition 3.1]{GSh} states that if $n$ is composite and $p$ is the smallest prime divisor of $n$, then
\[
s(A_n) \geqs cp\log p
\]
for some (undetermined) absolute constant $c$. We refer the reader to \cite[Propositions 3.2 and 3.3]{GSh} for explicit bounds on $u(A_n)$ and $s(A_n)$ when $n$ is a prime (also see \cite[Section 4]{BW75}).
\item[{\rm (d)}] Part (iii) extends \cite[Theorem 3.7(i)]{BH}, which states that $\gamma_u(A_n) = 2$ if $n \geqs 13$ is a prime. This can be viewed as a first step towards a classification of the finite simple groups $G$ with $\gamma_u(G)=2$ (see Corollary~\ref{cor:gamma2} below). In addition, if $n \geqs 13$ is a prime, then  $P(A_n,s,2)>0$ if and only if $s$ is an $n$-cycle (see Proposition \ref{p:an_udn_2}).
\item[{\rm (e)}] In part (iv), the case $n = 13$ is an anomaly. Indeed, with the aid of {\sc Magma} \cite{magma}, one can show that
\[
P_2(A_{13}) = \frac{4979}{46200}
\]
(see Remark~\ref{r:a13} for further details).
\end{itemize}
\end{remk}

Next we consider the finite simple groups of Lie type. Our main result for exceptional groups is the following.

\begin{theorem}\label{t:exmain}
Let $G$ be a finite simple exceptional group of Lie type over $\mathbb{F}_q$. 
\begin{itemize}\addtolength{\itemsep}{0.2\baselineskip}
\item[{\rm (i)}] We have $\gamma_u(G) \leqs 5$, with $\gamma_u(G) = 2$ if and only if 
\begin{equation}\label{e:exg}
G \in \{ {}^2B_2(q), {}^2G_2(q) \, (q \geqs 27),  {}^2F_4(q) \, (q \geqs 8),  {}^3D_4(q),  E_6^{\e}(q),  E_7(q),  E_8(q)\}.
\end{equation}
\item[{\rm (ii)}] If $\gamma_u(G)=2$, then $P_2(G) > 1-q^{-1}$.
\end{itemize}
\end{theorem}

\begin{remk}\label{r:exmain}
Let us make some remarks on the statement of Theorem \ref{t:exmain}:
\begin{itemize}\addtolength{\itemsep}{0.2\baselineskip}
\item[{\rm (a)}] The bound on $\gamma_u(G)$ in part (i) strengthens the result $\gamma_u(G) \leqs 6$ stated in \cite[Theorem 5.2]{BH}. The same result states that $\gamma_u(G) = 2$ if $G$ is ${}^2B_2(q)$, ${}^2G_2(q)$ (with $q \geqs 27$), or $E_8(q)$, and the complete classification of the exceptional groups $G$ with $\gamma_u(G)=2$ is the main feature of Theorem \ref{t:exmain}.   
\item[{\rm (b)}] The lower bound on $P_2(G)$ in part (ii) is essentially best possible. For example, if $G = {}^2B_2(q)$, then
\[
P_2(G) = 1 - \frac{(q^2-4)(q-\sqrt{2q}+1)+4}{q^2(q-1)(q+\sqrt{2q}+1)}
\]
(see Lemma \ref{l:suz} and Remark \ref{r:suz}). Stronger bounds are obtained in many cases. For instance, if $G = E_8(q)$ then $P_2(G) > 1-q^{-30}$ by Lemma \ref{l:e8}. 
\end{itemize}
\end{remk}

Finally, let us turn to the classical groups. The two-dimensional linear groups merit special attention and they are handled in the following result.

\begin{theorem}\label{t:psl2main}
Let $G = {\rm L}_{2}(q)$ with $q \geqs 4$.
\begin{itemize}\addtolength{\itemsep}{0.2\baselineskip}
\item[{\rm (i)}] If $q \geqs 11$ and $q \not\equiv 3 \imod{4}$, then 
\[
s(G) = u(G) = \left\{\begin{array}{ll}
q-1 & \mbox{if $q \equiv 1 \imod{4}$} \\
q-2 & \mbox{if $q$ is even.}
\end{array}\right.
\]
\item[{\rm (ii)}] If $q \equiv 3 \imod{4}$ and $q \geqs 11$, then $u(G) \geqs q-4$ and $s(G) \geqs q-3$. \item[{\rm (iii)}] If $q \equiv 3 \imod{4}$ and $q \geqs 11$ is a prime, then 
\[
s(G) \geqs \frac{1}{2}(3q-7) \quad \text{and} \quad s(G) - u(G) = \frac{1}{2}(q+1).
\]
\item[{\rm (iv)}] We have
\[
\gamma_u(G) = \left\{\begin{array}{ll}
4 & \mbox{if $q=9$} \\
3 & \mbox{if $q \in \{5,7\}$ or $q$ is even} \\
2 & \mbox{if $q \geqs 11$ is odd.}
\end{array}\right.
\]
\item[{\rm (v)}] If $q \geqs 11$ is odd, then 
\[
P_2(G) = \left\{\begin{array}{ll}
\frac{1}{2}\left(1+\frac{1}{q}\right) & \mbox{if $q \equiv 1 \imod{4}$} \\
\frac{1}{2}\left(1 - \frac{q+3}{q(q-1)}\right) & \mbox{if $q \equiv 3 \imod{4}$.}
\end{array}\right.
\]
In particular, $P_2(G) \geqs \frac{24}{55}$, with equality if and only if $q=11$.
\end{itemize}
\end{theorem}

\begin{remk}\label{r:psl2main}
Let us make some comments on the statement of Theorem \ref{t:psl2main}:
\begin{itemize}\addtolength{\itemsep}{0.2\baselineskip}
\item[{\rm (a)}] The spread of $G = {\rm L}_{2}(q)$ was first studied by Brenner and Wiegold in \cite{BW75}. According to \cite[Theorem 4.02]{BW75}, if $q \geqs 11$ or $q$ is even, then $s(G) = f(q)$ where
\[
f(q) = \left\{\begin{array}{ll}
q-1 & \mbox{if $q \equiv 1 \imod{4}$} \\
q-4 & \mbox{if $q \equiv 3 \imod{4}$} \\
q-2 & \mbox{if $q$ is even.} 
\end{array}\right.
\]
However, the proof in \cite{BW75} is incomplete and only the lower bound $s(G) \geqs f(q)$ is established. With some additional work, we can show that $s(G) = f(q)$ when $q \not\equiv 3 \imod{4}$ and $q \geqs 13$ (see Theorem \ref{t:l2qs}), which establishes part (i) of Theorem \ref{t:psl2main}. However, the case $q \equiv 3 \imod{4}$ is rather more complicated and we have been unable to compute $s(G)$ and $u(G)$ precisely. Note that part (ii) shows that the statement of \cite[Theorem 4.02]{BW75} is incorrect in this case. 
\item[{\rm (b)}] Part (iii) demonstrates that the difference $s(G)-u(G)$ can be arbitrarily large. As far as we are aware, this provides the first example of an infinite family of non-abelian finite groups with this property.
\item[{\rm (c)}] It is straightforward to handle the small values of $q$ excluded in parts (i), (ii) and (iii). Since ${\rm L}_{2}(4) \cong {\rm L}_{2}(5) \cong A_5$ and ${\rm L}_{2}(9) \cong A_6$, we see that $s(G) = u(G) = 2$ if $q \in \{4,5,9\}$. If $q=8$, then $s(G) = u(G) = 6$ as in part (i). Finally, for $q=7$ one can check that  $s(G)=5$ and $u(G)=3$. 
\item[{\rm (d)}] Part (iv) extends \cite[Proposition 6.4]{BH}, which states that $\gamma_u({\rm L}_{2}(q)) \leqs 4$, with equality if and only if $q=9$. We also note that $G={\rm L}_{2}(11)$ is the smallest simple group with $\gamma_u(G) = 2$.
\item[{\rm (e)}] It is worth highlighting the expression for $P_2(G)$ in part (v), which shows that $P_2(G) \to \frac{1}{2}$ as $q$ tends to infinity (cf. Corollary \ref{cor:prob}). 
\end{itemize}
\end{remk}

In order to state a result for all classical groups, it will be useful to write
\begin{align*}
\mathcal{A} & =\{ {\rm U}_{r+1}(q) \,:\, \text{$r \geqs 7$ odd} \} \cup \{ {\rm PSp}_{2r}(q) \,:\, \text{$r \geqs 3$ odd, $q$ odd} \} \cup \{ {\rm P\O}^+_{2r}(q) \,:\, \text{$r \geqs 5$ odd} \} \\
\mathcal{B} & = \{ {\rm Sp}_{2r}(q) \, :\, \text{$r \geqs 2$, $q$ even, $(r,q) \neq (2,2)$} \} \cup \{ \O_{2r+1}(q) \, :\, \text{$r \geqs 3$, $q$ odd} \} \\
\mathcal{C} & = \{ {\rm PSp}_{2r}(q) \, : \, \text{$r \geqs 5$ odd, $q$ odd} \} \cup \{ {\rm P\O}^{\pm}_{2r}(q) \, : \, \text{$r \geqs 4$ even} \}
\end{align*}

In the statement of the following result, $r$ denotes the (untwisted) Lie rank of $G$, which is the rank of the ambient simple algebraic group.

\begin{theorem}\label{t:classmain}
Let $G$ be a finite simple classical group over $\mathbb{F}_q$ of rank $r$.
\begin{itemize}\addtolength{\itemsep}{0.2\baselineskip}
\item[{\rm (i)}] We have $\gamma_u(G) \leqs 7r+56$. 
\item[{\rm (ii)}] We have $\gamma_u(G) =2$ only if one of the following holds:

\vspace{1mm}

\begin{itemize}\addtolength{\itemsep}{0.2\baselineskip}
\item[{\rm (a)}] $G = {\rm L}_{2}(q)$ and $q \geqs 11$ is odd.
\item[{\rm (b)}] $G = {\rm L}_{n}^{\e}(q)$, where $n$ is odd and $(n,q,\e) \not\in \{ (3,2,+), (3,4,+), (3,3,-), (3,5,-) \}$.
\item[{\rm (c)}] $G \in \mathcal{C}$.
\end{itemize}

\vspace{1mm}

\noindent Moreover, $\gamma_u(G)=2$ for the groups in {\rm (a)} and {\rm (b)}.

\item[{\rm (iii)}] If $\gamma_u(G) = 2$ and $G \not\in \mathcal{C} \cup \{{\rm L}_{2}(q)\,:\, \mbox{$q \geqs 11$ odd}\}$, then either $P_2(G) > \frac{1}{2}$, or $G = {\rm U}_{5}(2)$ and $P_2(G) = \frac{605}{1728}$. Moreover, $P_2(G) \to 1$ as $|G| \to \infty$. 
\end{itemize}
\end{theorem}

\begin{remk}\label{r:classmain}
Let us make some remarks on the statement of Theorem \ref{t:classmain}:
\begin{itemize}\addtolength{\itemsep}{0.2\baselineskip}
\item[{\rm (a)}] The bound on $\gamma_u(G)$ in part (i) is from \cite[Theorem 6.3]{BH}, where we also showed that $\gamma_u(G) \leqs 15$ if $G \in \mathcal{A}$ and $r \leqs \gamma_u(G) \leqs 7r$ if $G \in \mathcal{B}$.   
\item[{\rm (b)}] In part (ii), we have been unable to determine if $\gamma_u(G)=2$ for the groups $G$ in the collection $\mathcal{C}$. We refer the reader to Remarks \ref{r:case3} and \ref{r:case4} for further comments on the difficulties that arise in these special cases.
\end{itemize}
\end{remk}

We now present some general results concerning all finite simple groups. Let us write $\mathcal{D}$ for the classical groups arising in parts (ii)(a) and (ii)(b) of Theorem \ref{t:classmain}, and let $\mathcal{E}$ be the exceptional groups in \eqref{e:exg}. In addition, write
\begin{align*}
\mathcal{S} & = \{{\rm M}_{23}, {\rm J}_{1}, {\rm J}_{4}, {\rm Ru}, {\rm Ly}, {\rm O'N}, {\rm Fi}_{23}, {\rm Fi}_{24}', {\rm Th}, \mathbb{B}, \mathbb{M}\} \\
\mathcal{T} & = \{{\rm J}_{3}, {\rm He}, {\rm Co}_{1}, {\rm HN}\}
\end{align*}

The following result is an immediate corollary of Theorems \ref{t:altmain}(iii), \ref{t:exmain}(i) and \ref{t:classmain}(ii), together with \cite[Theorem 4.2]{BH} on sporadic groups. Note that $\gamma_u(G) \in \{2,3\}$ for each of the groups in $\mathcal{T}$, but we have been unable to determine the exact value in these cases (see Remark~\ref{r:spor_T}).

\begin{corol}\label{cor:gamma2}
Let $G$ be a finite simple group. Then $\gamma_u(G)=2$ only if \begin{itemize}\addtolength{\itemsep}{0.2\baselineskip}
\item[{\rm (i)}] $G \in \{A_n \,:\, \mbox{$n \geqs 13$ prime}\} \cup \mathcal{D} \cup \mathcal{E} \cup \mathcal{S}$; or
\item[{\rm (ii)}] $G \in \mathcal{C} \cup \mathcal{T}$.
\end{itemize}

\vspace{1mm}

\noindent Moreover, $\gamma_u(G)=2$ for the groups in {\rm (i)}.
\end{corol}

The next result follows from Theorems~\ref{t:altmain}(iv), \ref{t:exmain}(ii) and \ref{t:classmain}(iii).

\begin{corol}\label{cor:prob}
Let $(G_n)$ be a sequence of finite simple groups such that  $\gamma_u(G_n)=2$, 
\[
G_n \not\in \mathcal{C} \cup \{{\rm L}_{2}(q) \,:\, \mbox{$q \geqs 11$ odd}\}
\] 
 and $|G_n| \to \infty$. Then $P_2(G_n) \to 1$ as $n$ tends to infinity.
\end{corol}

We will also establish the following non-asymptotic result on the probability $P_2(G)$ (this will be proved in Section \ref{s:thm9}).

\renewcommand{\arraystretch}{1.3}
\begin{theorem}\label{t:prob2}
Let $G$ be a finite simple group such that $\gamma_u(G)=2$ and $G \not\in \mathcal{C} \cup \mathcal{T}$. Then $P_2(G) \leqs \frac{1}{2}$ if and only if one of the following holds:
\begin{itemize}\addtolength{\itemsep}{0.2\baselineskip}
\item[{\rm (i)}]  $G = {\rm L}_2(q)$, $q \geqs 11$, $q \equiv 3 \imod{4}$ and 
\[
\frac{24}{55} \leqs P_2(G)  = \frac{1}{2}\left(1 - \frac{q+3}{q(q-1)}\right).
\]
\item[{\rm (ii)}] $G \in \{A_{13},  {\rm U}_{5}(2), {\rm Fi}_{23}\}$ and
\[
P_2(G) = \left\{\begin{array}{ll}
\frac{4979}{46200}  & \mbox{if $G = A_{13}$} \\
\frac{605}{1728}    & \mbox{if $G = {\rm U}_{5}(2)$} \\
\frac{7700}{137241} & \mbox{if $G = {\rm Fi}_{23}$.} 
\end{array}\right.
\]
\end{itemize}
\end{theorem}
\renewcommand{\arraystretch}{1}

\begin{remk}
In terms of the asymptotics of $P_2(G)$, it is natural to ask whether or not this probability is bounded away from zero. That is, is there an absolute constant $\e>0$ such that $P_2(G) \geqs \e$ for every simple group $G$ with $\gamma_u(G)=2$? By Theorem \ref{t:prob2}, this is true for the relevant groups $G \not\in \mathcal{C}$. Therefore, in order to answer this question, it remains to consider the specific symplectic and orthogonal groups comprising the collection $\mathcal{C}$. 
\end{remk}

There is a natural generalisation of the uniform domination number for groups with uniform spread greater than one. Indeed, if $G$ is a finite group with $u(G) \geqs \ell$ for some positive integer $\ell$, then let $\gamma_u^{(\ell)}(G)$ be the smallest size of a set $S$ of conjugate elements such that for any nontrivial elements $x_1,\dots,x_\ell \in G$, there exists $y \in S$ such that $G = \<x_i,y\>$ for all $i$. Evidently, $\gamma_u^{(1)}(G) = \gamma_u(G)$. Since $u(G) \geqs 2$ for all finite non-abelian simple groups $G$, it is natural to consider $\gamma^{(2)}_u(G)$ for these groups. With this in mind, we can state the following result (see Section~\ref{s:thm9}).

\begin{theorem}\label{t:higher} 
Let $G$ be a finite simple group such that $\gamma_u(G)=2$ and $G \not\in \mathcal{C} \cup \mathcal{T}$. Then $\gamma_u^{(2)}(G) = 3$.
Moreover, the following hold:
\begin{itemize}\addtolength{\itemsep}{0.2\baselineskip}
\item[{\rm (i)}] If $G = A_n$ with $n > 13$ prime, then $\gamma_u^{(\ell)}(G) = \ell+1$ for all $1 \leqs \ell \leqs n$. 
\item[{\rm (ii)}] If $G$ is an exceptional group of Lie type over $\mathbb{F}_q$, then $\gamma_u^{(\ell)}(G) = \ell+1$ for all $1 \leqs \ell \leqs q$. 
\end{itemize}
\end{theorem}

\begin{remk}\label{r:higher}
Let us comment on Theorem~\ref{t:higher}.
\begin{itemize}
\item[(a)] Lemma~\ref{l:higher} states that if $G$ is a finite non-cyclic group, then $\gamma_u^{(\ell)}(G) \geqs \ell+1$, with equality if $P_2(G) > 1 - \ell^{-1}$ (this is clear for $\ell=1$). Therefore, the probabilistic results in Corollary~\ref{cor:prob} and Theorem~\ref{t:prob2} are crucial to the proof of Theorem~\ref{t:higher}.
\item[(b)] In certain cases, we can establish even stronger results on $\gamma_u^{(\ell)}(G)$. For example, if $G=E_8(q)$, then $\gamma_u^{(\ell)}(G) = \ell+1$ for all $\ell \leqs q^{30}$ (see Remark~\ref{r:exmain}(b)).
\end{itemize}
\end{remk}

\begin{remk}\label{r:saxl}
Suppose $G$ is a finite group with an element $s \in G$ that is contained in a unique maximal subgroup $H$ of $G$. In this situation, the class $s^G$ witnesses $\gamma_u(G)=2$ if and only if $b(G,G/H) = 2$, and one needs further information on the bases of size two in order to determine if the same class witnesses $\gamma_u^{(\ell)}(G)=\ell+1$ for some $\ell \geqs 2$. It turns out that this is encoded in the so-called \emph{Saxl graph} $\Sigma(G,G/H)$, which has vertex set $G/H$ and two vertices are adjacent if and only if they form a base for $G$ (see \cite{BGiu}). According to Lemma~\ref{l:higher2}, $\gamma_u^{(\ell)}(G)=\ell+1$ is witnessed by $s^G$ if and only if $\Sigma(G,G/H)$ has an $(\ell+1)$-clique (that is, a complete subgraph with $\ell+1$ vertices). In particular, the probabilistic condition in Remark \ref{r:higher}(a) is just a special case of Tur\'an's Theorem applied to the graph $\Sigma(G,G/H)$. In \cite{BGiu}, it is conjectured that if $G \leqs {\rm Sym}(\O)$ is a finite primitive group with $b(G,\O)=2$, then any pair of vertices of $\Sigma(G,\Omega)$ have a common neighbour and thus $\Sigma(G,\Omega)$ contains a triangle. In the above setting, this would imply that $\gamma_u^{(2)}(G)=3$. 
\end{remk}

Finally, let us return to the total domination number $\gamma_t(G)$ and recall that 
$\gamma_t(G) \leqs \gamma_u(G)$. It is natural to consider the relationship between these two numbers. For example, is there an absolute constant $C$ such that 
\[
\gamma_u(G) \leqs C\gamma_t(G)
\]
for all non-abelian finite simple groups $G$? By the results stated above, this is true for all even-degree alternating groups and all exceptional groups of Lie type. It also holds for the classical groups in the collections denoted $\mathcal{A}$, $\mathcal{B}$ and $\mathcal{D}$ above (indeed, one can modify the proof of \cite[Theorem 6.3(iii)]{BH} to get $\gamma_t(G) \geqs r$ for all $G \in \mathcal{B}$). It would be interesting to investigate if this relationship extends to the remaining classical groups, and also the alternating groups of odd degree.

\vs

\noindent \textbf{Notation.} Our group-theoretic notation is fairly standard. In particular, we adopt the notation from \cite{KL} for simple groups, so we write ${\rm L}_{n}(q) = {\rm PSL}_{n}(q)$ and ${\rm U}_{n}(q) = {\rm PSU}_{n}(q)$ for linear and unitary groups, and ${\rm P\O}_{n}^{\e}(q)$ is a simple orthogonal group, etc. In addition, if $G$ is a finite group, then we write $G^{\#}$ for the set of nontrivial elements in $G$ and $i_r(X)$ for the number of elements of order $r$ in a subset $X$ of $G$. For matrices, it will be convenient to write $[A_1^{n_1}, \ldots, A_k^{n_k}]$ for a block-diagonal matrix with a block $A_i$ occurring with multiplicity $n_i$. In addition, $J_i$ will denote a standard unipotent Jordan block of size $i$. Finally, for positive integers $a$ and $b$, we write $(a,b)$ for their greatest common divisor and we set $[a] = \{1, \ldots, a\}$.

\vs

\noindent \textbf{Acknowledgements.} 
The second author thanks the Engineering and Physical Sciences Research Council and the Heilbronn Institute for Mathematical Research for their financial support. Both authors thank Alexander Hulpke for his computational assistance and they thank an anonymous referee for their careful reading of the paper and several helpful comments and suggestions.

\section{Methods}\label{s:prelims}

In this section, we record some preliminary results which will be needed in the proofs of our main theorems. Throughout, let $G$ be a finite $2$-generated non-cyclic group and let $G^{\#}$ be the set of nontrivial elements of $G$.

\subsection{Spread}\label{ss:prelims_spread}

Define the spread $s(G)$ and uniform spread $u(G)$ as in the introduction. We begin by outlining a probabilistic approach which was first used by Guralnick and Kantor in \cite{GK} to prove that $u(G) \geqs 1$ for every non-abelian finite simple group $G$ and has since been instrumental in establishing several related results on the uniform spread of simple and almost simple groups in \cite{BGK,BG,GSh,Harper}. 

For $x,s \in G$, let
\begin{equation}\label{e:pxs}
P(x,s) = 1 - \frac{|\{z \in s^G \,:\, G = \la x, z \ra\}|}{|s^G|}
\end{equation}
be the probability that $x$ and a randomly chosen conjugate of $s$ do \emph{not} generate $G$. The following result is \cite[Lemma 2.1]{BG}.

\begin{lem}\label{l:bg1}
Let $k$ be a positive integer and assume there is an element $s \in G$ such that 
\[
\sum_{i=1}^{k}P(x_i,s) < 1
\]
for all $k$-tuples $(x_1, \ldots, x_k)$ of prime order elements in $G$. Then $u(G) \geqs k$ (with respect to the conjugacy class $s^G$).
\end{lem}

We can use fixed point ratios to estimate the sum in Lemma \ref{l:bg1}. Recall that if $G$ acts transitively on a finite set $\O$, then 
\[
\fpr(x,\O) = \frac{|C_{\O}(x)|}{|\O|} = \frac{|x^G \cap H|}{|x^G|}
\]
is the \emph{fixed point ratio} of $x \in G$, where $C_{\O}(x)$ is the set of fixed points of $x$ and $H$ is a point stabiliser. By \cite[Lemma 2.2]{BG}, we have
\begin{equation}\label{e:pbd}
P(x,s) \leqs \sum_{H \in \mathcal{M}(G,s)}\fpr(x,G/H),
\end{equation}
where $\mathcal{M}(G,s)$ is the set of maximal subgroups of $G$ which contain $s$. This yields the following corollary.

\begin{cor}\label{c:ug}
Let $k$ be a positive integer and assume there is an element $s \in G$ such that
\[
\sum_{H \in \mathcal{M}(G,s)}\fpr(x,G/H) < \frac{1}{k}
\]
for all $x \in G$ of prime order. Then $u(G) \geqs k$.
\end{cor}

Therefore, upper bounds on fixed point ratios can be used to bound the probability $P(x,s)$. In order to effectively apply this approach, one needs to identify an element $s \in G$ for which we can control the maximal overgroups in $\mathcal{M}(G,s)$ (ideally, $s$ should be contained in very few maximal subgroups).

When it is feasible to do so, this probabilistic approach will be complemented by computational methods implemented in \textsc{Magma} \cite{magma}. These methods are outlined in \cite[Section~2.3]{Harper} and we refer the reader to Breuer's manuscript \cite{Breuer} for a more detailed discussion.

\subsection{Uniform domination number}\label{ss:prelims_udn}

Let $G$ be a finite group with $u(G) \geqs 1$. Recall that a subset $S$ of $G^{\#}$ is a \emph{total dominating set} (TDS) for $G$ if for all $x \in G^{\#}$, there exists $y \in S$ such that $G = \la x,y\ra$, and the \emph{uniform domination number} of $\gamma_u(G)$ is the minimal size of a TDS for $G$ consisting of conjugate elements. This notion was first introduced in \cite{BH}, where close to best possible bounds on $\gamma_u(G)$ for simple groups $G$ were determined. Let us briefly recall the main methods developed in \cite{BH} to bound the uniform domination number.

Recall that if $G$ acts faithfully on a finite set $\O$, then a subset of $\O$ is a base for $G$ if its pointwise stabiliser in $G$ is trivial; the \emph{base size} of $G$, written $b(G,\O)$, is the minimal size of a base. The following result reveals an important connection between bases and total dominating sets consisting of conjugate elements.

\begin{lem}\label{l:udn}
Let $s \in G^{\#}$, let $H \in \mathcal{M}(G,s)$ and assume that $H$ is core-free. Then
\[
\min\{|S|\,:\, \mbox{{\rm $S \subseteq s^G$ is a TDS for $G$}}\} \geqs b(G,G/H),
\]
with equality if $\mathcal{M}(G,s) = \{H\}$. In particular, if for each $s \in G^{\#}$ there exists $H \in \mathcal{M}(G,s)$ with $b(G,G/H) \geqs c$, then $\gamma_u(G) \geqs c$.
\end{lem}

\begin{proof}
This follows from \cite[Corollaries 2.2 and 2.3]{BH}.
\end{proof} 

Probabilistic methods also play a key role in \cite{BH}. As in the introduction, for an element $s \in G$ and integer $c \geqs 1$, we define $P(G,s,c)$ to be the probability that $c$ randomly chosen conjugates of $s$ form a TDS for $G$ (see \eqref{e:pgsc}). Note that $\gamma_u(G) \leqs c$ if and only if $P(G,s,c)>0$ for some $s$. The next lemma provides a means of computing $P(G,s,2)$ in an important special case.

\begin{lem}\label{l:udn2}
Suppose $s \in G^{\#}$ and $\mathcal{M}(G,s) = \{H\}$ with $H$ core-free. Then 
\[
P(G,s,2) = \frac{r|H|^2}{|G|},
\]
where $r$ is the number of regular orbits of $H$ on $G/H$.
\end{lem}

\begin{proof}
First observe that if $\{x,y\}$ is a TDS for $G$, then so is $\{x^g,y^g\}$ for all $g \in G$, hence 
\[
P(G,s,2) = \frac{|\{s^g \in s^G \,:\, \mbox{$\{s,s^g\}$ is a TDS for $G$}\}|}{|s^G|}.
\]
Now $\{s,s^g\}$ is a TDS for $G$ if and only if $H \cap H^g = 1$, so 
\[
P(G,s,2) = \frac{|\{g \in G \,:\, H \cap H^g = 1\}|}{|C_G(s)||s^G|} = \frac{r|H|^2}{|G|}
\]
as required.
\end{proof}

The next result, which is \cite[Lemma 2.5]{BH}, shows that fixed point ratios can be used to bound the complementary probability $Q(G,s,c) = 1- P(G,s,c)$.

\begin{lem}\label{l:key}
Let $s \in G^{\#}$ and $c \in \mathbb{N}$. Then
\[
Q(G,s,c) \leqs \sum_{i=1}^{k} |x_i^G| \left(\sum_{H \in \mathcal{M}(G,s)}\fpr(x_i,G/H)\right)^c = : \what{Q}(G,s,c),
\]
where the $x_i$ represent the conjugacy classes in $G$ of elements of prime order.
\end{lem}

We can use the following result to estimate the bound that arises in Lemma \ref{l:key} (see \cite[Lemma 2.7]{BH}).

\begin{lem}\label{l:bd}
Let $\{H_1, \ldots, H_{\ell}\}$ be proper subgroups of $G$. Suppose that $x_1, \ldots, x_m$ represent distinct $G$-classes such that $\sum_{i}|x_i^G \cap H_j| \leqs A_j$ and $|x_i^G| \geqs B$ for all $i,j$. Then
\[
\sum_{i=1}^{m}|x_i^G|\left(\sum_{j=1}^{\ell}{\rm fpr}(x_i,G/H_j)\right)^c \leqs B^{1-c}\left(\sum_{j=1}^{\ell}A_j\right)^c
\]
for all positive integers $c$.
\end{lem}

Recall the definition of $\gamma_u^{(\ell)}(G)$ from the introduction, where $\ell$ is a positive integer. In particular, notice that $\gamma_u^{(1)}(G) = \gamma_u(G)$.

\begin{lem}\label{l:higher}
Let $\ell$ be a positive integer. Then $\gamma_u^{(\ell)}(G) \geqs \ell+1$, with equality if $Q(G,s,2) < 1/\ell$ for some $s \in G$. 
\end{lem}

\begin{proof}
Fix $s \in G$ and consider a subset $S = \{z_1,z_2,\dots,z_{\ell}\}$ of $s^G$ of size $\ell$. Since $G$ is not cyclic, there is no $j$ such that $G = \<z_i,z_j\>$ for all $i \in [\ell]$, whence $\gamma_u^{(\ell)}(G) \geqs \ell+1$.

Now assume that $Q(G,s,2) < 1/\ell$. We will use induction on $\ell$ to prove that there exists a subset $S$ of $s^G$ of size $\ell+1$ such that $\{z,z'\}$ is a TDS for $G$ for all distinct $z,z' \in S$. The base case $\ell = 1$ is clear, so let us assume $\ell \geqs 2$. Since $Q(G,s,2) < 1/\ell < 1/(\ell-1)$, by induction, there exist $z_1,\dots,z_{\ell} \in s^G$ such that $\{z_i,z_j\}$ is a TDS for all distinct $i,j \in [\ell]$. For $i \in [\ell]$, let $N_i \subseteq s^G$ be the set of conjugates $z$ of $s$ such that $\{ z_i,z \}$ is not a TDS for $G$. The bound $Q(G,s,2) < 1/\ell$ implies that $|N_i| < |s^G|/\ell$, so there exists $z_{\ell+1} \in s^G$ such that $z_{\ell+1} \not\in \bigcup_{i=1}^{\ell}N_i$ and hence $\{z_i,z_j\}$ is a TDS for all distinct $i,j \in [\ell+1]$. 

Now fix $S \subseteq s^G$ such that $|S|=\ell+1$ and $\{z,z'\}$ is a TDS for $G$ for all distinct $z,z' \in S$. Let $x_1,\dots,x_\ell \in G^{\#}$ be arbitrary elements. We claim that there exists $z \in S$ such that $G = \<x_i,z\>$ for all $i \in [\ell]$. Seeking  a contradiction, suppose that for all $z \in S$, there exists $i \in [\ell]$ such that $G \ne \<x_i, z \>$. Since $|S| > \ell$, there exists $i \in [\ell]$ and distinct $z,z' \in S$ such that $G \ne \<x_i,z\>$ and $G \ne \<x_i,z'\>$, but this contradicts the fact that 
$\{z,z'\}$ is a TDS. This completes the proof of the lemma.
\end{proof}

Recall that if $G$ is a group acting faithfully on a set $\Omega$, then the \emph{Saxl graph} $\Sigma(G,\Omega)$ has vertex set $\Omega$ and $\alpha, \beta \in \Omega$ are adjacent if and only if $\{\alpha,\beta\}$ is a base for the action of $G$ on $\Omega$ (see \cite{BGiu}).

\begin{lem}\label{l:higher2}
Fix $s \in G$ and assume that $\M(G,s)=\{ H \}$ for a core-free subgroup $H$ with $b(G,G/H)=2$. Then $\gamma_u^{(\ell)}(G) = \ell+1$ is witnessed by $s^G$ if and only if $\Sigma(G,G/H)$ has an $(\ell+1)$-clique. 
\end{lem}

\begin{proof}
First assume that $\Sigma(G,G/H)$ has an $(\ell+1)$-clique. Then there exist $\ell+1$ distinct conjugates $H^{g_1}, \ldots, H^{g_{\ell+1}}$ such that $H^{g_i} \cap H^{g_j} = 1$ for all $i \neq j$, which is equivalent to $\{ s^{g_i}, s^{g_j} \}$ being a TDS. Therefore, by the proof of Lemma~\ref{l:higher}, $\{s^{g_1}, \dots, s^{g_{\ell+1}}\}$ witnesses $\gamma_u^{(\ell)}(G) = \ell + 1$. 

Conversely, assume that $\gamma_u^{(\ell)}(G) = \ell + 1$ is witnessed by a set $S = \{ s^{g_1}, \dots, s^{g_{\ell+1}} \}$. To prove that $\Sigma(G,G/H)$ has an $(\ell+1)$-clique, it suffices to prove that $H^{g_i} \cap H^{g_j} = 1$ for all $i \neq j$. Suppose otherwise, say $1 \ne x \in H^{g_1} \cap H^{g_2}$. Then there is no element $z \in S$ such that $G=\<y,z\>$ for all $y \in \{ x, s^{g_3}, \dots, s^{g_{\ell+1}} \}$, which contradicts $S$ witnessing $\gamma_u^{(\ell)}(G) = \ell+1$. This completes the proof.
\end{proof}

In addition to the techniques described above, we will also use computational methods in \textsf{GAP} \cite{GAP} and \textsc{Magma} \cite{magma} to study $\gamma_u(G)$. We refer the reader to \cite[Section 2.3]{BH} and \cite{BH_comp} for a detailed discussion of these methods.

\section{Soluble groups}\label{s:sol}

In this section we will prove Theorem \ref{t:solmain}. With this goal in mind, we begin by recording an elementary lemma on the structure of the finite soluble groups we are interested in. Throughout this section, it will be convenient to let $\mathcal{S}$ be the set of finite non-abelian soluble groups with the property that every proper quotient is cyclic.

\begin{lem}\label{l:soluble}
Each $G \in \mathcal{S}$ is a primitive Frobenius group of the form $G = N{:}H$, where $N$, the socle of $G$, is an elementary abelian $p$-group for some prime $p$, and $H = \la h \ra$ acts faithfully and irreducibly on $N$. In particular, $H=C_G(h)$ has $|N|$ distinct $N$-conjugates and these subgroups intersect pairwise trivially.
\end{lem}

\begin{proof}
This is straightforward; see the proof of \cite[Proposition 1.1]{BGLMN}, for example.
\end{proof}

Since the spread of soluble groups has been studied by Brenner and Wiegold in \cite[Theorem~2.01]{BW75}, we may focus on the uniform spread and the uniform domination number. 

\begin{prop}\label{p:soluble}
Let $G = N{:}H \in \mathcal{S}$ as above. Let $1 \leqs k \leqs |N|-1$ and let $S$ be a subset of $h^G$ of size $k+1$. Then for any $k$ nontrivial elements $x_1,\dots,x_k$ in $G$, there exists $s \in S$ such that $G=\<x_i,s\>$ for all $i$. 
\end{prop}

\begin{proof}
Write $S = \{h^n \,:\, n \in M\}$ where $M$ is a subset of $N$ of size $k+1$. Let $x_1,\dots,x_k$ be arbitrary nontrivial elements of $G$. Since the distinct conjugates of $H$ have pairwise trivial intersection, each $x_i$ is contained in at most one conjugate of $H$. Therefore, there exists $n \in M$ such that $x_i \not\in H^n$ for all $i$. Fix $1 \leqs i \leqs k$ and write $x_i = n_ih_i$ with $n_i \in N$ and $h_i \in H^n$. Note that $n_i \neq 1$ since $x_i \not\in H^n$. We claim that $G = \<x_i,h^n\>$. To see this, first observe that $H^n \leqs \<x_i,h^n\>$ since $H^n=\<h^n\>$. In addition, $n_i \in \<x_i,h^n\>$ and $N = \la n_i^{H^n} \ra$ since $H^n$ acts irreducibly on $N$. Therefore $G = \<x_i,h^n\>$ and the result follows.
\end{proof}

\begin{cor}\label{c:soluble_us}
If $G  = N{:}H \in \mathcal{S}$ as above, then $u(G) = |N|-1$. 
\end{cor}

\begin{proof}
Since $u(G) \geqs |N|-1$ by Proposition~\ref{p:soluble}, it suffices to show that $u(G) \leqs |N|-1$. Seeking a contradiction, suppose that $G$ has uniform spread $|N|$ with respect to an element $s \in G$. Consider the set $\{h^n \,:\, n \in N \}$ of $|N|$ nontrivial elements. Since $G$ has uniform spread $|N|$ with respect to $s$, it follows that there exists $g \in G$ such that $G = \< s^g, h^n\>$ for all $n \in N$. Therefore, 
\[
s^g \in G \setminus \bigcup_{n \in N}H^n = N
\]
and thus $s \in N$. However, $\<n,s^x\> \leqs N < G$ for all $n \in N$ and  $x \in G$, which is a contradiction. Therefore, $u(G) \leqs |N|-1$ and the proof is complete.
\end{proof}

\begin{cor}\label{c:soluble_udn}
If $G  = N{:}H \in \mathcal{S}$ as above, then $\gamma_u(G) = 2$ and $P(G,h,2) = 1 - |N|^{-1}$.
\end{cor}

\begin{proof}
Both claims follow from the special case of Proposition~\ref{p:soluble} with $k=1$, noting that any two distinct $G$-conjugates of $h$ form a uniform dominating set.
\end{proof}

Finally, we determine the conjugacy classes that witness the properties established above. This completes the proof of Theorem \ref{t:solmain}.

\begin{prop}\label{p:soluble_witness}
If $G  = N{:}H \in \mathcal{S}$ as above, then the following are equivalent:
\begin{itemize}\addtolength{\itemsep}{0.2\baselineskip}
\item[{\rm (i)}]   $s^G$ witnesses $u(G) \geqs 1$.
\item[{\rm (ii)}]  $s^G$ witnesses $\gamma_u(G)=2$.
\item[{\rm (iii)}] $s$ generates a complement of $N$.
\end{itemize}
In particular, $P_2(G) = 1 - |N|^{-1}$.
\end{prop}

\begin{proof}
Proposition~\ref{p:soluble} proves that ${\rm (iii)}$ implies ${\rm (ii)}$, and evidently ${\rm (ii)}$ implies ${\rm (i)}$. We will prove that ${\rm (i)}$ implies ${\rm (iii)}$. Suppose that $u(G) \geqs 1$ with respect to $s^G$. Then for $1 \ne n \in N$, there exists $g \in G$ such that $G = \<n,s^g\>$. In particular, $G/N  = \<Nn,Ns^g\> = \<Ns^g\>$. Since $Ns^g$ generates $G/N$, the element $s^g$ generates a complement of $N$.
\end{proof}

\begin{rem}
It follows from the proofs of Proposition~\ref{p:soluble} and~\ref{p:soluble_witness} that two elements of $G$ form a total dominating set if and only if they generate  distinct complements of $N$.
\end{rem}

\section{Symmetric groups}\label{s:sym}

In this section, we will focus on symmetric groups and our main aim is to prove  Theorem~\ref{t:symmain}; alternating groups will be handled in Section~\ref{s:alt}. 

The spread of symmetric groups has been the subject of many papers, spanning several decades. In 1939, Piccard  \cite{Piccard} proved that if $n \geqs 3$, then $S_n$ has positive spread if and only if $n \neq 4$. In addition, she showed that the alternating group $A_n$ has positive spread for $n \geqs 4$. Building on these results, the spread and uniform spread of symmetric and alternating groups were studied by Binder in a series of papers \cite{Binder68,Binder70I,Binder70II,Binder73} in the late 1960s. In particular, he shows that 
\begin{equation}
s(S_n) = \left\{
\begin{array}{ll}
0 & \text{if $n=4$} \\
2 & \text{if $n \ne 4$ is even} \\
3 & \text{if $n$ is odd}
\end{array}\right. 
\label{e:s_sym}
\end{equation}
for $n \geqs 3$ (see \cite{Binder68,Binder70I}). Here Binder uses the term \emph{$k$-fold coherent} to describe a group $G$ with $s(G) \geqs k$, and he refers to a total dominating set as a \emph{complete set of complements} (the term \emph{spread} was first introduced by Brenner and Wiegold in \cite{BW75}). Binder studies the uniform spread of symmetric groups in \cite{Binder70II}, proving that $u(S_n) \geqs 1$ for $n \not\in \{4,6\}$. Note that the two excluded cases are genuine exceptions. Indeed, $s(S_4) = u(S_4)=0$ since $S_4$ has a non-cyclic proper quotient, and one can check that $u(S_6)=0$ (see \cite[Theorem~2]{BG}, for example). 

As a consequence of Binder's work, it follows that $u(S_n) \in \{1,2\}$ if $n \geqs 8$ is even, and $u(S_n) \in \{1,2,3\}$ if $n$ is odd. Our first aim is to determine the exact uniform spread of $S_n$ for all $n$; the main results are Theorems~\ref{p:sn_odd} and \ref{t:sn_even}. Along the way, we will also prove \eqref{e:s_sym}. With this goal in mind, the following preliminary result will be useful. 

\begin{lem}\label{l:small_elements}
Let $G = S_n$ with $n \geqs 9$, let $s = (1,2,\dots,n) \in G$ and let $i,j,k,l \in [n]$ be distinct numbers. 
\begin{itemize}\addtolength{\itemsep}{0.2\baselineskip}
\item[{\rm (i)}]   $G = \< (i,j), s \>$ if and only if $j-i$ and $n$ are coprime.
\item[{\rm (ii)}]  $\< (i,j,k), s \> \geqs A_n$ if and only if $j-i$ and $k-j$ are coprime.
\item[{\rm (iii)}] $\< (i,j)(k,l), s \> \geqs A_n$ if $j-i$ and $n$ are coprime and $i,j,k,l \leqs n/2$.
\end{itemize}
\end{lem}

\begin{proof}
Let $x$ be $(i,j)$, $(i,j,k)$ and $(i,j)(k,l)$ in parts (i), (ii) and (iii), respectively. In all three cases, it suffices to show that $\<x,s\>$ acts primitively on $[n]$ since any primitive permutation group of degree $n \geqs 9$ and minimal degree at most $4$ must contain $A_n$ (see \cite[Example~3.3.1]{DM_book}). Cases (i) and (ii) are straightforward, so let us consider (iii). Let $\Pi$ be a nontrivial partition of $[n]$ stabilised by $s$. Then each part of $\Pi$ is a congruence class modulo $a$, where $1<a<n$ and $a$ divides $n$. If $x$ also stabilises $\Pi$, then either each cycle of $x$ is contained in a part of $\Pi$, or $a=n/2$ and either $\{i,k\}$ and $\{j,l\}$ are parts of $\Pi$, or $\{i,l\}$ and $\{j,k\}$ are parts of $\Pi$. The first of these possibilities is excluded by the condition that $j-i$ and $n$ are coprime, and the second is ruled out by the condition $i,j,k,l \leqs n/2$ since each part of $\Pi$ contains a number greater than $n/2$. The result follows. 
\end{proof}

\begin{rem}\label{r:small_elements}
By Lemma~\ref{l:small_elements}(i), if $z \in S_n$ is an $n$-cycle then 
$S_n = \<z,(i,iz)\>$
for each $i \in [n]$. 
\end{rem}

Finally, we will refer to the \emph{shape} of a permutation $x \in S_n$ to mean the list of lengths of the disjoint cycles comprising $x$. For example, $(1,2)(3,4) \in S_7$ has shape $[2^2,1^3]$. 

\subsection{Uniform spread for odd degrees}

We will first determine the uniform spread of odd-degree symmetric groups, which is significantly easier than the even-degree case. We will also compute the spread of these groups.

\begin{thm}\label{p:sn_odd}
Let $G = S_n$ with $n \geqs 5$ odd. Then $u(G)=2$ and $s(G)=3$.
\end{thm}

\begin{proof}
The case $n=5$ can be handled using {\sc Magma}, so let us assume $n \geqs 7$. We begin by establishing lower bounds. Write $n=2m+1$ and let $C$ be the class of elements of shape $[m,m+1]$. Let $x_1$, $x_2$ and $x_3$ be nontrivial elements of $G$ and assume that 
\[
\{x_1,x_2,x_3\} \neq \{ (1,2), (2,3), (1,3) \}^g
\] 
for all $g \in G$. Therefore, we may fix two disjoint subsets $\{a_1,a_2,a_3\}$, 
$\{b_1,b_2,b_3\}$ of $\{1,2,\dots,n\}$ such that $a_ix_i=b_i$ for each $i$. Choose $z \in C$ so that $a_1,a_2,a_3$ are contained in the $m$-cycle of $z$, and $b_1,b_2,b_3$ are contained in the $(m+1)$-cycle. Clearly, $z$ is contained in a unique maximal intransitive subgroup $H$. Since $(m,m+1)=1$, $z$ is not contained in a transitive imprimitive subgroup. In addition, since $z$ is odd and $z^{m+1}$ is an $m$-cycle, a theorem of Marggraf (see \cite[Theorem 13.5]{W}) implies that $z$ is not contained in a primitive maximal subgroup. Therefore, $\mathcal{M}(G,z) = \{ H \}$ with $H = S_m \times S_{m+1}$. However, by construction $\<x_i,z\>$ is transitive for each $i$, so $G = \<x_1,z\> = \<x_2,z\> = \<x_3,z\>$.

In particular, by setting $x_3=x_2$ in the above paragraph, we have shown that $u(G) \geqs 2$ with respect to the class $C$. Moreover, to show that $s(G) \geqs 3$, it suffices to prove that there exists an element $w \in G$ such that $G = \<(1,2),w\> = \<(2,3),w\> = \<(1,3),w\>$. By Lemma~\ref{l:small_elements}(i), we can take $w=(1,2,\dots,n)$, hence $s(G) \geqs 3$. 

Let us now turn to upper bounds. Let $s \in G$ be such that 
\[
G = \< (1,2), s \> = \< (1,3), s \> = \< (2,3), s \>.
\] 
Suppose $s$ stabilises a $k$-element subset $A$ of $[n]$ with $k<n$, and let $B$ be the complement of $A$. Since  $G = \< (1,2), s \>$, we may assume, without loss of generality, $1 \in A$ and $2 \in B$. Since  $G = \< (1,3), s \>$ and $1 \in A$ we must have $3 \in B$. However, this gives $\< (2,3), s \> \leqs G_{\{B\}} < G$, which is a contradiction. Therefore, $s$ is an $n$-cycle. 

Let us draw two conclusions from this observation. First, suppose that $u(G) \geqs 3$ is witnessed by a class $s^G$. We have just demonstrated that $s$ is an $n$-cycle. However, $n$ is odd, so $s \in A_n$ and thus $\<(1,2,3), s^g \> \leqs A_n < G$ for all $g \in G$, which is a contradiction. Therefore, $u(G) \leqs 2$ and thus $u(G)=2$. Second, suppose that $s(G) \geqs 4$. Then there exists $t \in G$ such that $G = \< (1,2), t\> = \< (2,3), t \> = \< (1,3), t \> = \< (1,2,3), t \>$. As before, $t$ must be an $n$-cycle and once again we reach a contradiction since $\< (1,2,3), t \> \leqs A_n < G$. Therefore, $s(G) \leqs 3$ and the proof is complete.
\end{proof}

\subsection{Uniform spread for even degrees}

Now let us turn to the uniform spread of symmetric groups of even degree. Our main aim is to prove the following result.

\begin{thm}\label{t:sn_even_us}
Let $G = S_n$ with $n \geqs 8$ even. Then $u(G) \geqs 2$. 
\end{thm}

Our proof of Theorem~\ref{t:sn_even_us} relies on the probabilistic method described in Section~\ref{ss:prelims_spread} and it is based on the arguments in \cite[Section~7]{GK} and \cite[Section~6]{BGK}. In addition, we use some of the ideas in Binder's proof of the weaker bound $u(G) \geqs 1$ in \cite{Binder70II}.

Our first priority is to establish some fixed point ratio estimates for the action of $G= S_n$ on partitions. For a divisor $l$ of $n$ with $1 < l <n$, let $\Pi_l$ be the set of partitions of $[n]$ into $l$ parts of equal size. Then $G$ has a natural action on $\Pi_l$ and for each $x \in G$ we write $\Fix_l(x)$ for the set of partitions in $\Pi_l$ stabilised by $x$. Whenever we refer to partitions we mean partitions whose parts have equal size.

We will establish the following bounds, which may be of independent interest (note that in both lemmas, there are no conditions on the parity of $n$).

\begin{lem}\label{l:fprs_odd}
Let $G = S_n$ with $n \geqs 8$ and assume that $x \in G$ has shape $[p^k,1^{n-pk}]$ where $p$ is an odd prime and $k \geqs 1$. Then
\[
\fpr(x,\Pi_l) <
\left\{
\begin{array}{ll}
1/l^2 & \text{if $(p,k)=(3,1)$ or $l > n/6$} \\
1/l^3 & \text{otherwise}
\end{array}
\right.
\]
\end{lem}

\begin{lem}\label{l:fprs_even}
Let $G = S_n$ with $n \geqs 14$ and assume that $x \in G$ has shape $[2^k,1^{n-2k}]$ where $k \geqs 1$. Then
\[
\fpr(x,\Pi_l) <
\left\{
\begin{array}{ll}
1/l    & \text{if $k=1$} \\
6/n^2  & \text{if $k=2$ and $l=n/2$} \\ 
1/l^2  & \text{if $k=2$ and $2 < l < n/2$} \\
33/128 & \text{if $k=2$ and $l=2$} \\
1/l^2  & \text{if $k \geqs 3$ and $l > n/6$} \\
1/l^3  & \text{if $k \geqs 3$ and $l \leqs n/6$}
\end{array}
\right.
\]
\end{lem}

\begin{rem} Let us comment on Lemmas~\ref{l:fprs_odd} and~\ref{l:fprs_even}.
\begin{itemize}
\item[(a)] In \cite[Proposition~6.7]{BGK}, similar upper bounds on $\fpr(x,\Pi_l)$ are obtained when $n$ is odd and $x \in A_n$ has prime order. Therefore, we can view Lemmas~\ref{l:fprs_odd} and~\ref{l:fprs_even} as an extension of this result to all $n$ and all $x \in S_n$ of prime order. 
\item[(b)] When $n$ is small, it is straightforward to verify the bounds in Lemmas~\ref{l:fprs_odd} and~\ref{l:fprs_even} with the aid of {\sc Magma}. Indeed, we will prove these results in this way for $n < 30$. 
\end{itemize}
\end{rem}

We begin by recording some preliminary results (once again, note that there are no conditions on the parity of $n$ in Lemmas \ref{l:injection_odd} and \ref{l:injection_even}).

\begin{lem}\label{l:injection_odd}
Let $G = S_n$ and assume that $x \in G$ has shape $[p^k,1^{n-pk}]$, where $p$ is an odd prime and $k \geqs 1$.
\begin{itemize}\addtolength{\itemsep}{0.2\baselineskip}
\item[{\rm (i)}]   If $l \leqs \frac{n}{3}$, then
\[
\fpr(x,\Pi_l) \leqs \fpr((1,2,3),\Pi_l).
\]
\item[{\rm (ii)}]  If $p \geqs 5$ and $l \leqs \frac{n}{5}$, then
\[
\fpr(x,\Pi_l) \leqs \fpr((1,2,3,4,5),\Pi_l).
\]
\item[{\rm (iii)}] If $p=3$, $k \geqs 2$ and $l \leqs \frac{n}{6}$, then
\[
\fpr(x,\Pi_l) \leqs \fpr((1,2,3)(4,5,6),\Pi_l).
\]
\end{itemize}
\end{lem}

\begin{proof}
The bound in part (i) is obtained by constructing an injection 
\[
f\colon \Fix_l(x) \to \Fix_l((1,2,3)),
\]
as in the proof of \cite[Proposition~6.7(c1)]{BGK}. Similarly, we refer the reader to \cite[Proposition~7.4(ii)]{GK} for parts (ii) and (iii), which are established in a similar fashion. 
\end{proof}

\begin{lem}\label{l:injection_even}
Let $G = S_n$ with $n \geqs 30$ and let $x=(1,2)(3,4)\cdots(2k-1,2k) \in G$ with $k \geqs 2$.
\begin{itemize}\addtolength{\itemsep}{0.2\baselineskip}
\item[{\rm (i)}]  If $3 \leqs l \leqs \frac{n}{3}$, then
\[
\fpr(x,\Pi_l) \leqs \fpr((1,2)(3,4),\Pi_l).
\]
\item[{\rm (ii)}] If $k \geqs 3$ and $3 \leqs l \leqs \frac{n}{6}$, then
\[
\fpr(x,\Pi_l) \leqs \fpr((1,2)(3,4)(5,6),\Pi_l).
\]
\end{itemize} 
\end{lem}
 
In the proof of Lemma~\ref{l:injection_even} we will make use of the convenient notation for partitions introduced by Guralnick and Kantor in \cite[Section~7]{GK}.

\begin{nota}\label{n:partitions}
When defining a map $f$ between subsets of $\Pi_l$, for a partition $\Pi$ we will specify the image $f(\Pi)$ by giving some of the elements of $\{1,2,\dots,n\}$, separating parts by $/$ and assuming that the unspecified points, denoted by $*$, are in the same parts of $f(\Pi)$ as they are in $\Pi$. For $x = (1,2)(3,4)\cdots(2k-1,2k) \in S_n$ and $i \in [n]$  in the support of $x$, we will write $i' = i-1$ if $i$ is even and $i' = i+1$ if $i$ is odd (so $(i,i')$ is the cycle of $x$ containing $i$).
\end{nota}

\begin{proof}[Proof of Lemma~\ref{l:injection_even}]
To establish part (i), we proceed as in the proof of \cite[Proposition~6.7(c2)]{BGK} (note that in \cite{BGK}, it is assumed that $k$ is even and $n$ is odd). The claim is clear if $k=2$, so let us assume $k \geqs 3$. Define a map
\[
f\colon \Fix_l(x) \to \Fix_l((1,2)(3,4))
\]
as follows. Let $\Pi \in \Fix_l(x)$. If $\Pi \in \Fix_l((1,2)(3,4))$, then set $f(\Pi) = \Pi$. Otherwise, by considering the possible types of partitions in $\Fix_l(x) \setminus \Fix_l((1,2)(3,4))$, we define $f(\Pi)$ as in Table \ref{tab:f0}. Note that $f(\Pi) \in \Fix_l((1,2)(3,4)) \setminus \Fix_l(x)$ in each case.

\begin{table}
\renewcommand{\arraystretch}{1.2}
\begin{center}
\[
\begin{array}{cl} \hline
{\rm (I)} & \hspace{5mm} \begin{array}{l}
(1,a,b,* \ / \ 2,a',b',* \ / \ 3,* \ / \ 4,* \ / \dots)  \\
\quad \mapsto (1,2,a,* \ / \ 3,4,b,* \ / \ a',* \ / \ b',* \ / \dots) 
\end{array} \\
{\rm (II)} & \hspace{5mm} \begin{array}{l}
(1,2,a,* \ / \ 3,b,* \ / \ 4,b',* \ / \dots) \\
\quad \mapsto (1,2,b,* \ / \ 3,4,* \ / \ a,b',* \ / \dots) 
\end{array} \\
{\rm (III)} & 
\left\{
\begin{array}{ll}
\begin{array}{l} (3,4,a,* \ / \ 1,b,* \ / \ 2,b',* \ / \dots) \\ \quad \mapsto (3,4,b,* \ / \ 1,2,* \ / \ a,b',* \ / \dots) \end{array} & \text{if $\frac{n}{l} \geqs 4$} \\
\begin{array}{l} (3,4,a \ / \ 1,b,c \ / \ 2,b',c' \ / \ d,* \ / \dots) \\ \quad \mapsto (3,4,c' \ / \ 1,2,d \ / \ a,b,c \ / \ b',* \ / \ \dots) \end{array} & \text{if $\frac{n}{l} = 3$} \\
\end{array}\right.  \\
{\rm (IV)} & \hspace{5mm} \begin{array}{l}
(1,3,* \ / \ 2,4,* \ / \dots)  \\
\quad \mapsto (1,2,* \ / \ 3,4,* \ / \dots) 
\end{array} \\
{\rm (V)} & \hspace{5mm} \begin{array}{l}
(1,4,a,* \ / \ 2,3,a',* \ / \ b,c,d,* \ / \dots ) \text{ where $b' \not\in \{c,d\}$ } \\
\quad \mapsto (1,2,a,* \ / \ 3,4,b,* \ / \ a',c,d,* \ / \dots) 
\end{array} \\ \hline
\end{array}
\]
\end{center}
\caption{The map $f$ in the proof of Lemma \ref{l:injection_even}(i)}
\label{tab:f0}
\renewcommand{\arraystretch}{1}
\end{table}

In order to establish the desired bound in part (i), it suffices to show that $f$ is injective. To do this, let $\Sigma$ be a partition in the image of $f$ and write $f(\Pi)=\Sigma$ for some $\Pi \in \Fix_l(x)$. We need to show that there is a unique choice for $\Pi$. Clearly, if $\Sigma \in \Fix_l(x)$ then $\Pi = \Sigma$, so let us assume $\Sigma \not\in \Fix_l(x)$.  Note that in every case, $\{1,2\}$ is a subset of a part $A$ of $\Sigma$, and $\{3,4\}$ is a subset of some other part $B$. 

From the definition of $f$, we see that there are three separate cases to consider, according to the number $m \in \{2,3,4\}$ of parts of $\Sigma$ that are not mapped to parts of $\Sigma$ by $x$. We will refer to the \emph{type} of $\Pi$ by the label recorded in the first column of Table \ref{tab:f0}.

First assume $m=4$, so $\Pi$ has type (I) or (III) (with $n/l = 3$ in the latter case). If no point in any of these four parts is fixed by $x$, then $\Pi$ has type (I). To prove that $\Pi$ is uniquely determined in this case, it suffices to show that $a$ and $b$ are uniquely determined (indeed, these numbers determine $a'$ and $b'$ and, together with $\Sigma$, these four values determine $\Pi$). Now $a$ is the unique point in $A \setminus \{1,2\}$ which is not mapped by $x$ into $B$, and similarly $b$ is the unique point in $B \setminus \{3,4\}$ not mapped into $A$. Therefore, $\Pi$ is uniquely determined by $\Sigma$. Now assume $\Pi$ has type (III) and $n/l = 3$. Here, $d$ is the unique point in $A \setminus \{1,2\}$ and $c'$ is the unique point in $B \setminus \{3,4\}$, so we have determined $c$ and $d$. Now $a$ is the unique point fixed by $x$ in the part of $\Sigma$ containing $c$, and $b$ is the remaining point in this part. The values of $a$, $b$, $c$ and $d$, together with the partition $\Sigma$, uniquely determine $\Pi$.

Now assume $m=3$, so $\Pi$ is of type (II), (III) (with $n/l \geqs 4$) or (V). Necessarily, $A$ and $B$ are two of the three parts of $\Sigma$ that are not mapped to parts of $\Sigma$ by $x$. Let $C$ be the third such part. We begin by demonstrating that the type of $\Pi$ is determined by $\Sigma$. First assume that $n/l=3$. If a point of $B$ is mapped by $x$ into $C$, then $\Pi$ has type (II), otherwise, $\Pi$ has type (V). Now assume that $n/l \geqs 4$. If a point of $A$ is mapped by $x$ into $B$, then $\Pi$ has type (V). Otherwise, if at least two points of $A$ are mapped into $C$ by $x$, then $\Pi$ has type (III); else $\Pi$ has type (II).

We now show that $\Pi$ is determined by $\Sigma$ when $m=3$. First assume that $\Pi$ has type (II). The unique point of $A$ not mapped into $A$ by $x$ is $b$, and the unique point in $C\setminus\{b'\}$ not mapped into $B$ is $a$; this determines $\Pi$. Next assume that $n/l \geqs 4$ and $\Pi$ has type (III). The unique point in $B$ not mapped into $B$ is $b$, and the unique point in $C \setminus \{b'\}$ not mapped into $A$ is $a$; this determines $\Pi$. Finally, suppose $\Pi$ has type (V). Now $a$ is the unique point in $A$ not mapped by $x$ into $B$, and $b$ is the unique point in $B$ not mapped into $A$; this determines $\Pi$ (since $c$ and $d$ are not moved by $f$).

Finally, suppose $m=2$. Here $\Pi$ has type (IV) and thus $\Pi$ is the partition obtained by interchanging $2$ and $3$ in $\Sigma$. Therefore, $\Pi$ is uniquely determined in this case.

We have now shown that the map $f$ is injective, which completes the proof of part (i). 

Now consider (ii) and note that we may assume $k \geqs 4$. We define a map 
\[
f\colon \Fix_l(x) \to \Fix_l((1,2)(3,4)(5,6))
\]
as follows. Let $\Pi \in \Fix_l(x)$. If $\Pi \in \Fix_l((1,2)(3,4)(5,6))$, then we define $f(\Pi) = \Pi$. Now assume that $\Pi \not\in \Fix_l((1,2)(3,4)(5,6))$. For brevity, we will handle multiple possibilities at once by letting
\begin{equation}\label{e:sig}
\sigma \in \{ 1, (1,3,5)(2,4,6), (1,5,3)(2,6,4) \}
\end{equation}
and writing $\mathbf{i}=i\sigma$ for $1 \leqs i \leqs 6$. With this notation, we define $f$ as in Table \ref{tab:f1}. Note that $f(\Pi) \in \Fix_l((1,2)(3,4)(5,6))$. 

\begin{table}
\renewcommand{\arraystretch}{1.2}
\begin{center}
\[
\begin{array}{cl} \hline
{\rm (I)} & \begin{array}{l}
(1,a,b,* \ / \ 2,a',b',* \ / \ 3,c,* \ / \ 4,c',* \ / 5,* \ / 6,* \ / \dots)  \\
\quad \mapsto (1,2,a,* \ / \ 3,4,b,* \ / \ 5,6,* \ / \ c,a',* \ / \ b',* \ / \ c',* \ / \dots) 
\end{array} \\
{\rm (II)} & \begin{array}{l}
(\on,\tw,* \ / \ \tr,a,b,* \ / \ \fo,a',b',* \ / \fv,* \ / \ \si,* \ / \dots)  \\
\quad \mapsto (\on,\tw,* \ / \ \tr,\fo,a,* \ / \fv,\si,b,* \ / \ a',* \ / \ b',* \ / \dots)
\end{array} \\ 
{\rm (III)} & \begin{array}{l}
(\on,\tw,* \ / \ \tr,\fo,a,* \ / \ \fv,b,* \ / \ \si,b',* \ / \dots)  \\
\quad \mapsto (\on,\tw,* \ / \ \tr,\fo,b,* \ / \ \fv,\si,* \ / \ a,b',* \ / \dots)
\end{array} \\ 
{\rm (IV)} & \begin{array}{l}
(\on,\tw,\tr,\fo,a,* \ / \ \fv,b,* \ / \ \si,b',* \ / \dots)  \\
\quad \mapsto (\on,\tw,\tr,\fo,b,* \ / \ \fv,\si,* \ / \ a,b',* \ / \dots)
\end{array} \\ 
{\rm (V)} & \begin{array}{l}
(\on,\tw,* \ / \ \tr,\fv,* \ / \ \fo,\si,* \ / \dots)  \\
\quad \mapsto (\on,\tw,* \ / \ \tr,\fo,* \ / \ \fv,\si,* \ / \dots)
\end{array} \\ 
{\rm (VI)} & \begin{array}{l}
(\on,\tw,a,* \ / \ \tr,\si,b,* \ / \ \fo,\fv,* \ / \dots)  \\
\quad \mapsto (\on,\tw,b,* \ / \ \tr,\fo,a,* \ / \ \fv,\si,* \ / \dots)
\end{array} \\ 
{\rm (VII)} & \begin{array}{l}
(\on,\tr,a,* \ / \ \tw,\fo,a',* \ / \ \fv,b,* \ / \ \si,b',* \ / \dots) \\
\quad \mapsto (\on,\tw,a,* \ / \ \tr,\fo,b,* \ / \ \fv,\si,* \ / \ a',b',* \ / \dots)
\end{array} \\ 
{\rm (VIII)} & \begin{array}{l}
(\on,\fo,a,* \ / \ \tw,\tr,a',* \ / \ \fv,b,c,* \ / \ \si,b',c',* \ / \dots) \\
\quad \mapsto \left\{ 
\begin{array}{ll} 
(\on,\tw,a,* \ / \ \tr,\fo,a',* \ / \ \fv,\si,c,* \ / \ b,b',c',* \ / \dots) & \text{if $b$ is even} \\ 
(\on,\tw,a,* \ / \ \tr,\fo,c,* \ / \ \fv,\si,a',* \ / \ b,b',c',* \ / \dots) & \text{if $b$ is odd} \\
\end{array} \right.
\end{array} \\ 
{\rm (IX)} & \begin{array}{l}
(1,3,5,* \ / \ 2,4,6,c,* \ / \ a,b,d,* \ / \dots) \\
\quad \mapsto (1,2,a,* \ / \ 3,4,b,d,* \ / \ 5,6,c,* \ / \dots)
\end{array} \\ 
{\rm (X)} & \begin{array}{l}
(1,3,6,* \ / \ 2,4,5,c,* \ / \ a,b,d,* \ / \dots) \\
\quad \mapsto (3,4,a,* \ / \ 5,6,b,d,* \ / \ 1,2,c,* \ / \dots)
\end{array} \\ 
{\rm (XI)} & \begin{array}{l}
(1,4,5,* \ / \ 2,3,6,c,* \ / \ a,b,d,* \ / \dots) \\
\quad \mapsto (5,6,a,* \ / \ 1,2,b,d,* \ / \ 3,4,c,* \ / \dots)
\end{array} \\ 
{\rm (XII)} & \begin{array}{l}
(1,4,6,a,b,* \ / \ 2,3,5,a',b',* \ / \ c,d,e,g,* \ / \dots) \\
\quad \mapsto (1,2,a,c,d,* \ / \ 3,4,b,e,g,* \ / \ 5,6,a',b',* \ / \dots)
\end{array} \\ \hline
\end{array}
\]
\end{center}
\caption{The map $f$ in the proof of Lemma \ref{l:injection_even}(ii). (In cases (IX)--(XI), $\{a,b,d\} \cap \{a',b',d'\} = \emptyset$. In case (XII), if $c'$ or $d'$ is in  $\{c,d,e,g,*\}$ then $c'=d$, and if $e'$ or $g'$ is in  $\{c,d,e,g,*\}$ then $e'=g$.)}
\label{tab:f1}
\renewcommand{\arraystretch}{1}
\end{table}

We claim that $f$ is injective. To see this, let $\Sigma = f(\Pi)$ be a partition in the image of $f$. As before, we need to show that $\Pi$ is uniquely determined by $\Sigma$. We may assume that $\Sigma \not\in {\rm Fix}_{l}(x)$.

First assume that $\{1,2\}$, $\{3,4\}$ and $\{5,6\}$ are not subsets of distinct parts of $\Sigma$, so $\Pi$ has type (IV). Here $b$ is the unique point in the part of 
$\Sigma$  containing $\on$, $\tw$, $\tr$ and $\fo$ which is not mapped by $x$ into that part. In addition, there is a unique part all of whose points other than exactly two are mapped into the part containing $\fv$ and $\si$; $a$ is the unique point other than $b'$ in this part not mapped into the part containing $\fv$ and $\si$. This determines $\Pi$.

For the remainder, we may assume that  $\{1,2\}$, $\{3,4\}$ and $\{5,6\}$ are subsets of distinct parts $A$, $B$ and $C$ of $\Sigma$. Where appropriate, we will write $\bA = A\sigma$, $\bB = B\sigma$ and $\bC = C\sigma$ for $\s$ as in \eqref{e:sig}. There are four cases to consider, according to the number $m \in \{2,3,4,6\}$ of parts of $\Sigma$ that are not mapped to parts of $\Sigma$ by $x$.

First assume $m=6$, so $\Pi$ has type (I). The unique point in $A\setminus \{1,2\}$ not mapped into $B$ by $x$ is $a$, and the unique point in $B \setminus \{3,4\}$ not mapped into $A$ is $b$. There is a unique part of $\Sigma$ all of whose points other than exactly two are mapped into $C$, and $c$ is the unique such point other than $a'$. This determines $\Pi$.

Next assume $m=4$, so $\Pi$ has type (II), (VII) or (VIII). If one of $A$, $B$ or $C$ is fixed by $x$, then $\Pi$ has type (II). In this case, the unique point in $\bB \setminus \{\tr,\fo\}$ not mapped into $\bC$ is $a$, and the unique point in $\bC \setminus \{\fv,\si\}$ not mapped into $\bB$ is $b$; this determines $\Pi$. Now assume that none of the parts $A$, $B$ and $C$ are fixed by $x$. Let $P$ be the unique part of $\Sigma$ other than $\bA$, $\bB$ or $\bC$ which is not mapped to a part by $x$. There are two cases to consider. First suppose $P$ contains a cycle $Z$ of $x$, in which case $\Pi$ has type (VIII) and $Z = \{b,b'\}$. If $\bA \setminus \{\on, \tw\}$ and $\bB \setminus \{\tr,\fo\}$ are interchanged by $x$, then $b$ is even; otherwise, $b$ is odd and $c$ is the unique point in $\bB$ not mapped into $\bA$. In both cases, this uniquely determines $\Pi$ (since $c$ is not moved by $f$ when $b$ is even). Now suppose that $P$ does not contain a cycle of $x$, so $\Pi$ has type (VII). Here, the unique point in $\bA \setminus \{\on, \tw\}$ not mapped into $\bB$ by $x$ is $a$, and the unique point in $\bB \setminus \{\tr,\fo\}$ not mapped into $\bA$ by $x$ is $b$; this determines $\Pi$.

Now suppose $m=3$, so $\Pi$ has type (III), (VI) or (IX)--(XII). First assume one of $A$, $B$ or $C$ is fixed by $x$, so $\Pi$ has type (III). The unique point in $\bB$ not mapped back into $\bB$ is $b$. There is a unique part of $\Sigma$ all of whose points other than exactly two are mapped into $\bC$, and $a$ is the unique such point other than $b'$. This determines $\Pi$.

Now assume none of $A,B$ and $C$ are fixed by $x$. If all but one point of one of $A$, $B$ or $C$ is mapped back into that part, then $\Pi$ has type (VI). Here the unique point in $\bA$ not mapped back into $\bA$ is $b$, and the unique point in $\bB \setminus \{\tr,\fo\}$ not mapped into $\bC$ is $a$; this determines $\Pi$.

To complete the analysis of the case $m=3$, we may assume that $\Pi$ is of type (IX)--(XII). If all but one point of $A \setminus \{1,2\}$ are mapped into $B$, then $\Pi$ has type (IX). The unique point in $A \setminus \{1,2\}$ not mapped into $B$ is $a$, the point in $C$ that is mapped into $A$ is $c$ and the points in $B \setminus \{3,4\}$ not mapped into $A$ are $b$ and $d$; this determines $\Pi$. If all but one point of $B \setminus \{3,4\}$ are mapped into $C$, then $\Pi$ has type (X), and if all but one point of $C \setminus \{5,6\}$ are mapped into $A$, then $\Pi$ has type (XI); in both cases $\Pi$ is determined as before. Finally, suppose $\Pi$ has type (XII). The point in $A$ mapped into $C$ is $a$, and the point in $B$ mapped into $C$ is $b$. In addition, the two points in $A \setminus \{1,2,a\}$ not mapped into $B$ are $c$ and $d$, and the two points in $B \setminus \{3,4,b\}$ not mapped into $A$ are $e$ and $g$. This determines $\Pi$.

Finally, if $m=2$ then $\Pi$ has type (V) and we can determine $\Pi$ by interchanging $\fo$ and $\fv$ in $\Sigma$.

This proves that $f$ is injective and completes the proof of (ii).
\end{proof}

\begin{lem}\label{l:n_over_two}
Let $G = S_n$ with $n \geqs 30$ even and suppose $x \in G$ has shape $[p^k,1^{n-pk}]$, where $p$ is a prime, $k \geqs 1$ and 
$(p,k) \not\in \{ (2,1), (2,2), (2,3), (3,2)\}$.
If $l = \frac{n}{2}$, then 
\[
\fpr(x,\Pi_l) < \frac{1}{l^2}.
\]
\end{lem}

\begin{proof}
First assume that $p$ is odd. Let $\Pi \in \Fix_l(x)$. Since $n/l = 2 < p$, no cycle of $x$ is contained in a part of $\Pi$. Consequently, if $k$ is odd, then $\Fix_l(x)$ is empty, so we will assume that $k \geqs 2$ is even. Then $x$ moves $t=pk/2$ parts in $k/2$ orbits of size $p$ and thus $l-t$ parts of $\Pi$ are fixed by $x$. With this in mind, let us compute $\fpr(x,\Pi_l)$. 

First observe that $|\Pi_l| = (2l)!/(l!\,2^l)$. To construct an $x$-stable partition in $\Pi_l$,  we must first partition the $2l-2t$ fixed points of $x$ into $l-t$ parts of size $2$. The number of ways this can be done is $(2l-2t)!/((l-t)!\,2^{l-t})$. We must then partition the remaining $2t=pk$ points into $pk/2$ parts of size $2$, which are permuted by $x$ in $k/2$ orbits of size $p$. This amounts to partitioning the $pk/2$ parts into $k/2$ sets of the form $\{ \{i_1,j_1\}, \dots, \{i_p,j_p\} \}$  where $(i_1,\dots,i_p)$ and $(j_1,\dots,j_p)$ are cycles of $x$. Therefore, an $x$-stable partition of these $pk$ points corresponds to a partition of the cycles of $x$ into pairs, together with a choice of partition $\{ \{i_1,j_{i\s}\}, \dots, \{i_p,j_{p\s}\} \}$ for each pair $\{ (i_1,\dots,i_p), (j_1,\dots,j_p) \}$, where $\s$ is a power of $(1,2,\dots,p)$. There are $k!/((k/2)!\,{2}^{k/2})$ ways of partitioning the $k$ cycles of $x$ into pairs, and for each pair there are $p$ choices for $\s$. Therefore, there are $p^{k/2} \cdot k!/((k/2)!\,{2}^{k/2})$ different ways to partition the $pk$ points moved by $x$. In this way, we conclude that
\[
\fpr(x,\Pi_l) = p^{k/2} \cdot \frac{k!}{{\left(\frac{k}{2}\right)}!\,{2}^{k/2}} \cdot \frac{(2l-2t)!}{(l-t)!\,2^{l-t}} \cdot \frac{l!\,2^l}{(2l)!}.
\]
By applying Stirling's approximation, we calculate
\begin{align}
\fpr(x,\Pi_l) &= p^{k/2} \cdot \frac{k!}{{\left(\frac{k}{2}\right)}!\,{2}^{k/2}} \cdot \frac{(2l-2t)!}{(l-t)!\,2^{l-t}} \cdot \frac{l!\,2^l}{(2l)!} \label{eq:stirling_start}\\
              &\leqs \frac{e^3}{2^{3/2}\pi^{3/2}} \cdot \frac{p^{k/2}k^{k+1/2}e^{-k}2^{2l-2t+1/2}(l-t)^{2l-2t+1/2}e^{-(2l-2t)}l^{l+1/2}e^{-l}2^l2^{k/2+1/2}}{k^{k/2+1/2}e^{-k/2}2^{k/2}(l-t)^{l-t+1/2}e^{-(l-t)}2^{l-t}l^{2l+1/2}e^{-2l}2^{2l+1/2}} \nonumber \\
              &= \frac{e^3}{2\pi^{3/2}} \cdot \left(\frac{e}{2t}\right)^{(p-1)k/2} \cdot \frac{t^t(l-t)^{l-t}}{l^l} \label{eq:stirling_end} 
              \end{align}
which in turn is at most
\[     
\frac{t^t(l-t)^{l-t}}{l^l} \leqs \frac{5^5(l-5)^{l-5}}{l^l} < \frac{5^5}{15^3} \cdot \frac{1}{l^2} < \frac{1}{l^2}
\]
since $l = n/2 \geqs 15$ and $t=pk/2 \geqs 5$ (because $p \geqs 3$ and we are assuming $(p,k) \ne (3,2)$).

Now assume that $p=2$ and $k \geqs 4$. In this case, a partition $\Pi \in \Fix_l(x)$ has $i$ pairs of parts interchanged by $x$, and $l-2i$ parts stabilised by $x$. Moreover, $l-k$ of these $l-2i$ parts are fixed pointwise. Therefore, by counting as we did above,
\[
\fpr(x,\Pi_l) = \left( \sum_{i=0}^{\lfloor{k/2}\rfloor} \frac{k!}{(k-2i)!\,i!\,2^i} \cdot 2^i \right) \cdot \frac{(2l-2k)!}{(l-k)!\,2^{l-k}} \cdot \frac{l!\,2^l}{(2l)!}.
\]
Observe that
\[
\sum_{i=0}^{\lfloor{k/2}\rfloor} \frac{{\lfloor{k/2}\rfloor}!}{(k-2i)!i!} = \sum_{i=0}^{\lfloor{k/2}\rfloor} \binom{\lfloor{k/2}\rfloor}{i} \frac{(\lfloor{k/2}\rfloor-i)!}{(k-2i)!}
  \leqs \sum_{i=0}^{\lfloor{k/2}\rfloor} \binom{\lfloor{k/2}\rfloor}{i} \frac{1}{(\lfloor{k/2}\rfloor-i+1)^{\lfloor{k/2}\rfloor-i}}
\]
which is at most
\[
 \sum_{i=0}^{\lfloor{k/2}\rfloor} \binom{\lfloor{k/2}\rfloor}{i} \frac{1}{2^i} =\left(\frac{3}{2}\right)^{\lfloor{k/2}\rfloor}.
\]
Therefore,
\begin{align*}
\fpr(x,\Pi_l) &\leqs \left(\frac{3}{2}\right)^{\lfloor{k/2}\rfloor} \cdot \frac{k!}{{\lfloor{k/2}\rfloor}!}  \cdot \frac{(2l-2k)!}{(l-k)!\,2^{l-k}} \cdot \frac{l!\,2^l}{(2l)!} \\
              & \leqs \left(\frac{3}{2}\right)^{\lfloor{k/2}\rfloor} \cdot \lceil{k/2}\rceil \cdot 2^{k/2} \cdot \frac{k!}{{\lceil{k/2}\rceil}!\,2^{k/2}} \cdot \frac{(2l-2k)!}{(l-k)!\,2^{l-k}} \cdot \frac{l!\,2^l}{(2l)!}.
\end{align*}
By repeating the manipulations between \eqref{eq:stirling_start} and \eqref{eq:stirling_end} (with $p=2$ and $t=pk/2=k$) we get
\[
\fpr(x,\Pi_l) \leqs \left(\frac{3}{2}\right)^{\lfloor{k/2}\rfloor} \left\lceil{\frac{k}{2}}\right\rceil \, \frac{e^3}{2\pi^{3/2}} \left(\frac{e}{2k}\right)^{k/2} \frac{k^k(l-k)^{l-k}}{l^l} \leqs \frac{3e^4}{8\pi^{3/2}} \left(\frac{3e}{4k}\right)^{k/2-1} \frac{k^k(l-k)^{l-k}}{l^l}.
\]
If $k \geqs 5$, then
\[
\fpr(x,\Pi_l) \leqs \frac{k^k(l-k)^{l-k}}{l^l} \leqs \frac{5^5(l-5)^{l-5}}{l^l} < \frac{5^5}{15^3} \frac{1}{l^2} < \frac{1}{l^2},
\]
and for $k=4$ we get 
\[
\fpr(x,\Pi_l) \leqs \frac{3e^4}{8\pi^{3/2}} \frac{3e}{16} \frac{4^4{11}^{11}}{15^{13}} \frac{1}{l^2} < \frac{1}{l^2}.
\]
This completes the proof.
\end{proof}

We can now establish the main fixed point ratio bounds.

\begin{proof}[Proof of Lemma~\ref{l:fprs_odd}]
Let $x \in G$ have shape $[p^k,1^{n-pk}]$, where $p$ is an odd prime and $k \geqs 1$. For $n < 30$ we can verify the desired bound using {\sc Magma}, so assume that $n \geqs 30$.

First assume $(p,k)=(3,1)$. Up to conjugacy, we may assume $x=(1,2,3)$. A partition in $\Pi_l$ is stabilised by $x$ if and only if it has a part containing $\{1,2,3\}$. By counting the partitions with this property, we deduce that 
\[
\fpr(x,\Pi_l) \leqs \frac{(n-3)!}{\left(\frac{n}{l}-3\right)! \, \left(\left(\frac{n}{l}\right)!\right)^{l-1} \, (l-1)!} \cdot \frac{\left(\left(\frac{n}{l}\right)!\right)^l \, l!}{n!} = \frac{\left(\frac{n}{l}\right)\left(\frac{n}{l}-1\right)\left(\frac{n}{l}-2\right)l}{n(n-1)(n-2)} < \frac{1}{l^2}.
\]
A similar calculation shows that $\fpr(x,\Pi_l)<\frac{1}{l^4}$ when $(p,k) = (5,1)$.

Next assume $(p,k)=(3,2)$, say $x=(1,2,3)(4,5,6)$. If $l < n/2$, then a partition stabilised by $x$ either has a part containing $\{1,2,3,4,5,6\}$, or a part containing $\{1,2,3\}$ and another containing $\{4,5,6\}$. Therefore, 
\begin{align*}
\fpr(x,\Pi_l) &\leqs \left( \frac{(n-6)!}{\left(\frac{n}{l}-6\right)!\,\left(\left(\frac{n}{l}\right)!\right)^{l-1}\,(l-1)!} + \frac{(n-6)!}{\left(\left(\frac{n}{l}-3\right)!\right)^2\,\left(\left(\frac{n}{l}\right)!\right)^{l-2}\,(l-2)!} \right) \cdot \frac{\left(\left(\frac{n}{l}\right)!\right)^l\,l!}{n!} \\
            &< \frac{1}{l^5} + \frac{l-1}{l^5} \cdot \frac{n(n-1)(n-2)}{(n-3)(n-4)(n-5)} < \frac{1}{l^5} + \frac{1.4}{l^4}
\end{align*}
and thus $\fpr(x,\Pi_l) < \frac{1}{l^3}$. Now assume $l=n/2$, so $l \geqs 15$. Here we must also consider the partitions containing three parts for which $\{1,2,3\}$ and $\{4,5,6\}$ are transversals. The proportion in $\Pi_l$ of such partitions is
\[
3 \cdot \frac{(n-6)!}{2^{l-3}\,(l-3)!} \cdot \frac{2^l \, l!}{n!} \leqs \frac{3}{(2l-3)(2l-4)(2l-5)} \leqs \frac{0.6}{l^3}
\]
and this gives $\fpr(x,\Pi_l) < \frac{1}{l^5} + \frac{1.4}{l^4} + \frac{0.6}{l^3} < \frac{1}{l^3}$.

Now assume $(p,k) \not\in \{ (3,1), (3,2) \}$. If $l = n/2$, then the desired result follows from Lemma \ref{l:n_over_two}. Similarly, if $n/6 < l \leqs n/3$, then we combine  Lemma \ref{l:injection_odd}(i) with the above calculation in the case $(p,k)=(3,1)$. Finally, let us assume $l \leqs n/6$. If $p=3$ and $k \geqs 3$, then we appeal to  Lemma \ref{l:injection_odd}(iii) and the above calculation for $(p,k)=(3,2)$. Similarly, if $p \geqs 5$ then the desired bound follows via Lemma \ref{l:injection_odd}(ii) and the case $(p,k)=(5,1)$ handled above.
\end{proof}

\begin{proof}[Proof of Lemma~\ref{l:fprs_even}]
Set $x = (1,2)(3,4)\cdots(2k-1,2k)$. As in the proof of Lemma~\ref{l:fprs_odd}, we may assume that $n \geqs 30$. The desired bound is straightforward to verify when $k=1$ or $2$  (we proceed as in the proof of Lemma~\ref{l:fprs_odd} with $(1,2,3)$ and $(1,2,3)(4,5,6)$, respectively).

Next assume $k=3$ and $l \leqs n/6$. The partitions stabilised by $x$ are exactly those which have a part containing $\{1,2,3,4,5,6\}$, or a part containing the union of two of $\{1,2\}$, $\{3,4\}$, $\{5,6\}$ and another part  containing the third, or a part containing $\{1,2\}$, another containing $\{3,4\}$ and a third which contains $\{5,6\}$. Therefore,
\begin{align*}
\fpr(x,\Pi_l) &\leqs \left( \frac{(n-6)!}{\left(\frac{n}{l}-6\right)!\,\left(\left(\frac{n}{l}\right)!\right)^{l-1}\,(l-1)!} + \frac{3(n-6)!}{\left(\frac{n}{l}-2\right)!\,\left(\frac{n}{l}-4\right)!\,\left(\left(\frac{n}{l}\right)!\right)^{l-2}\,(l-2)!} \right. \\
            & \qquad + \left. \frac{(n-6)!}{\left(\left(\frac{n}{l}-2\right)!\right)^3\,\left(\left(\frac{n}{l}\right)!\right)^{l-3}\,(l-3)!} \right) \cdot \frac{\left(\left(\frac{n}{l}\right)!\right)^l\,l!}{n!} \\
            &< \frac{1}{l^5} + \frac{3(l-1)}{l^5} + \frac{(l-1)(l-2)}{l^5} \leqs \frac{1}{l^3}. 
\end{align*}

Now assume $k=3$ and $l=n/2$. Here there are three additional types of partition to  consider. In one case, we consider those partitions which contain two parts for which $\{1,2\}$ and $\{3,4\}$ are transversals and a third part which is $\{5,6\}$. We calculate that the proportion of $l$-part partitions satisfying this condition is equal to 
\[
\frac{(n-6)!}{2^{l-3}\,(l-3)!} \cdot \frac{2^l \, l!}{n!} \leqs \frac{1}{(n-3)(n-4)(n-5)} \leqs \frac{1.6}{n^3}.
\]
The two other cases arise from interchanging the roles of $\{1,2\}$, $\{3,4\}$ and $\{5,6\}$, so we obtain the same proportion and $\fpr(x,\Pi_l) < \frac{8}{n^3} + \frac{4.8}{n^3} < \frac{1}{l^2}$.

To complete the argument, we may assume that either $k=3$ and $n/6 < l < n/2$, or $k \geqs 4$. If $k \geqs 4$ and $l=n/2$ then Lemma~\ref{l:n_over_two} gives $\fpr(x,\Pi_l) < \frac{1}{l^2}$ as required, so we may assume $l < n/2$. Here Lemma~\ref{l:injection_even} and the above bounds imply that
\[
\fpr(x,\Pi_l) \leqs \left\{ 
\begin{array}{ll}
\fpr((1,2)(3,4)(5,6),\Pi_l) < 1/l^3 & \text{if $3 \leqs l \leqs n/6$} \\
\fpr((1,2)(3,4),     \Pi_l) < 1/l^2 & \text{if $n/6 < l < n/2$.}
\end{array}
\right.
\]
Therefore, to complete the proof of the lemma, we may assume that $l=2$. 

First assume $k < n/2$. If $\Pi \in \Fix_l(x)$, then each cycle of $x$ must be  contained in one of the two parts of $\Pi$, so $\Fix_l(x) \subseteq \Fix_l((1,2)(3,4)(5,6))$ and 
\[
\fpr(x,\Pi_l) \leqs \fpr((1,2)(3,4)(5,6),\Pi_l) < \frac{1}{l^3}.
\]
Finally, suppose $k=n/2$. If $\Pi \in \Fix_l(x)$, then either each cycle of $x$ is contained in a part of $\Pi$, or each cycle of $x$ contains a point from each part of $\Pi$. Since there are at most $2^{n/2-1}$ partitions of each of these types, it follows that 
\[
\fpr(x,\Pi_l) \leqs 2\cdot2^{n/2-1} \cdot \frac{2\,{\left(\left(\frac{n}{2}\right)!\right)}^2}{n!} \leqs  2^{n/2} \cdot \frac{2\,(\frac{n}{2})^{n+1}\,e^{-n}e^2}{n^{n+1/2}e^{-n}(2\pi)^{1/2}} < \frac{1}{8} = \frac{1}{l^3},
\]
noting that $n \geqs 30$. This completes the proof.
\end{proof}

We are now in a position to prove Theorem~\ref{t:sn_even_us}.

\begin{proof}[Proof of Theorem \ref{t:sn_even_us}]
If $n < 30$, then the result can be verified computationally with {\sc Magma} (see the end of Section~\ref{ss:prelims_spread}). Now assume $n \geqs 30$. 

Let $G=S_n$ and let $s$ be an $n$-cycle. Then $\M(G,s) = \mathcal{I} \cup \mathcal{P}$, where $\mathcal{I}$ consists of exactly one imprimitive subgroup $S_{n/l} \wr S_l$ for each divisor $1< l < n$ of $n$, and $\mathcal{P}$ contains primitive subgroups of the form ${\rm P}\Gamma{\rm L}_d(q)$ for pairs $(q,d)$ satisfying $n=(q^d-1)/(q-1)$ (see \cite[Theorem~3]{Jones02}). Moreover, by the proof of \cite[Proposition~6.7]{BGK} we see that $\mathcal{P}$ contains at most $(n-1)/d$ subgroups of the form ${\rm P}\Gamma{\rm L}_d(q)$ for each pair $(q,d)$. 

Suppose $x \in G$ has prime order and let $H \in \M(G,s)$. First assume $H \in \mathcal{P}$, say $H= {\rm P}\Gamma{\rm L}_d(q)$ with $n=(q^d-1)/(q-1)$. Then $|H| \leqs n^{\log_2{n}+1}$ and
\[
|x^G| \geqs \frac{2^{3n/4} \left(\frac{n}{e}\right)^{n/4}}{8\sqrt{\pi n}},
\]
by \cite[Lemma~6.6]{BGK}. Since $|\mathcal{P}| \leqs \frac{1}{2}(n-1)\log_2n$ (see \cite[Lemma 3.9]{BH}), it follows that
\[
\sum_{H \in \mathcal{P}} \fpr(x,G/H) \leqs \sum_{H \in \mathcal{P}} \frac{|H|}{|x^G|} \leqs \frac{1}{2}(n-1)\log_2{n} \cdot \frac{8n^{\log_2{n}+1}\sqrt{\pi n}}{2^{3n/4} \left(\frac{n}{e}\right)^{n/4}} \leqs \frac{n^{\log_2{n}+3}4\sqrt{\pi}}{n^{n/4} \left(\frac{8}{e}\right)^{n/4}} < 0.01.
\]

Now assume $H \in \mathcal{I}$, say $H = S_{n/l} \wr S_l$. Note that the action of $G$ on $G/H$ is equivalent to the action of $G$ on $\Pi_l$. Therefore, 
if $x \not\in (1,2)^G \cup (1,2)(3,4)^G \cup (1,2,3)^G$, then Lemmas~\ref{l:fprs_odd} and~\ref{l:fprs_even} imply that 
\[
\sum_{H \in \mathcal{I}} \fpr(x,G/H) <  \sum_{\substack{l \mid n \\ 1 < l \leqs \frac{n}{6}}} \frac{1}{l^3} + \frac{5^2+4^2+3^2+2^2}{n^2} < \sum_{l=2}^{\infty} \frac{1}{l^3} + \frac{54}{n^2} < 0.21 + \frac{54}{30^2} < 0.27.
\]
Similarly, if $x \in (1,2)(3,4)^G \cup (1,2,3)^G$, then
\[
\sum_{H \in \mathcal{I}} \fpr(x,G/H) < \frac{33}{128} + \sum_{\substack{l \mid n \\ 2 < l < \frac{n}{2}}} \frac{1}{l^2} + \frac{6}{n^2} < \frac{33}{128} + \frac{\pi^2}{6} - \frac{5}{4} + \frac{6}{30^2} < 0.66.
\]
\label{page:bounds}

In order to establish the desired bound $u(G) \geqs 2$, it suffices to show that if $x,y \in G$ have prime order, then there exists an $n$-cycle $z$ such that
$G = \<x,z\> = \<y,z\>$. Let $P(x,s)$ be the probability that $x$ and a randomly chosen conjugate of $s$ do not generate $G$ (see \eqref{e:pxs}). Then by Lemma \ref{l:bg1}, it is sufficient to show that 
\[
P(x,s) + P(y,s) < 1.
\]
There are three cases to consider, according to the possibilities for $x$ and $y$.

\vs

\noindent \emph{Case 1. $x,y \not\in (1,2)^G$ and either $x$ or $y$ is not in $(1,2)(3,4)^G \cup (1,2,3)^G$.}

\vs

By applying the bound in \eqref{e:pbd}, together with the above fixed point ratio estimates, we obtain 
\[
P(x,s) + P(y,s) \leqs \sum_{H \in \M(G,s)} \fpr(x,G/H) \; + \!\!\! \sum_{H \in \M(G,s)} \fpr(y,G/H) < 0.02 + 0.27 + 0.66 = 0.95
\]
and the result follows.

\vs

\noindent \emph{Case 2.} $x \in (1,2)^G$.

\vs

Without loss of generality, we may assume that $x = (i,j)$ and 
\[
y = (1,2,\dots,p)(p+1,p+2,\dots,2p)\cdots((k-1)p+1,(k-1)p+2,\dots,kp)
\]
for some prime $p$ and integer $k \geqs 1$. In the proof of \cite[Theorem~2]{Binder70II}, Binder constructs an $n$-cycle $z$ such that $G = \<y,z\>$. Typically, we will show that there exists $g \in N_G(\<y\>)$ such that $G = \<x^g,z\>$, whence $G = \<x,z^{g^{-1}}\> = \<y,z^{g^{-1}}\>$ (in one particular case below, we work with a different $n$-cycle to the one given by Binder). Let $\mathcal{S} \subseteq [n]$ be the set of points moved by both $x$ and $y$. We will consider five cases. 

\vs

\noindent \emph{Case 2(a). $k=1$.}

\vs

Set $z=(1,2,\dots,n)$ and note that $G = \<y,z\>$. By conjugating by an element of $N_G(\<y\>)$ if necessary, we may assume that $x$ is $(p+1,p+2)$, $(p,p+1)$ or $(1,2)$ if $|\mathcal{S}| = 0$, $1$ or $2$, respectively. In each case, $G = \<x,z\>$ by Lemma~\ref{l:small_elements}(i) and the result follows.

\vs

\noindent \emph{Case 2(b). $k \geqs 2$, $p \geqs 3$ and $kp < n$.}

\vs

Let $u=1$ if $3$ divides $n-kp$ and $u=0$ otherwise. Then consider the $n$-cycle
\[
z = (1, \alpha_1, p+1, 2p+1, \dots, (k-1)p+1, kp, kp-1, \dots, 3, \beta_1, \beta_2, \dots, \beta_t, 2, \gamma_1, \dots, \gamma_u),
\]
where the first ellipsis represents an arithmetic sequence with difference $p$ and the second ellipsis represents the entire decreasing sequence from $kp-2$ to $4$, omitting any numbers that occur earlier in the cycle. By \cite[Theorem~2]{Binder70II}, we have $G = \<y,z\>$.

By arguing as in Case~2(a), we may assume that $x=(1, \alpha_1)$ if $|\mathcal{S}|=1$ and $x=(kp-1,kp)$ or $x=(p,p+2)$ if $|\mathcal{S}|=2$. In both cases, $G = \<x,z\>$ (see Remark~\ref{r:small_elements}). Now assume $|\mathcal{S}|=0$, so $n-kp \geqs 2$. If $n-kp > 2$ then we may assume that $x=(\beta_1,\beta_2)$ and again $G = \<x,z\>$. 

Therefore, to complete the analysis of Case 2(b), we may assume that $n-kp = 2$ and $x = (kp+1,kp+2)$. Here we must deviate from the proof of \cite[Theorem~2]{Binder70II} and we choose a different $n$-cycle $z$. (Indeed, if $3$ divides $n$, then  $x = (\alpha_1,\beta_1)$ and $\beta_1z^3 = \alpha_1$, so it is straightforward to see that $G \ne \<x,z\>$.) In particular, let 
\[
z = (p+1, 2p+1, \dots, (k-1)p+1, kp, \dots, 2, 1, kp+1, kp+2)
\]
(adopting the same conventions as above for the ellipses) and observe that 
\begin{align*}
yz     & = (p+1,p,kp+1,kp+2)(2p,2p+1)\cdots((k-1)p,(k-1)p+1) \\
(yz)^2 & = (p+1,kp+1)(p,kp+2).
\end{align*}

First we claim that $G = \< y,z \>$. To see this, suppose that $\<y,z\>$ stabilises a nontrivial partition $\Pi$ of $[n]$ (into parts of equal size). Then $(yz)^2$ stabilises $\Pi$. If $(yz)^2$ acts nontrivially on the set of parts of $\Pi$, then each part has size two and two parts of $\Pi$ are either $\{p,p+1\}$ and $\{kp+1,kp+2\}$, or $\{p,kp+1\}$ and $\{p+1,kp+2\}$. However, since $z$ stabilises $\Pi$, the parts of $\Pi$ must be of the form $\{a,az^{n/2}\}$ with $a \in [n]$. Since $kp+2=(kp+1)z$ and $p+1=(kp+2)z$, both of these options are impossible and we have reached a contradiction. Therefore, each cycle of $(yz)^2$ is contained in a part of $\Pi$. Let $A$ be the part of $\Pi$ containing $p$ and $kp+2$. Since $y$ fixes $kp+2$, we conclude that $y$ fixes $A$. In particular, $1,2,\dots,p \in A$. Since $z$ stabilises $\Pi$, we know that $A = \{ 1z^{li} \mid 1 \leqs i \leqs n/l\}$ for some divisor $1 < l < n$. But this is a contradiction since $2z = 1 \in A$ and we conclude that $\<y,z\>$ does not stabilise a nontrivial partition of $[n]$. In particular, $\<y,z\>$ is a primitive subgroup of $G$ containing a double transposition, so $G = \<y,z\>$ (see \cite[Example~3.3.1]{DM_book}, noting that $z$ is odd).

Finally, we observe that $G = \<x,z\>$ since $x = (kp+1,kp+2)$ and $(kp+1)z=kp+2$ (see Remark~\ref{r:small_elements}).

\vs

\noindent \emph{Case 2(c). $k \geqs 2$, $p \geqs 3$ and $kp = n$.}

\vs

Here we set
\[
z = (1, 2, p+1, \dots, (k-1)p+1, pk, pk+1, \dots, 3)
\]
and we note that $G = \<y,z\>$ by the proof of \cite[Theorem~2]{Binder70II}. By replacing $x$ by a suitable $N_G(\<y\>)$-conjugate, we may assume that $x = (1,2)$ or $x=(pk,pk+1)$. In both cases, it is easy to see that $G = \<x,z\>$.

\vs

\noindent \emph{Case 2(d). $k \geqs 2$, $p=2$ and $2k < n$.}

\vs

Here we define 
\[
z = (1, \alpha_1, \dots, \alpha_s, 2, \beta_1, \dots, \beta_t, 3, 5, \dots, 2k-1, 2k, 2k-2, \dots, 6, 4),
\]  
where $t=0$ if $n-2k$ is even, otherwise $s=0$ (there are no additional conditions imposed on the $\a_i$ and $\b_j$). Again, we have $G = \<y,z\>$ by the proof of \cite[Theorem~2]{Binder70II}, and without loss of generality we may assume that $x$ is one of $(2k-1,2k)$,  $(2k-2,2k)$ or
\[
x = \left\{ \begin{array}{ll} (\alpha_1,\alpha_2) & \text{if $n-2k$ is even} \\ (\beta_1,\beta_2) & \text{if $n-2k$ is odd} \end{array} \right.
\quad \text{or} \quad
x = \left\{ \begin{array}{ll} (1, \alpha_1) & \text{if $n-2k$ is even} \\ (2,\beta_1) & \text{if $n-2k$ is odd.} \end{array} \right.
\]  
In every case, one checks that $G = \<x,z\>$.

\vs

\noindent \emph{Case 2(e). $k \geqs 2$, $p=2$ and $2k=n$.}

\vs

Set
\[
z = (1,3,2,4,5,7,\dots,2k-1,2k,2k-2,\dots,6)
\]
and note that $G = \<y,z\>$ by the proof of \cite[Theorem~2]{Binder70II}. Without loss of generality, we may assume that $x = (1,3)$ or $x=(2k-1,2k)$ and in both cases we have $G = \<x,z\>$. 

\vs

\noindent \emph{Case 3. $x,y \in (1,2)(3,4)^G \cup (1,2,3)^G$.}

\vs

If $x,y \in (1,2,3)^G$, then we may assume that $x = (1,2,3)$ and depending on the size of $\mathcal{S}$ we may assume that the support of $y$ is one of $\{2,3,4\}$, $\{3,4,5\}$ and $\{4,5,6\}$. Moreover, since we are at liberty to replace $y$ by $y^{-1}$, we may assume that
\[
y \in \{ (2,3,4), \, (3,4,5), \, (4,5,6) \}.
\] 
Similarly, if $x \in (1,2)(3,4)^G$ and $y \in (1,2,3)^G$, then we may assume that $x = (1,2)(3,4)$ and
\[ 
y \in \{ (1,2,3), \, (1,2,5), \, (2,3,5), \, (4,5,6), \, (5,6,7) \}.
\] 
Finally, if $x,y \in (1,2)(3,4)^G$, then we may assume that $x = (1,2)(3,4)$ and 
\begin{align*} 
y \in \{ &(1,3)(2,4), \, (1,2)(3,5), \, (1,2)(5,6), \, (2,3)(5,6), \, \\
         &(3,6)(4,5), \, (1,6)(4,5), \, (1,5)(6,7), \, (5,6)(7,8) \}.
\end{align*}
In all three cases, Lemma~\ref{l:small_elements} implies that $G = \<x,z\> = \<y,z\>$ for $z=(1,2,\dots,n)$, unless $x=(1,2)(3,4)$ and $y=(1,3)(2,4)$. In this exceptional case, $z = (1,2,\dots,n)^{(2,3)}$ has the desired property.

\vs

This completes the proof of Theorem~\ref{t:sn_even_us}.
\end{proof}

We can now determine the spread and uniform spread of even-degree symmetric groups.

\begin{thm}\label{t:sn_even}
Let $G = S_n$ with $n \geqs 8$ even. Then $s(G) = u(G) = 2$.
\end{thm}

\begin{proof}
By Theorem~\ref{t:sn_even_us}, we have $2 \leqs u(G) \leqs s(G)$, so it suffices to prove that $s(G) \leqs 2$. As noted in the proof of Proposition~\ref{p:sn_odd}, if 
\[
G = \<(1,2),s\> = \<(2,3),s\> = \<(1,3),s\>,
\]
then $s$ is an $n$-cycle. By Lemma~\ref{l:small_elements}(i), if 
\[
G = \<(1,2), (1,2,\dots,n)^g\> = \<(2,3), (1,2,\dots,n)^g\>,
\]
then $1g \not\equiv 2g \imod{2}$ and $2g \not\equiv 3g \imod{2}$. This implies that $1g \equiv 3g \imod{2}$ and thus $G \ne \<(1,3), (1,2,\dots,n)^g\>$. This shows that  $s(G) \leqs 2$, as required. 
\end{proof}

\subsection{Uniform domination}

In order to complete the proof of Theorem~\ref{t:symmain}, it remains to establish the bounds on the total and uniform domination numbers. We begin by establishing lower bounds on the total domination numbers (for use in Section \ref{s:alt}, it is convenient to include alternating groups in the following proposition).

\begin{prop}\label{p:sn_an_tdn_lower}
Let $n \geqs 5$ and let $p$ be the smallest prime divisor of $n$.
\begin{itemize}\addtolength{\itemsep}{0.2\baselineskip}
\item[{\rm (i)}]   $\gamma_t(S_n) \geqs \log_2{n}$.
\item[{\rm (ii)}]  $\gamma_t(A_n) \geqs \log_p{n}$.
\item[{\rm (iii)}] $\gamma_t(A_n) \geqs 3$ if $n$ is composite.
\end{itemize}
\end{prop}

\begin{proof}
Let $G$ be $S_n$ or $A_n$ and let $S = \{s_1,\dots,s_c\}$ be a total dominating set for $G$. Without loss of generality, assume that $s_1,\dots,s_j$ are $n$-cycles and $s_{j+1},\dots,s_c$ are not $n$-cycles (we allow $j=0$). To begin with, we will prove that $\gamma_t(G) \geqs \log_p{n}$. This is clear if $n$ is prime, so we may assume $n$ is composite.

If $1 \leqs i \leqs j$ then $s_i$ stabilises a partition $C_i$ of $[n]$ with parts $C_{i1}, \dots, C_{ip}$ of size $n/p$. Similarly, if $j < i \leqs c$ then $s_i$ stabilises a proper subset $A_i$ of $[n]$. For $j < i \leqs c$, it will be convenient to write $C_i = (C_{i1},C_{i2}) = (A_i, [n]\setminus A_i)$. Set $R = [p]^j \times [2]^{c-j}$ and define $f\colon [n] \to R$ as $f(x) = (l_1,\dots,l_c)$ where $x \in C_{il_i}$. For $1 \leqs i \leqs j$, let $G_{C_i}$ be the stabiliser in $G$ of the partition $C_i$ (so $G_{C_i}$ is isomorphic to $(S_{n/p} \wr S_p) \cap G$). Similarly, if $j < i \leqs c$ then let $G_{\{A_i\}}$ be the setwise stabiliser of $A_i$.

We claim that $|R| \geqs n-1$. To see this, suppose that $f(x)=f(y)=f(z)$ for three distinct points $x,y,z \in [n]$. Then $(x, y, z)$ is contained in $G_{C_i}$ for $i \leqs j$ and $G_{\{A_i\}}$ for $i > j$. Therefore, $G \ne \<s_i,(x,y,z)\>$ for all $1 \leqs i \leqs c$, which is a contradiction since $S$ is a total dominating set for $G$. It follows that the preimage under $f$ of any point has size at most two. In fact, if $f(x)=f(y)$ and $f(z)=f(w)$ for four distinct points $x,y,z,w \in [n]$, then $(x,y)(z,w) \in G$ is contained in $G_{C_i}$ for $i \leqs j$ and in $G_{\{A_i\}}$ for $i > j$. As before, this is a contradiction and we deduce that at most one point in $R$ has a preimage of size two. This justifies the claim.

Now $|R| = p^j2^{c-j} \leqs p^c$, so $p^c \geqs n-1$. Since $p$ divides $n$, it follows that $p^c \geqs n$ and thus $\gamma_t(G) \geqs \log_p{n}$. This establishes part (ii) of the proposition, and also part (i) when $n$ is even. Therefore, it remains to prove part (iii), together with part (i) when $n$ is odd. 

In view of part (ii), in order to prove (iii) we may assume that $G=A_n$ and $n=p^2$, so $p$ is odd. Suppose $\gamma_t(G)=2$, say $\{s_1,s_2\}$ is a total dominating set. If $s_1$ and $s_2$ are both $n$-cycles, then $s_1 \in H_1$ and $s_2 \in H_2$ for stabilisers $H_1$ and $H_2$ of $p$-part partitions of $[n]$. But Lemma \ref{l:udn} implies that $G$ has a base of size two in its action on the set $\Pi_p$ of $p$-part partitions, which is a contradiction (see  \cite[Remark~5.3]{James}). Therefore, we may assume that $s_2$ is not an $n$-cycle, so $2p \geqs |R| \geqs n - 1 = p^2-1$, which is absurd. We conclude that $\gamma_t(G) \geqs 3$.

Finally, let us assume $G = S_n$ and $n$ is odd. Here we allow $n$ to be prime. As above, if $i > j$, then $s_i$ stabilises a proper subset $A_i$ of $[n]$ and we write $(C_{i1},C_{i2}) = (A_i, [n]\setminus A_i)$. Also observe that $s_1,\dots,s_j \in A_n$, so for each $g \in A_n$ there exists $j < i \leqs c$ such that $S_n=\<g,s_i\>$. In particular, note that $j<c$.

Define a map $f'\colon [n] \to [2]^{c-j}$ as $f'(x)=(l_{j+1},\dots,l_c)$ where $x \in C_{il_i}$. First assume that $j=0$. Suppose that there exist two distinct points $x,y \in [n]$ such that $f'(x)=f'(y)$. Then $(x, y) \in G$ is contained in $G_{\{A_i\}}$ for all $i$, so $G \ne \<s_i,(x,y)\>$  and we have reached a contradiction. Therefore, $f'$ is injective, which implies that $2^c \geqs n$ and $\gamma_t(G) \geqs \log_2{n}$. Now assume $j \geqs 1$. The above argument for $f$ implies that the map $f'$ has the property that the preimage of any point has size at most two, and at most one point has a preimage of size two. Therefore, $2^{c-j} \geqs n-1$ and $\gamma_t(G) \geqs \log_2(n-1)+j \geqs \log_2{n}$.
\end{proof}

\begin{cor}\label{c:sn_udn}
Let $G = S_n$ with $n \geqs 5$. Then $\gamma_u(G) \geqs 3$.
\end{cor}

\begin{proof}
Since $n \geqs 5$, Proposition~\ref{p:sn_an_tdn_lower} implies that $\gamma_u(G) \geqs \gamma_t(G) \geqs \lceil \log_2{n} \rceil \geqs 3$. 
\end{proof}

We will now complete the proof of Theorem~\ref{t:symmain} when $n$ is odd.

\begin{prop}\label{p:sn_odd_udn}
Let $G = S_n$ with $n = 2m+1$ and $m \geqs 2$. Then
\[
\log_2{n} \leqs \gamma_t(G) \leqs  \gamma_u(G) = b(G,\Omega) \leqs 2\log_2{n},
\]
where $\Omega$ is the set of $m$-element subsets of $[n]$.
\end{prop}

\begin{proof}
By Proposition~\ref{p:sn_an_tdn_lower}, we have  
\[
\log_2{n} \leqs \gamma_t(G) \leqs \gamma_u(G). 
\]
Let $s \in G$ be a witness to the bound $u(G) \geqs 1$. Since $s$ must be an odd permutation, it follows that $s \in H = S_k \times S_{n-k}$ for some $1 \leqs k \leqs m$. Now \cite[Corollary~2.2]{Hal} gives $b(G,G/H) \geqs b(G,\Omega)$ and thus $\gamma_u(G) \geqs b(G,\Omega)$ by Lemma \ref{l:udn}. Now fix an element $s \in G$ of shape $[m,m+1]$. As explained in the proof of Theorem~\ref{p:sn_odd}, we have $\M(G,s) = \{H\}$ with $H = S_m \times S_{m+1}$. Therefore, $\gamma_u(G) \leqs b(G,G/H) = b(G,\Omega)$. This proves that $\gamma_u(G) = b(G,\Omega)$. Finally, by applying \cite[Theorem~4.2]{Hal}, we conclude that 
\[
b(G,\O) \leqs \log_{\left\lceil \frac{n}{m} \right\rceil}{n} \cdot \left\lceil \frac{m+1}{m} \right\rceil \leqs 2\log_2{n},
\]
and the result follows.
\end{proof}

It remains to prove Theorem~\ref{t:symmain} when $n$ is even.

\begin{prop}\label{p:sn_even_udn}
Let $G=S_n$ with $n \geqs 6$ even. Then $\gamma_u(G) \leqs 3n\log_2{n}$.
\end{prop}

\begin{proof}
Let $s=(1,2,\dots,n)$ and $x \in \mathcal{P}$, where $\mathcal{P}$ is the set of elements of prime order in $G$. If $x \not\in (1,2)^G$, then $P(x,s) < 0.67$ by the proof of Theorem~\ref{t:sn_even_us} (see p.\pageref{page:bounds}). Therefore,
\[
\sum_{x \in \mathcal{P} \setminus (1,2)^G}^{} P(x,s)^{2n\log_2{n}} < n! \cdot 0.67^{-n\log_{0.67}{n}} = n!/n^n < 1.
\]
Consequently, there exists a subset $A \subseteq s^G$ such that $|A| \leqs 2n\log_2{n}$ and for all $x \in \mathcal{P}\setminus(1,2)^G$ there exists $z \in A$ with $G=\<x,z\>$. 

Let $1 \leqs k < n$. If $(k,n) = 1$, then let 
\[
B_k = \{ (1,1+k,1+2k,\dots,1+(n-1)k) \},
\]
where addition is carried out modulo $n$. If $(k,n) = d > 1$, then write $l=n/d-1$ and let 
\[
B_k = \{ b_k, c_k \}
\]
where
\begin{gather*}
b_k = (1,1+k,1+2k,\dots,1+lk,2,2+k,2+2k,\dots,2+lk,\dots,d,d+k,d+2k,\dots,d+lk) \\
c_k = (1+k,1+2k,\dots,1+lk,1,2+k,2+2k,\dots,2+lk,2,\dots,d+k,d+2k,\dots,d+lk,d).
\end{gather*}
Set $B = \bigcup_{k=1}^{n-1} B_k$ and note that $|B| \leqs 2(n-1)$. Let $x = (i,j) \in (1,2)^G$ with $k=j-i>0$. Then there exists $z \in B_k$ such that $iz=j$, which implies that $G = \<x,z\>$.

We conclude that for all $x \in \mathcal{P}$, there exists $z \in A \cup B$ such that $G=\<x,z\>$. Moreover, 
\[
|A \cup B| \leqs 2n\log_2{n} + 2(n-1) \leqs 3n\log_2{n}
\]
and the proof of the proposition is complete.
\end{proof}

\begin{rem}\label{r:sn_even_udn}
By combining Propositions~\ref{p:sn_an_tdn_lower} and~\ref{p:sn_even_udn}, we deduce that 
\[
\log_2{n} \leqs \gamma_u(S_n) \leqs 3n\log_2{n}
\] 
when $n$ is even. It would be interesting to see if it is possible to close the gap between the lower and upper bounds on $\gamma_u(S_n)$ in this case.
\end{rem}

\section{Alternating groups}\label{s:alt}

In this section we prove Theorem \ref{t:altmain}. We start by recording the spread and uniform spread of even-degree alternating groups, which were determined by Brenner and Wiegold in \cite[(3.01)--(3.05)]{BW75}. If $n \geqs 4$ is even, then  
\[
s(A_n) = u(A_n) =
\left\{
\begin{array}{ll}
2 & \text{if $n=6$} \\
4 & \text{otherwise}
\end{array}
\right.
\]

The situation for odd-degree alternating groups is more complicated. With the aid of {\sc Magma}, one can check that 
\[
s(A_5) = u(A_5) = 2, \;\; s(A_7) = u(A_7) = 3.
\] 
Now assume $G = A_n$ with $n \geqs 9$ odd. By \cite[Proposition~6.7]{BGK}, we have $u(G) \geqs 3$ and by making a minor modification to the proof of this result, we can establish the following slightly stronger bound.

\begin{prop}\label{p:an_odd}
Let $G = A_n$ with $n \geqs 9$ odd. Then $u(G) \geqs 4$.
\end{prop}

\begin{proof}
If $n \leqs 29$, then the result can be verified computationally (see \cite[Table~6]{BGK}). Now assume $n \geqs 31$ and let $s \in G$ be an $n$-cycle. As in the proof of \cite[Proposition~6.7]{BGK}, we have  
\[
\M(G,s) = \mathcal{I} \cup \mathcal{P}_1 \cup \mathcal{P}_2,
\]
where $\mathcal{I}$ consists of exactly one imprimitive subgroup $(S_{n/l} \wr S_l) \cap G$ for each divisor $l$ of $n$ with $1<l<n$, $\mathcal{P}_1$ contains at most $(n-1)/d$ subgroups isomorphic to ${\rm P}\Gamma{\rm L}_d(q) \cap G$ for each pair $(q,d)$ with $n=(q^d-1)/(q-1)$, and $\mathcal{P}_2 = \{ N_G(\<s\>)\}$ if $n$ is prime (otherwise $\mathcal{P}_2$ is empty).

Suppose $x \in G$ has prime order. As in the proof of \cite[Proposition~6.7]{BGK}, we have
\begin{gather*}
\sum_{H \in \mathcal{P}_1} \fpr(x,G/H) \leqs \frac{n^{\log_2{n}+3}4\sqrt{\pi}}{n^{n/4} \left(\frac{8}{e}\right)^{n/4}} < 10^{-6} \\
\sum_{H \in \mathcal{P}_2} \fpr(x,G/H) \leqs 2 \left(\frac{4}{n+1}\right)^{(n-3)/2} < 10^{-8}.
\end{gather*}
The action of $G$ on the set of cosets of $H=(S_{n/l} \wr S_l) \cap G$ is equivalent to the action on $\Pi_l$, so by applying Lemmas~\ref{l:fprs_odd} and \ref{l:fprs_even}, noting that $x$ is even and $2 < l < \frac{n}{2}$ since $n$ is odd, we see that $\fpr(x,\Pi_l) < \frac{1}{l^{2}}$. Therefore,  
\[
\sum_{H \in \mathcal{I}} \fpr(x,G/H) \leqs \sum_{\substack{l \mid n \\ 1 < l < n}} \frac{1}{l^2} < \sum_{m=1}^{\infty} \frac{1}{(2m+1)^2} = \frac{\pi^2}{8} - 1
\]
and thus
\[
\sum_{H \in \M(G,s)} \fpr(x,G/H) < \frac{\pi^2}{8} - 1 + 10^{-6} + 10^{-8} < \frac{1}{4}.
\]
By applying Corollary \ref{c:ug}, we conclude that $u(G) \geqs 4$. 
\end{proof}

\begin{rem}
As noted in Remark~\ref{r:altmain}(c), the uniform spread of odd-degree alternating groups can be arbitrarily large. Indeed, \cite[Theorem~1.1]{GSh} states that if $(n_i)$ is a sequence of natural numbers tending to infinity, then $u(A_{n_i})$ tends to infinity if and only if the least prime divisor of $n_i$ tends to infinity (see Theorem~\ref{t:star}(ii) and \eqref{e:pan}).
\end{rem}

We now turn to uniform domination.

\begin{prop}\label{p:an_udn}
Let $G=A_n$ with $n \geqs 5$ and let $p$ be the smallest prime divisor of $n$. Then
\[
\log_p{n} \leqs \gamma_t(G) \leqs \gamma_u(G) \leqs c\log_2{n}
\]
where $c=2$ if $n$ is even and $c=77$ if $n$ is odd.
\end{prop}

\begin{proof}
This is a combination of Proposition~\ref{p:sn_an_tdn_lower}(ii) and \cite[Theorem~2]{BH}.
\end{proof}

\begin{rem}
Let us consider the upper bound in Proposition \ref{p:an_udn}. Let $G = A_n$, where 
$n = 2m \geqs 8$ is even, and let $k$ be the greatest odd integer strictly less than $m$. For $1 < \ell < n$, let $\Sigma_{\ell}$ be the set of $\ell$-element subsets of $[n]$. To establish the upper bound on $\gamma_u(G)$ in the proposition, we show that if $s \in G$ has shape $[k,n-k]$ then $\M(G,s) = \{H\}$ with $H = (S_k \times S_{n-k}) \cap G$ and thus $\gamma_u(G) \leqs b(G,\Sigma_k)$. This is very similar to Proposition~\ref{p:sn_odd_udn} for odd-degree symmetric groups, where we showed that $\gamma_u(S_{2m+1}) = b(S_{2m+1},\Omega)$ for the set $\O$ of $m$-element subsets of $[2m+1]$. Therefore, it is natural to ask if equality holds, that is, do we have  
$\gamma_u(G) = b(G,\Sigma_k)$?
 
Let $S \subseteq s^{G}$ be a TDS for $G$ and note that $s$ acts intransitively on $[n]$ since $n$ is even. By \cite[Corollary~2.2(2)]{Hal}, $b(G,\Sigma_i) \geqs b(G,\Sigma_j)$ if $i \leqs j \leqs m$. Therefore, we can conclude that $|S| \geqs b(G,\Sigma_k)$ unless $s$ has shape $[m,m]$, or $m$ is odd and $s$ has shape $[m+1,m-1]$. In particular, if $\gamma_u(G) < b(G,\Sigma_k)$, then this has to be witnessed by a conjugacy class $s^G$ of elements of shape $[m,m]$, or of shape $[m+1,m-1]$ if $m$ is odd. In these two cases, $s$ could be contained in several imprimitive maximal overgroups and since the base size $b(G,\Sigma_k)$ is not known exactly, it is difficult to use our probabilistic method to determine if $s^G$ does indeed witness $\gamma_u(G) < b(G,\Sigma_k)$.
\end{rem}

\begin{prop}\label{p:an_udn_2}
Let $G=A_n$ with $n \geqs 5$. 
\begin{itemize}\addtolength{\itemsep}{0.2\baselineskip}
\item[{\rm (i)}]  $\gamma_u(G)=2$ if and only if $n \geqs 13$ is a prime number. 
\item[{\rm (ii)}] If $\gamma_u(G)=2$, then $P(G,s,2)>0$ if and only if $s$ is an $n$-cycle.
\end{itemize}
\end{prop}

\begin{proof}
If $\gamma_u(G)=2$, then Proposition~\ref{p:sn_an_tdn_lower}(iii) implies that $n$ is prime and thus part (i) follows via \cite[Proposition 3.8 and Remark 3.12]{BH}. For (ii), if $s \in G$ is not an $n$-cycle, then $s$ is contained in a maximal intransitive subgroup $H$ and thus Lemma \ref{l:udn} implies that $P(G,s,2)=0$ since $b(G,G/H) \geqs 3$. 
\end{proof}

\begin{rem}
The proof of Proposition~\ref{p:an_udn_2} shows that if $G = A_n$ with $n \geqs 13$, then $\gamma_t(G)=2$ if and only if $n$ is prime. Moreover, the only possible witnesses are two (not necessarily conjugate) $n$-cycles.
\end{rem}

In order to complete the proof of Theorem \ref{t:altmain}, it remains to consider the probability $P_2(A_n)$ when $n \geqs 13$ is a prime. Set
\[
\mathcal{H} = \left\{ n \in \mathbb{N}\,:\, \mbox{$n =\frac{q^d-1}{q-1}$ for some prime power $q$ and integer $d \geqs 2$} \right\}
\]
and observe that $3, 5, 7, 13, 17, 31, 73, 127, 257, 307$ are the ten smallest primes in $\mathcal{H}$. 

\begin{prop}\label{p:an_p2g_1}
Let $G = A_n$, where $n \geqs 13$ is a prime with $n \not\in \mathcal{H}$. Then 
\[
P_2(G) > 1 - n^{-2}.
\]
\end{prop}

\begin{proof}
By Proposition \ref{p:an_udn_2}, we have $P_2(G) = P(G,s,2)$ where $s \in G$ is an $n$-cycle. To begin with, let us assume $n \ne 23$. Then $\mathcal{M}(G,s) = \{H\}$, where $H = N_G(\la s \ra) = C_n{:}C_{(n-1)/2}$, and we have 
$P_2(G) = 1- Q(G,s,2)$. Moreover, Lemma \ref{l:key} gives
\[
Q(G,s,2) \leqs \sum_{i=1}^{k}|x_i^G|\cdot {\rm fpr}(x_i,G/H)^2 = \what{Q}(G,s,2),
\]
where the $x_i$ represent the conjugacy classes in $G$ of elements of prime order.

Let $x \in H$ be an element of prime order $t$. If $t=n$ then $|x^G \cap H| = (n-1)/2=a$ and $|x^G| = \frac{1}{2}(n-1)!=b$. Now assume $t$ divides $(n-1)/2$. Here $x$ has a unique fixed point on $[n]$, so 
\[
|x^G \cap H| = i_t(H) = n(t-1)=a_t, \quad |x^G| = \frac{n!}{((n-1)/t)!t^{(n-1)/t}} = b_t
\]
and thus
\[
\what{Q}(G,s,2) = a^2/b + \sum_{t \in \pi}a_t^2/b_t,
\]
where $\pi$ is the set of prime divisors of $(n-1)/2$. Now $|\pi| \leqs \log_2 ((n-1)/2)$ and one checks that $a_t^2/b_t \leqs a_2^2/b_2$ for all $t \in \pi$, so 
\[
P_2(G) \geqs 1 - \frac{n-1}{2(n-2)!}- \log_2 ((n-1)/2) \cdot \frac{n^2((n-1)/2)!2^{(n-1)/2}}{n!}
\]
and the desired bound follows.

Finally, let us assume that $n=23$. Here $\mathcal{M}(G,s)=\{H,K\}$ with $H \cong K \cong {\rm M}_{23}$ and with the aid of {\sc Magma} we calculate that 
\[
\what{Q}(G,s,2) = 4\sum_{i=1}^{k}|x_i^G|\cdot {\rm fpr}(x_i,G/H)^2 = \frac{27704}{178562475}.
\]
The result follows.
\end{proof}

\begin{rem}
In the previous proposition, it is easy to compute $P_2(G)$ precisely when $m=(n-1)/2$ is a prime and $n \ne 23$. As before, let $s \in G$ be an $n$-cycle and write $\M(G,s) = \{H\}$. Let $r$ be the number of regular orbits of $H$ on $G/H$. By arguing as in the proof of \cite[Proposition 3.2]{BGiu} we calculate that 
\[
r = \frac{|G:H| - n(m^2-m-1)-1}{|H|}
\]
and thus 
\[
P_2(G) = 1 - \frac{n^3-4n^2-n+4}{4(n-2)!}
\]
by Lemma \ref{l:udn2}.
\end{rem}

\begin{prop}\label{p:an_p2g_2}
Let $G = A_n$, where $n > 13$ is a prime. Then 
\[
P_2(G) > 1 - n^{-1}.
\]
\end{prop}

\begin{proof}
In view of the previous proposition, we may assume $n \in \mathcal{H}$. Let $s \in G$ be an $n$-cycle, so $P_2(G) = P(G,s,2)$. If $n \geqs 73$, then by arguing as in the proof of \cite[Proposition 3.8]{BH}, we obtain
\[
P_2(G) > 1 - \left(\frac{2\log_2 n}{n-1}\right)^2
\]
and the result follows if $n>257$. For $n \in \{73,127,257\}$, the desired bound is easily obtained by inspecting the proof of \cite[Proposition 3.8]{BH}. For example, the proof  gives
\[
P_2(G) > 1 - \left(\frac{\ell}{C}\right)^2,
\]
where $\ell = 1 + (n-1)\sum_{(q,d)}\frac{1}{d}$ and $C = n(n-1)/2$ (the sum in the expression for $\ell$ is over all prime powers $q$ and integers $d \geqs 2$ with $n= (q^d-1)/(q-1)$). For $n = 73$ we have $n = (8^3-1)/(8-1)$ only, so $\ell=25$, $C = 2628$ and the desired bound quickly follows. The cases when $n$ is $127$ or $257$ are just as straightforward.

Finally, if $n=17$ then $\M(G,s) = \{H,K\}$, where $H$ and $K$ are non-conjugate maximal subgroups isomorphic to ${\rm P\Gamma L}_{2}(16)$. Here we calculate
\[
\what{Q}(G,s,2) = \frac{335848}{42567525} < \frac{1}{17}
\]
and the result follows. Similarly, if $n=31$ then $\what{Q}(G,s,2)< \frac{1}{31}$.
\end{proof}

\begin{rem}\label{r:a13}
The case $G = A_{13}$ requires special attention. As above, we have $P_2(G) = P(G,s,2)$ with $s \in G$ a $13$-cycle and we observe that  
\[
\M(G,s) = \{H,K_1,K_2,L_1,L_2\},
\]
where $H = N_G(\la s \ra) = C_{13}{:}C_6$ and each of the remaining subgroups are isomorphic to ${\rm P\Gamma L}_{3}(3)$, with $K_1,K_2$ conjugate, and $L_1,L_2$ conjugate. One checks that 
\[
\what{Q}(G,s,2) = \frac{4230997}{1108800} >1,
\]
so the probabilistic approach does not yield $P_2(G)>0$. However, as noted in the proof of \cite[Proposition 3.8]{BH}, by randomly choosing conjugates of a fixed $13$-cycle, we can identify a TDS for $G$. For example, 
\[
\{ (1,2,3,4,5,6,7,8,9,10,11,12,13), (1,2,3,4,5,6,8,9,12,7,11,10,13) \}
\]
has the desired property (see \cite[Section 1.2.4]{BH_comp} for further details). This shows that $P_2(G)>0$, but further work is needed to compute this probability precisely. To do this, we use the fact that $\{s,s^g\}$ is a TDS for $G$ if and only if $A \cap B = 1$ for all $A \in \M(G,s)$ and all $B \in \M(G,s^g)$ (see \cite[Lemma 2.1]{BH}). In this way, we can use \textsc{Magma} to show that
\[
P_2(G) = \frac{|\{s^g \in s^G \,:\, \mbox{$\{s,s^g\}$ is a TDS for $G$}\}|}{|s^G|} = \frac{4979}{46200}
\]
and we conclude that $A_{13}$ is a genuine exception to the bound in Proposition \ref{p:an_p2g_2}.
\end{rem}

\section{Exceptional groups of Lie type}\label{s:excep}

In this section, we assume $G$ is a finite simple exceptional group of Lie type over $\mathbb{F}_q$. Our aim is to prove Theorem \ref{t:exmain}. Note that Theorem \ref{t:star} (see the introduction) implies that $u(G) \geqs 3$, and $u(G) \to \infty$ as $q \to \infty$.

\begin{rem}
By applying Corollary \ref{c:ug}, it is possible to determine explicit lower bounds on $u(G)$ in terms of $q$. For example, suppose $G = {}^2B_2(q)$ with $q=2^{2m+1}$ and $m \geqs 1$. Let $s \in G$ be an element of order $q-\sqrt{2q}+1$. By inspecting \cite[Theorem 9]{Suz}, which lists the maximal subgroups of $G$, one can show that $\M(G,s) = \{H\}$ with 
$H = N_G(\la s \ra) = C_{q-\sqrt{2q}+1}{:}C_4$ (see the proof of Lemma \ref{l:suz}). Let $x \in G$ be an element of prime order $r$. Since $G$ contains a unique conjugacy class of involutions, it follows that 
\[
{\rm fpr}(x,G/H) = \frac{i_2(H)}{i_2(G)} = \frac{q-\sqrt{2q}+1}{(q^2+1)(q-1)} = \frac{1}{(q+\sqrt{2q}+1)(q-1)}
\]
if $r=2$. Similarly, if $r$ is odd then we may assume $r$ divides $q-\sqrt{2q}+1$ (otherwise ${\rm fpr}(x,G/H)=0$), so $|C_G(x)| = q-\sqrt{2q}+1$ and we deduce that
\[
{\rm fpr}(x,G/H) \leqs \frac{r-1}{|x^G|} \leqs \frac{(q-\sqrt{2q}+1)(q-\sqrt{2q})}{q^2(q-1)(q^2+1)}<\frac{1}{(q+\sqrt{2q}+1)(q-1)}.
\]
Therefore, Corollary \ref{c:ug} implies that
\[
u(G) \geqs (q+\sqrt{2q}+1)(q-1) - 1.
\] 
\end{rem}

Let us now turn to the uniform domination number. In \cite[Theorem 5.2]{BH}, we established the bound $\gamma_u(G) \leqs 6$ for every finite simple exceptional group $G$. By applying recent work of the first author in \cite{Bur18}, we can prove a stronger result.

\begin{thm}\label{t:udn5}
If $G$ is a finite simple exceptional group of Lie type, then $\gamma_u(G) \leqs 5$. 
\end{thm}

\begin{proof}
First assume there exists an element $s \in G$ with $\M(G,s) = \{H\}$ for some maximal subgroup $H$ of $G$. Then Lemma \ref{l:udn} implies that $\gamma_u(G) = b(G,H)$ and by applying the main theorem of \cite{Bur18}, noting that $H$ is non-parabolic, it follows that $\gamma_u(G) \leqs 5$.

By arguing as in the proof of \cite[Theorem 5.2]{BH}, it remains to consider the cases
\begin{itemize}\addtolength{\itemsep}{0.2\baselineskip}
\item[(a)] $G = F_4(q)$, $q=2^a$, $a \geqs 2$; and
\item[(b)] $G = G_2(q)$, $q=3^a$, $a \geqs 2$.
\end{itemize}

First consider (a). As noted in the proof of \cite[Theorem 5.2]{BH}, there exists an element $s \in G$ with $\M(G,s) = \{H,K\}$ and $H \cong K \cong {}^3D_4(q).3$. It suffices to show that 
\begin{equation}\label{e:wh}
\what{Q}(G,s,5) := 32\sum_{i=1}^{k}|x_i^G|\cdot \fpr(x_i,G/H)^5 < 1,
\end{equation}
where $x_1, \ldots, x_k$ are representatives of the conjugacy classes in $G$ of elements of prime order. Let $\bar{G} = F_4$ and $\bar{H} = D_4$ be the corresponding algebraic groups defined over the algebraic closure of $\mathbb{F}_q$. Let $V_{26}$ be one of the $26$-dimensional irreducible modules for $\bar{G}$ and note that
\begin{equation}\label{e:v26}
V_{26}\downarrow \bar{H} = V(\l_1) \oplus V(\l_3) \oplus V(\l_4) \oplus 0^2,
\end{equation}
where $V(\l_1)$ is the natural module for $\bar{H}$, $V(\l_3)$ and $V(\l_4)$ are the two spin modules and $0$ is the trivial module (see \cite[Chapter 12]{Th}, for example). Let $x \in H$ be an element of prime order $r$.

First assume $r=2$, so $x \in {}^3D_4(q)$. There are two classes of involutions in ${}^3D_4(q)$, labelled $A_1$ and $A_1^3$ in the notation of \cite{Spal}. As elements of $\bar{H}$,  the involutions in the first class are of type $a_2$ and those in the second are of type $c_4$ (in the notation of \cite{AS}). By considering the decomposition in \eqref{e:v26}, we deduce that if $x$ is in the $A_1$ class of ${}^3D_4(q)$, then it has Jordan form $[J_2^6,J_1^{14}]$ on $V_{26}$, and similarly $[J_2^{12},J_1^2]$ if it is in the class labelled $A_1^3$. By inspecting \cite[Table 3]{Law}, we conclude that the involutions in the $A_1$ class of ${}^3D_4(q)$ are in the $G$-class labelled 
$A_1$, and the others are in $G$-class $A_1\tilde{A}_1$. The relevant class sizes in $G$ are given in \cite[Table 22.2.4]{LS_book} and we deduce that the contribution to 
$\what{Q}(G,s,5)$ from involutions is precisely $32(a_1b_1^5 + a_2b_2^5)$, where
\[
a_1 = (q^4+1)(q^{12}-1), \;\; b_1 = \frac{1}{q^6+q^4+q^2+1}\]
and
\[
a_2 = q^4(q^4+q^2+1)(q^8-1)(q^{12}-1),\;\; b_2 = \frac{1}{q^2(q^2+1)(q^8-1)}.
\]

Now assume $r>2$, so $x$ is semisimple. By arguing as in the proof of \cite[Lemma 3.17]{Bur18}, we deduce that the combined contribution to $\what{Q}(G,s,5)$ from the elements with $C_{\bar{G}}(x)^0 = B_3T_1$ or $C_3T_1$ is at most
\[
32\cdot 2q^{15}(q^4+1)(q^{12}-1)\cdot \left(\frac{48(q+1)}{q^{9}(q-1)^4}\right)^5<q^{-7}.
\]
Similarly, the contribution from regular semisimple elements is less than $32q^{-33}$. 
For all other semisimple elements, the bound in \cite[(17)]{Bur18} gives ${\rm fpr}(x,G/H)<q^{-11}$ and it follows that the remaining contribution is less than $32q^{50}(q^{-11})^5 = 32q^{-5}$. In conclusion, 
\[
\what{Q}(G,s,5) < 32(a_1b_1^5 + a_2b_2^5) + q^{-7} + 32q^{-33} + 32q^{-5} < 1
\]
and thus $\gamma_u(G) \leqs 5$.

Finally, let us consider case (b). Here there is an element $s \in G$ such that $\M(G,s) = \{H,K\}$ with $H \cong K \cong {\rm SU}_{3}(q).2$ and so it suffices to show that \eqref{e:wh} holds. Let $x \in H$ be an element of prime order $r$. Let $\bar{G} = G_2$ and $\bar{H} = A_2$ be the corresponding algebraic groups over the algebraic closure of $\mathbb{F}_{q}$.

First assume $r=3$, so $x \in {\rm SU}_{3}(q)$. Let $V_7$ be one of the $7$-dimensional irreducible modules for $\bar{G}$ and note that  
\[
V_{7}\downarrow \bar{H} = V_3 \oplus V_3^* \oplus 0,
\]
where $V_3$ is the natural module for $\bar{H}$ (and $V_{3}^*$ is its dual). If $x \in H$ has Jordan form $[J_2,J_1]$ on $V_3$, then the above decomposition implies that $[J_2^2,J_1^3]$ is the Jordan form of $x$ on $V_7$. Similarly, the regular unipotent elements in $H$ have Jordan form $[J_3^2,J_1]$. By inspecting \cite[Table 1]{Law}, we deduce that $x$ belongs to the $G$-classes labelled $A_1$ and $G_2(a_1)$ in the two respective cases, whence the contribution to $\what{Q}(G,s,5)$ from elements of order $3$ is precisely $32(a_1b_1^5+a_2b_2^5)$, where
\[
a_1 = q^6-1,\; b_1 = \frac{1}{q^2+q+1},\; a_2 = \frac{1}{2}q^2(q^2-1)(q^6-1),\; b_2 = \frac{2}{q(q^3-1)}.
\]

Next assume $r=2$. Now $G$ contains $a_3=q^4(q^4+q^2+1)$ involutions, which form a single conjugacy class, and thus 
\[
|x^G \cap H| = i_2(H) = \frac{|{\rm GU}_{3}(q)|}{|{\rm GU}_{2}(q)||{\rm GU}_{1}(q)|}+\frac{|{\rm SU}_{3}(q)|}{|{\rm SO}_{3}(q)|} = q^2(q^2-q+1)(q+2).
\] 
Therefore,
\[
{\rm fpr}(x,G/H) = \frac{q^2(q^2-q+1)(q+2)}{q^4(q^4+q^2+1)} = \frac{(q^2-q+1)(q+2)}{q^2(q^4+q^2+1)} = b_3
\]
and $32a_3b_3^5$ is the contribution from involutions. Finally, if $r \geqs 5$ then the proof of \cite[Lemma 4.31]{BLS} gives
\[
{\rm fpr}(x,G/H) < \frac{4(q+1)^2}{(q-1)^2}\cdot q^{-4} \leqs \frac{25}{4}q^{-4}
\]
and thus the total contribution to $\what{Q}(G,s,5)$ from elements of order at least $5$ is less than
\[
32q^{14}\cdot \left(\frac{25}{4}q^{-4}\right)^5 < 6q^{-1}.
\]
We conclude that
\[
\what{Q}(G,s,5) < 32\sum_{i=1}^{3}a_ib_i^5 + 6q^{-1}<1.
\]
The result follows.
\end{proof}

By \cite[Theorem 5.2]{BH}, we have $\gamma_u(G) = 2$ if $G$ is one of ${}^2B_2(q)$, ${}^2G_2(q)$ (with $q \geqs 27$) or $E_8(q)$. The following theorem, which is the main result of this section, completely determines the simple exceptional groups $G$ with $\gamma_u(G)=2$.

\begin{thm}\label{t:exceptional_udn}
Let $G$ be a finite simple exceptional group of Lie type over $\mathbb{F}_q$. Then $\gamma_u(G)=2$ if and only if 
\[
G \in \{ {}^2B_2(q), \, {}^2G_2(q) \, (q \geqs 27), \, {}^2F_4(q) \, (q \geqs 8), \, {}^3D_4(q), \, E_6^{\e}(q), \, E_7(q), \, E_8(q)\}.
\]
Moreover, if $\gamma_u(G)=2$ then $P_2(G) > 1-q^{-1}$.
\end{thm}

In order to prove Theorem \ref{t:exceptional_udn}, we need to record some preliminary results.

\begin{lem}\label{l:par}
Let $G$ be a finite simple group of Lie type and let $\mathcal{P}$ be the union of the parabolic subgroups of $G$. Then $G \setminus \mathcal{P}$ is the set of regular semisimple elements that are not contained in a maximal torus of a Levi factor of a parabolic subgroup of $G$.
\end{lem}

\begin{proof}
If $u \in G$ is a nontrivial unipotent element, then $C_G(u)$ is contained in a parabolic subgroup of $G$. Therefore, each $x \in G \setminus \mathcal{P}$ is a regular semisimple element and the result follows.
\end{proof}

Recall that if $G$ is a finite simple exceptional group and $H$ is a maximal subgroup of $G$, then $b(G,G/H) \leqs 6$ by the main theorem of \cite{BLS}. The next two lemmas, which may be of independent interest, show that $b(G,G/H)=2$ in two special cases. These results will play an essential role in the proof of Theorem \ref{t:exceptional_udn} when $G = E_7(q)$ or $E_6^{\e}(q)$. 

\begin{nota}
In Lemmas \ref{l:e7} and \ref{l:e6}, we use the notation $P(G,H,2)$ to denote the probability that $H \cap H^g = 1$ for a randomly chosen conjugate $H^g$. Equivalently, this is the probability that two randomly chosen points in $G/H$ form a base for $G$, with respect to the natural action of $G$ on $G/H$. In particular, $b(G,G/H) = 2$ if and only if $P(G,H,2)>0$. Moreover, if $Q(G,H,2) = 1-P(G,H,2)$ denotes the complementary probability, then
\begin{equation}\label{e:qhat}
Q(G,H,2) \leqs \sum_{i=1}^{k}|x_i^G|\cdot \fpr(x_i,G/H)^2 =:\what{Q}(G,H,2),
\end{equation}
where $x_1, \ldots, x_k$ represent the conjugacy classes of elements of prime order in $G$ (see the proof of \cite[Theorem 1.3]{LSh99}, for example).
\end{nota}

\begin{lem}\label{l:e7}
If $G = E_7(q)$ and $H = ({\rm L}_{2}(q^3) \times {}^3D_4(q)).3$ then $P(G,H,2) > 1-q^{-2}$.
\end{lem}

\begin{proof}
It suffices to show that $\what{Q}(G,H,2)<q^{-2}$. Let $\bar{G} = E_7$ and $\bar{H} = A_1^3D_4.S_3$ be the corresponding ambient algebraic groups over the algebraic closure of $\mathbb{F}_q$. Let $\mathcal{L}(\bar{G})$ be the Lie algebra of $\bar{G}$ and let $V_{56}$ be the $56$-dimensional irreducible module for $\bar{G}$. Fix a set of fundamental dominant weights $\{\l_1, \l_2,\l_3,\l_4\}$ for the $D_4$ factor of $\bar{H}$. It will be useful to record the restrictions  
\begin{align}
\begin{split}
\mathcal{L}(\bar{G})\downarrow A_1^3D_4 = & \; \mathcal{L}(A_1^3D_4) \oplus (V_2 \otimes V_2 \otimes 0 \otimes V(\l_4)) \oplus (V_2 \otimes 0 \otimes V_2 \otimes V(\l_3)) \label{e:vres} \\
&  \oplus  (0 \otimes V_2 \otimes V_2  \otimes V(\l_1))
\end{split} \\
\begin{split}
V_{56} \downarrow A_1^3D_4 = & \; (V_2 \otimes V_2 \otimes V_2 \otimes 0) \oplus (V_2 \otimes 0 \otimes 0 \otimes V(\l_1)) \oplus (0 \otimes V_2 \otimes 0 \otimes V(\l_3)) \label{e:vres2} \\
& \oplus (0 \otimes 0 \otimes V_2 \otimes V(\l_4)).
\end{split}
\end{align}
(see \cite[Chapter 12]{Th}, for example). Here $V_2$ and $0$ denote the natural and trivial modules for $A_1$, respectively, and we write $V(\l_i)$, with $i=1,3,4$, for the $8$-dimensional irreducible module for $D_4$ with highest weight $\l_i$ (so $V(\l_1)$ is the natural module, and $V(\l_3)$, $V(\l_4)$ are the two spin modules). 

Set $H = H_0.\la \tau \ra = H_0.3$, where $H_0 = {\rm L}_{2}(q^3) \times {}^3D_4(q)$ and $\tau$ induces a  field automorphism $\varphi$ on ${\rm L}_{2}(q^3)$ and a graph automorphism on ${}^3D_4(q)$. If we view $H_0$ as a subgroup of the connected component $\bar{H}^0$, then we may assume that 
\[
H_0 = \{ (x,x^{\varphi},x^{\varphi^2},y) \,:\, x \in {\rm L}_{2}(q^3),\, y \in {}^3D_4(q)\} < \bar{H}^0.
\]
In particular, note that $\tau$ cyclically permutes the three $A_1$ factors of $\bar{H}^0$. Write 
\begin{equation}\label{e:ab}
\what{Q}(G,H,2) = \a+\b,
\end{equation}
where $\a$ and $\b$ denote the  contributions from unipotent and semisimple elements, respectively. We refer the reader to \cite[Table 22.2.2]{LS_book} and \cite[Section 3]{FJ} for detailed 
information on the unipotent and semisimple conjugacy classes in $G$ (including centralizer orders). In particular, we will adopt the notation from \cite{LS_book} for labelling the unipotent classes in $\bar{G}$, which is consistent with the standard Bala-Carter notation.

Let $x \in H$ be an element of prime order $r$ and write $q=p^f$ with $p$ a prime. First assume $r=p=2$. If $x$ is a long root element, then \cite[Proposition 1.13(iii)]{LLS0} implies that $x^G \cap H$ is the set of long root elements in the ${}^3D_4(q)$ factor of $H_0$, so $|x^G \cap H|<q^{10}$, $|x^G|>\frac{1}{2}q^{34}$ and the contribution to $\a$ is less than $2q^{-14}$. In the remaining cases, we have $|x^G|>\frac{1}{2}q^{52}$ and we note that 
\[
i_2(H) = (i_2({\rm L}_{2}(q^3))+1)(i_2({}^3D_4(q))+1)-1 < q^{21},
\]
so the contribution to $\a$ from these elements is less than $\frac{1}{2}q^{52}(2q^{-31})^2 = 2q^{-10}$. Therefore, $\a < 2q^{-14}+2q^{-10}$ when $p=2$. 

Next assume $r=p \geqs 3$. If $x \in H \setminus H_0$ then $r=3$ and $x$ belongs to one of the $G$-classes labelled $A_2^2$ or $A_2^2A_1$ (see the proof of \cite[Proposition 5.12]{BGS2}). Therefore, 
$|x^G| > \frac{1}{2}q^{84}$ and thus the contribution to $\a$ is less than $\frac{1}{2}q^{84}(2q^{-45})^2 = 2q^{-6}$ since $|H|<q^{39}$. Now assume $x^G \cap H \subseteq H_0$. If $\dim x^{\bar{G}} \geqs 64$ then $|x^G|>\frac{1}{2}q^{64}$ and thus the contribution to $\a$ from these elements is less than $\frac{1}{2}q^{64}(2q^{-34})^2 = 2q^{-4}$ since $H_0$ contains fewer than  
$q^6\cdot q^{24} = q^{30}$ elements of order $p$. Now assume $\dim x^{\bar{G}} < 64$, in which case $x$ belongs to one of the $\bar{G}$-classes labelled $A_1$, $A_1^2$ or $(A_1^3)^{(1)}$. As above, the contribution from long root elements is less than $2q^{-14}$. More generally, we can use \eqref{e:vres2} to calculate the Jordan form of each unipotent element $x \in H_0$ on $V_{56}$, which in turn allows us to determine the $\bar{G}$-class of $x$ by inspecting \cite[Table 7]{Law}. 

For example, suppose $x = yz \in H_0$, where $y \in {\rm L}_{2}(q^3)$ is a long root element and $z \in {}^3D_4(q)$ is in the class labelled $A_2'$ in \cite{Spal}. In terms of \eqref{e:vres2}, $y$ has Jordan form $[J_2]$ on $V_2$ and $z$ has Jordan form $[J_3^2,J_1^2]$ on each $V(\l_i)$. Therefore, $x$ has Jordan form
\[
(J_2 \otimes  J_2 \otimes J_2) \oplus (J_2 \otimes [J_3^2,J_1^2])^3 = \left\{\begin{array}{ll}
\mbox{$[J_4^7,J_2^{14}]$} & \mbox{if $p \geqs 5$} \\
\mbox{$[J_3^{14},J_2^7]$} & \mbox{if $p=3$}
\end{array}\right.
\]
on $V_{56}$ and by inspecting \cite[Table 7]{Law} we conclude that $x$ is in the $\bar{G}$-class  $A_2A_1^3$. 

In this way, we see that there are no unipotent elements in $H_0$ that belong to the $\bar{G}$-class $A_1^2$. Similarly, the elements in the class $(A_1^3)^{(1)}$ correspond to long root elements in the ${\rm L}_{2}(q^3)$ factor. Therefore, if $x$ is in $(A_1^3)^{(1)}$ then $|x^G|>\frac{1}{2}q^{54}$, $|x^G \cap H|< q^6$ and so the contribution from these elements is less than $\frac{1}{2}q^{54}(2q^{-48})^2 = 2q^{-42}$. To summarise, we have  
\[
\a< 2q^{-6} + 2q^{-4} + 2q^{-14} + 2q^{-42}
\]
for all $q \geqs 2$.

Now let us turn to $\b$. If $r=2$ then $|x^G|>\frac{1}{2}(q+1)^{-1}q^{55}$ and we calculate that
\[
i_2(H)  = (i_2({\rm L}_{2}(q^3))+1)(i_2({}^3D_4(q))+1)-1 < (q^6+1)(q^8(q^8+q^4+1)+1) < q^{23},
\]
which means that the contribution to $\b$ from involutions is less than $4q^{-8}$. Now assume $r \geqs 3$. If $\dim x^{\bar{G}} \geqs 84$ then $|x^G|>\frac{1}{2}(q+1)^{-1}q^{85}$ and since $|H|<q^{39}$ we see that the contribution to $\b$ from these elements is less than $4q^{-6}$. 

Finally, let us assume $r \geqs 3$ and $\dim x^{\bar{G}}<84$, so $C_{\bar{G}}(x)^0 = D_6T_1$ or $E_6T_1$. Set $V = \mathcal{L}(\bar{G})$ and note that $\dim C_{V}(x) = \dim C_{\bar{G}}(x)$ (see \cite[Section 1.14]{Carter}). If $x \in H \setminus H^0$ then $r=3$ and by considering \eqref{e:vres} we calculate that $\dim C_V(x) = 35+a$, where $a = \dim C_{\mathcal{L}(D_4)}(x)$. Since $a \leqs 28$ we have $\dim C_V(x) \leqs 63$ and thus $C_{\bar{G}}(x)^0 \not\in \{D_6T_1, E_6T_1\}$. For the remainder, we may assume that $x^G \cap H \subseteq H_0$. Write $x=yz \in H_0$, where $y \in {\rm L}_{2}(q^3)$ and $z \in {}^3D_4(q)$. 

We claim that $\dim C_V(x) < 67$ if $y \ne 1$ and $z \ne 1$. To see this, let $d$ denote the codimension of the largest eigenspace of $z$ on the natural $D_4$-module $V(\l_1)$. Since $z \in {}^3D_4(q)$, it follows that $d$ is also the codimension of the largest eigenspace of $z$ on both $V(\l_3)$ and $V(\l_4)$. In addition, $d \geqs 4$ and $\dim C_{D_4}(z) \leqs 10$ (see the proof of \cite[Lemma 2.12]{fpr4}, for example). Set 
\[
W_1 = \mathcal{L}(A_1^3D_4),\;\; W_2 = V_2 \otimes V_2 \otimes 0 \otimes V(\l_4).
\]
Then $\dim C_{W_1}(x) \leqs 1+1+1+10= 13$ and by applying \cite[Lemma 3.7]{LSh} we deduce that $\dim C_{W_2}(x) \leqs 32-4d \leqs 16$. Similarly, the $1$-eigenspace of $x$ on each of the other two summands in the decomposition \eqref{e:vres} has dimension at most $16$, so $\dim C_V(x) \leqs 13+48=61$ and the claim follows. In a similar fashion, one checks that $\dim C_V(x) < 79$ if $y=1$ and $z \ne 1$. It follows that $H$ contains fewer than $q^9$ semisimple elements $x$ with $C_{\bar{G}}(x) = E_6T_1$, and since $|x^G|>\frac{1}{2}(q+1)^{-1}q^{55}$ we deduce that their contribution to $\b$ is less than $q^{-34}$. Similarly, if $C_{\bar{G}}(x) = D_6T_1$ then $|x^G|>\frac{1}{2}(q+1)^{-1}q^{67}=b$ and there are fewer than $a=q^9+q^{28}$ semisimple elements $x \in H_0$ of the form $x=yz$ with $y=1$ or $z=1$, so the contribution here is less than $q^{-8}$.

We conclude that 
\[
\b < 4q^{-8} + 4q^{-6}+ q^{-34}+ q^{-8}.
\]
By combining this estimate with the above bound on $\a$, we see that
\[
\what{Q}(G,H,2) < 2q^{-6} + 2q^{-4} + 2q^{-14} + 2q^{-42} + 4q^{-8} + 4q^{-6}+ q^{-34}+ q^{-8} < q^{-2}
\]
and the result follows.
\end{proof}

\begin{lem}\label{l:e6}
If $G = E_6^{\e}(q)$ and $H = {\rm L}_{3}^{\e}(q^3).3$ then $P(G,H,2) > 1-q^{-4}$.
\end{lem}

\begin{proof}
It suffices to show that $\what{Q}(G,H,2)<q^{-4}$. Let $\bar{G} = E_6$ and $\bar{H} = A_2^3.S_3$ be the corresponding algebraic groups over the algebraic closure of $\mathbb{F}_q$ and let $V = \mathcal{L}(\bar{G})$ be the Lie algebra of $\bar{G}$. We have
\begin{equation}\label{e:e6res}
V \downarrow A_2^3 = \mathcal{L}(A_2^3) \oplus (V_3 \otimes V_3 \otimes V_3) \oplus (V_3^* \otimes V_3^* \otimes V_3^*),
\end{equation}
where $V_3$ denotes the natural module for $A_2$. It will be convenient to set 
\[
W_1 = \mathcal{L}(A_2^3), \;\; W_2 = V_3 \otimes V_3 \otimes V_3.
\] 
Also set $\mathbb{F} = \mathbb{F}_{q^{3u}}$, where $u=1$ if $\e=+$ and $u=2$ if $\e=-$.

Write $H = H_0.\la \varphi \ra = H_0.3$, where $H_0 = {\rm L}_{3}^{\e}(q^3)$ and $\varphi$ is a field automorphism of order $3$.
Without any loss of generality, we may assume that 
\[
H_0 = \{(x,x^{\varphi},x^{\varphi^2}) \,: \, x \in {\rm L}_{3}^{\e}(q^3)\} < \bar{H}^0,
\]
so $\varphi$ cyclically permutes the three $A_2$ factors of $\bar{H}^0$. In order to estimate $\a$ and $\b$, as defined in \eqref{e:ab},  we will refer repeatedly to the information on unipotent and semisimple conjugacy classes in \cite[Table 22.2.3]{LS_book} and \cite[Section 2]{FJ}, respectively. As in the proof of the previous lemma, we will use the notation from \cite{LS_book} for labelling unipotent classes.

Write $q=p^f$, with $p$ a prime, and let $x \in H$ be an element of prime order $r$. We begin by considering $\a$. First assume $r=p=2$. Here $H_0$ has a unique class of involutions (with Jordan form $[J_2,J_1]$ on $V_3$) and by considering the decomposition in \eqref{e:e6res} we calculate that $x$ has Jordan form $[J_2^{38}, J_1^2]$ on $V$. By inspecting \cite[Table 6]{Law}, it follows that $x$ is in the $\bar{G}$-class $A_1^3$, so $|x^G|>\frac{1}{2}q^{40}$, $|x^G \cap H| = i_2(H) = (q^3+\e)(q^9-\e)$ and we deduce that $\a< 2q^{-15}$ when $p=2$.

Now assume $r = p \geqs 3$. If $x \in H \setminus H_0$ then $r=3$ and the proof of \cite[Proposition 5.13]{BGS2} implies that $x$ is in one of the $\bar{G}$-classes labelled $A_2^2$ or $A_2^2A_1$. Therefore, $|x^G|>\frac{1}{2}q^{48}=a$ and we note that there are fewer than $i_3(H) < 2q^{15}(q^3+1)=b$ such elements in $H$ (see \cite[Proposition 1.3]{LLS}). It follows that the contribution to $\a$ from these elements is less than $a(b/a)^2<q^{-10}$. Now assume $x^G \cap H \subseteq H_0$, so $x \in H_0$. As above, we calculate that the long root elements $x \in H_0$ are contained in the $\bar{G}$-class labelled $A_1^3$. In this case, $|x^G|>\frac{1}{2}q^{40}$, $|x^G \cap H|<2q^{12}$ and thus the contribution is less than $q^{-14}$. Similarly, if $x \in H_0$ is a regular unipotent element then it has Jordan form 
\[
\left\{\begin{array}{l}
\mbox{$[J_5,J_3]^3$} \\
\mbox{$[J_3^2,J_1^2]^3$}
\end{array}\right. \oplus (J_3 \otimes J_3 \otimes J_3)^2 = \left\{\begin{array}{ll}
\mbox{$[J_5^{11},J_3^7,J_1^2]$} & \mbox{if $p \geqs 5$} \\
\mbox{$[J_3^{24},J_2^3]$} & \mbox{if $p=3$}
\end{array}\right.
\]
on $V$. By inspecting \cite[Table 6]{Law}, it follows that $x$ is in the $\bar{G}$-class $A_2^2A_1$ if $p=3$ and $D_4(a_1)$ if $p \geqs 5$. Therefore $|x^G|>\frac{1}{6}q^{54}$ and there are fewer than $2q^{18}$ such elements in $H$, so the contribution to $\a$ is less than $q^{-15}$. We conclude that 
\[
\a< q^{-10}+q^{-14}+q^{-15}.
\]

To complete the proof, it remains to estimate $\b$. Set $\bar{D} = C_{\bar{G}}(x)$ and let us first assume $\dim x^{\bar{G}} = 54$, in which case $\bar{D}^0 = A_2^3$, $r=3$ and $|x^G|>\frac{1}{6}q^{54}$. Since $i_3(H) < 2q^{15}(q^3+1)$ by \cite[Proposition 1.3]{LLS}, it follows that the contribution to $\b$ from these elements is less than $2q^{-14}$. Similarly, if $\dim x^{\bar{G}} \geqs 56$ then $|x^G|>\frac{1}{2}(q+1)^{-1}q^{57}=a$, $r \geqs 5$ and the bound $|H_0|<q^{24}$ implies that the contribution to $\b$ is less than $2q^{-7}$. 

For the remainder, we may assume that $\dim x^{\bar{G}} \leqs 52$ and thus
\[
\bar{D}^0 \in \{D_5T_1, A_5A_1, A_5T_1, D_4T_2, A_4A_1T_1, A_4T_2\}.
\]
Suppose $r=2$. Now $H$ has a unique conjugacy class of involutions and by considering \eqref{e:e6res} we calculate that $\dim C_V(x) = 38$ and thus 
$\bar{D}^0 = A_5A_1$. Therefore, $|x^G|>\frac{1}{2}q^{40}$ and $|x^G \cap H| = i_2(H) = q^6(q^6+\e q^3+1)$, so the contribution to $\b$ from involutions is less than $q^{-15}$.

Now assume $r \geqs 3$. If $x \in H \setminus H_0$ then $r=3$ and $x$ acts as a field automorphism on $H_0$, inducing a cyclic permutation on the three factors of $\bar{H}^0$. It follows that $\dim C_{W_1}(x) = 8$ and $\dim C_{W_2}(x) = \dim C_{W_2^*}(x) = 11$, so $\dim C_V(x) = 30$ and thus $\bar{D}^0 = D_4T_2$. Therefore, $|x^G|>\frac{1}{6}(q+1)^{-2}q^{50}$ and there are fewer than $i_3(H)<2q^{15}(q^3+1)$ such elements in $H$, so the contribution to $\b$ is less than $q^{-5}$. For the remainder, we may assume $r \geqs 3$ and $x^G \cap H \subseteq H_0$. In particular, $x \in H_0$. 

If $x$ is a regular semisimple element of $H_0$ then $\dim C_{W_1}(x)=6$ and by applying \cite[Lemma 3.7]{LSh} we deduce that 
\begin{equation}\label{e:cw2}
\dim C_{W_2}(x)  = \dim C_{W_2^*}(x)  \leqs 9.
\end{equation} 
Therefore, $\dim C_V(x) \leqs 24$ and thus $\dim x^{\bar{G}} \geqs 54$, which means that the contribution from these elements has already been accounted for. Now assume $x$ is non-regular, so $r$ divides $q^3-\e$ and we may assume $x$ lifts to an element of ${\rm GL}_{3}^{\e}(q^3)$ with eigenvalues $1,1,\l$ for some nontrivial $r$-th root of unity $\l \in \mathbb{F}$. In particular, note that $r \leqs q^2+q+1$. Now $\dim C_{W_1}(x) = 12$ and one checks that \eqref{e:cw2} holds, which gives $\dim x^{\bar{G}} \geqs 48$. Therefore, $|x^G|>\frac{1}{6}(q+1)^{-2}q^{50}=a$ and we calculate that there are fewer than
\[
\sum_{r \in \pi} (r-1) \cdot \frac{|{\rm GL}_{3}^{\e}(q^3)|}{(q^3-\e)|{\rm GL}_{2}^{\e}(q^3)|} < \log(q^3+1)\cdot q(q+1) \cdot 2q^{12}=b
\]
such elements in $H$, where $\pi$ is the set of odd prime divisors of $q^3-\e$. It follows that the combined contribution to $\b$ from these elements is less than $a(b/a)^2< q^{-10}$.
Therefore, 
\[
\b < 2q^{-14} + 2q^{-7} + q^{-15} + q^{-5} + q^{-10}
\]
and by combining this with the above estimate for $\a$, we deduce that  
\[
\what{Q}(G,H,2) < q^{-10}+q^{-14}+q^{-15}+2q^{-14} + 2q^{-7} + q^{-15} + q^{-5} + q^{-10} < q^{-4}
\]
as required.
\end{proof}

We are now ready to prove Theorem \ref{t:exceptional_udn}. We partition the proof into a sequence of lemmas. We begin by handling the rank $1$ groups 
$G \in \{ {}^2B_2(q), {}^2G_2(q) \}$ where we can compute $P(G,s,2)$ precisely for an appropriate element $s \in G$.

\begin{lem}\label{l:suz}
Let $G = {}^2B_2(q)$, where $q=2^{2m+1}$ and $m \geqs 1$, and let $s \in G$ be an element of order $q-\sqrt{2q}+1$. Then 
\[
P(G,s,2) = 1 - \frac{(q^2-4)(q-\sqrt{2q}+1)+4}{q^2(q-1)(q+\sqrt{2q}+1)} > 1 - q^{-1}
\]
and thus $\gamma_u(G)=2$.
\end{lem}

\begin{proof}
First we claim that $\mathcal{M}(G,s) = \{H\}$, where $H = N_G(\la s \ra) = K{:}L$ and $K,L$ are cyclic groups of order $q-\sqrt{2q}+1$ and $4$, respectively. By considering the orders of the maximal subgroups of $G$ (see \cite[Theorem 9]{Suz}), we see that every subgroup in 
$\mathcal{M}(G,s)$ is a conjugate of $H$. Now, if $r$ is a prime divisor of $|K|$, then $H$ has a unique subgroup $J$ of order $r$, which is contained in $K$. Therefore, if $g \in G$ and $J \leqs H \cap H^g$, then $J = J^g$ and thus $g \in N_G(J)=H$. Therefore, $|H \cap H^g| \in \{1,2,4\}$ for all $g \in G \setminus H$ and thus $\mathcal{M}(G,s) = \{H\}$ as claimed. Moreover, it follows that the nontrivial, nonregular $H$-orbits on $\O = G/H$ have length $q-\sqrt{2q}+1$ and $2(q-\sqrt{2q}+1)$.

Let $t$ be the number of regular orbits of $H$ on $\O$. Then Lemma \ref{l:udn2} gives
\begin{equation}\label{e:ps2}
P(G,s,2) = \frac{t|H|^2}{|G|}
\end{equation}
and so it remains to determine $t$. By \cite[Proposition 18]{Suz}, $G$ has two (non-real) conjugacy classes of elements of order $4$, both of size $|G|/2q$ (see \cite[Lemma 3.2]{FP}, for example). Similarly, $H$ also has two classes of such elements, both of size $q-\sqrt{2q}+1$, which are not fused in $G$. Also note that $G$ and $H$ both have unique conjugacy classes of involutions, of size $|G|/q^2$ and $q-\sqrt{2q}+1$, respectively.

Write $H = G_{\a}$ for a point $\a \in \Omega$. Any element $x \in H$ of order $4$ has 
\[
\frac{|x^G \cap H|}{|x^G|}\cdot |G:H| = \frac{1}{2}q
\]
fixed points on $\O = G/H$, one of which is $\a$. Clearly, $x$ acts fixed-point-freely on the orbits of length $2(q-\sqrt{2q}+1)$ and $4(q-\sqrt{2q}+1)$. If 
$\Gamma = H/J$ with $J=C_4$, then $x$ has 
\[
\frac{|x^H \cap J|}{|x^H|}\cdot |H:J| = 1
\]
fixed point on $\Gamma$. This implies that $H$ has precisely $q/2-1$ orbits of size $q-\sqrt{2q}+1$. Similarly, any element $y \in H$ of order $2$ has $\frac{1}{4}q^2$ fixed points on $\O$, which are distributed so that $y$ has $1$ fixed point on each $H$-orbit of length $q-\sqrt{2q}+1$ and $2$ on those of length $2(q-\sqrt{2q}+1)$. It follows that $H$ has 
\[
\frac{\frac{1}{4}q^2-1-\left(\frac{1}{2}q-1\right)}{2} = \frac{1}{8}q(q-2)
\]
orbits of length $2(q-\sqrt{2q}+1)$ and we conclude that   
\[
t = \frac{|G:H| - \frac{1}{2}(q-2)\cdot (q-\sqrt{2q}+1)-\frac{1}{8}q(q-2)\cdot 2(q-\sqrt{2q}+1) - 1}{|H|}.
\]
The result follows.
\end{proof}

\begin{rem}\label{r:suz}
Define $G$ and $s$ as in Lemma \ref{l:suz}. We claim that $P_2(G)=P(G,s,2)$, which shows that the general bound $P_2(G)>1-q^{-1}$ in Theorem \ref{t:exceptional_udn} is essentially best possible. To see this, let $x$ be any nontrivial element of $G$ and observe that $x$ is either contained in a Borel subgroup $B$ of $G$, or it normalises a cyclic maximal torus of order $q + \e\sqrt{2q}+1$ with $\e=\{+,-\}$ (this follows from Lemma \ref{l:par}). If $x \in B$, then $P(G,x,2)=0$ since $b(G,G/B)>2$. Therefore, we may assume $x$ is a regular semisimple element and $H^{\e} \in \M(G,x)$, where $H^{\e} = N_G(\la y \ra) = C_{q+\e \sqrt{2q}+1}{:}C_4$ for some element $y \in G$ of order $q+\e \sqrt{2q}+1$. Note that $|C_G(x)| = |C_G(y)| = q+\e \sqrt{2q}+1$ and $\M(G,y) = \{H^{\e}\}$, so $P(G,x,2) \leqs P(G,y,2)$. If $\e=-$ then $P(G,y,2) = P(G,s,2)$. On the other hand, if $\e=+$, then 
\[
P(G,y,2) = \frac{r|H^+|^2}{|G|},
\]
where $r$ is the number of regular orbits of $H^+$ on $G/H^+$. By arguing as in the proof of Lemma \ref{l:suz}, we deduce that
\[
r = \frac{|G:H^+| - \frac{1}{2}(q-2)\cdot (q+\sqrt{2q}+1)-\frac{1}{8}q(q-2)\cdot 2(q+\sqrt{2q}+1) - 1}{|H^+|}
\]
and one checks that $P(G,y,2)< P(G,s,2)$. This justifies the claim. 
\end{rem}

\begin{lem}\label{l:ree}
Let $G = {}^2G_2(q)$, where $q=3^{2m+1}$ and $m \geqs 1$, and let $s \in G$ be an element of order $q-\sqrt{3q}+1$. Then 
\[
P(G,s,2) = 1 - \frac{(q^3+2q^2-3q-6)(q-\sqrt{3q}+1)+6}{q^3(q^2-1)(q+\sqrt{3q}+1)} > 1-q^{-2}
\]
and thus $\gamma_u(G)=2$.
\end{lem}

\begin{proof}
This is very similar to the proof of Lemma \ref{l:suz}. Firstly, we observe that $\mathcal{M}(G,s) = \{H\}$ and $H = K{:}L$, where $K$ and $L$ are cyclic groups of order $q+\sqrt{3q}+1$ and $6$, respectively. In view of \eqref{e:ps2}, it suffices to compute $t$, the number of regular $H$-orbits on $\O = G/H$. By arguing as in the proof of Lemma \ref{l:suz}, we see that $|H \cap H^g| \in \{1,2,3,6\}$ for all $g \in G \setminus H$, so the nontrivial, nonregular $H$-orbits on $\O$ have length $c(q-\sqrt{3q}+1)$ for $c \in\{1,2,3\}$. 

In order to compute $t$, we need to consider the conjugacy classes of elements of order $2$, $3$ and $6$ in both $G$ and $H$ (the conjugacy classes of $G$ are determined in \cite{Ward} and we refer the reader to \cite[Section 3]{Jones0} for a convenient summary of the main facts we need). Both $G$ and $H$ have unique conjugacy classes of involutions, of size $q-\sqrt{3q}+1$ and $|G|/q(q^2-1)$, respectively. Similarly, $H$ has two (non-real) classes of elements of order $3$, both of size $q-\sqrt{3q}+1$, and the same is true for elements of order $6$. Now $G$ has three classes of elements of order $3$, two of size $|G|/2q^2$ and one of size $|G|/q^3$; the first two classes are non-real and they both meet $H$ (see \cite[Lemma 2.3(b)]{FP2}). Similarly, $G$ has two (non-real) classes of elements of order $6$, both of size $|G|/2q$.

Let $x \in H$ be an element of order $6$. In the usual manner, we calculate that $x$ has $q/3$ fixed points on $\O$. Moreover, if $\Gamma = H/J$ is an $H$-orbit of length $q-\sqrt{3q}+1$, then $x$ has a unique fixed point on $\Gamma$, whence $H$ has $q/3-1$ orbits of length $q-\sqrt{3q}+1$. Next suppose $y \in H$ has order $3$ and note that $y$ has $\frac{1}{3}q^2$ fixed points on $\O$. Then $y$ has a unique fixed point on each $H$-orbit of length $q-\sqrt{3q}+1$, and two fixed points on the orbits of 
length $2(q-\sqrt{3q}+1)$. This implies that $H$ has 
\[
\frac{\frac{1}{3}q^2 - \left(\frac{1}{3}q-1\right)-1}{2} = \frac{1}{6}q(q-1)
\]
orbits of length $2(q-\sqrt{3q}+1)$. Finally, let $z \in H$ be an involution. First we calculate that $z$ has $q(q^2-1)/6$ fixed points on $\O$. Now $z$ has a unique fixed point on each $H$-orbit of length $q-\sqrt{3q}+1$, and three fixed points on the $H$-orbits of length $3(q-\sqrt{3q}+1)$. Therefore, $H$ has 
\[
\frac{\frac{1}{6}q(q^2-1) - \left(\frac{1}{3}q-1\right)-1}{3} = \frac{1}{18}q(q^2-3)
\]
orbits of length $3(q-\sqrt{3q}+1)$.

Putting this together, we conclude that 
\[
t = \frac{|G:H| - \left(\frac{1}{3}q-1\right)a - \frac{1}{3}q(q-1)a - \frac{1}{6}q(q^2-3)a - 1}{|H|},
\]
where $a=q-\sqrt{3q}+1$. The result follows.
\end{proof}

\begin{rem}
Note that ${}^2G_2(3)' \cong {\rm L}_{2}(8)$, so $\gamma_u({}^2G_2(3)')=3$ by Proposition \ref{p:psl2}.
\end{rem}

To complete the proof of Theorem \ref{t:exceptional_udn}, it remains to handle the simple exceptional groups of rank at least two. 

\begin{lem}\label{l:e8}
Let $G=E_8(q)$ and let $s \in G$ be an element of order $q^8+q^7-q^5-q^4-q^3+q+1$. Then $P(G,s,2)>1-q^{-30}$ and $\gamma_u(G) = 2$. 
\end{lem}

\begin{proof}
This follows from the proof of \cite[Theorem 5.2]{BH}.
\end{proof}

\begin{lem}\label{l:2f4}
If $G = {}^2F_4(q)'$, then $\gamma_u(G) = 2$ if and only if $q \geqs 8$. Moreover, if $q \geqs 8$ and $s \in G$ has order $q^2+\sqrt{2q^3}+q+\sqrt{2q}+1$, then $P(G,s,2) > 1-q^{-3}$.
\end{lem}

\begin{proof}
First assume $q=2$. Here 
\[
G = \bigcup_{g \in G}H^g \cup  \bigcup_{g \in G}K^g,
\]
where $H = 2.[2^8].5.4$ and $K = {\rm L}_{3}(3).2$ (see \cite{Pel}, for example), and it is easy to check that $b(G,G/H) = b(G,G/K)=3$. Therefore $\gamma_u(G) \geqs 3$. In fact, by carrying out a random search in {\sc Magma} (see \cite[Section~1.2.4]{BH_comp}) we can demonstrate that $\gamma_u(G)=3$ (witnessed by the class {\tt 16A} in the notation of the {\sc Atlas} \cite{ATLAS}).

Now assume $q \geqs 8$ and let $s \in G$ be an element of order $\ell = q^2+\sqrt{2q^3}+q+\sqrt{2q}+1$. By \cite[Section 4(c)]{Wei}, we have $\mathcal{M}(G,s) = \{H\}$ with $H = C_{\ell}{:}C_{12}$. Since $|x^G|>\frac{1}{2}q^{11}$ for all $1 \ne x \in G$ (see \cite{Shin}), by applying Lemma \ref{l:bd} we deduce that  
\[
\what{Q}(G,s,2) < \frac{1}{2}q^{11}(2q^{-11}\cdot |H|)^2< q^{-3}
\]
and the result follows.
\end{proof}

\begin{lem}\label{l:3d4}
Let $G = {}^3D_4(q)$ and let $s \in G$ be an element of order $q^4-q^2+1$. Then $P(G,s,2)>1-q^{-4}$ and $\gamma_u(G)=2$.
\end{lem}

\begin{proof}
By \cite[Section 4(e)]{Wei}, we have $\mathcal{M}(G,s) = \{H\}$ with $H = C_{q^4-q^2+1}{:}C_4$. Let $x \in H$ be an element of prime order $r$. We claim that $|x^G|>q^{16}$. If $x$ is semisimple, then $|x^G|$ can be read off from \cite[Proposition 2.2]{DM} and the desired bound follows. Now assume $r=p=2$. From the proof of \cite[Lemma 4.6]{BTh}, we deduce that $H$ contains a unique class of involutions and they belong to the $A_1^3$ class of $G$ (in the notation of \cite{Spal}). This gives $|x^G|>q^{16}$ as required and by applying Lemma \ref{l:bd} we conclude that 
\[
\what{Q}(G,s,2) < q^{16}(q^{-16}\cdot |H|)^2< q^{-4}
\]
for all $q \geqs 2$. The result follows.
\end{proof}

\begin{lem}\label{l:e72}
Let $G = E_7(q)$ and let $s \in G$ be an element of order 
\[
\frac{(q^3-1)(q^4-q^2+1)}{(2,q-1)}.
\] 
Then $P(G,s,2)>1-q^{-2}$ and $\gamma_u(G)=2$.
\end{lem}

\begin{proof}
In view of Lemma \ref{l:e7}, it suffices to show that $\mathcal{M}(G,s) = \{H\}$ with $H = ({\rm L}_{2}(q^3) \times {}^3D_4(q)).3$. Set $T = \la s \ra$ and $N = N_G(T)$. By the main theorem of \cite{LSS}, $N$ is not maximal in $G$, so $N < H < G$ for some maximal rank subgroup $H$. By considering the order of $s$, it is clear that $H$ is not a maximal parabolic subgroup of $G$. Moreover, further inspection of \cite[Tables 5.1 and 5.2]{LSS} shows that $H = ({\rm L}_{2}(q^3) \times {}^3D_4(q)).3$ is the only option. Now $G$ has a unique conjugacy class of maximal subgroups of this form and we will write $n$ for the number of $G$-conjugates of $H$ containing $s$. We need to show that $n=1$. Note that $C_G(s) = C_H(s)$, so
\[
n = \frac{|s^G \cap H|}{|s^G|} \cdot \frac{|G|}{|H|} \geqs \frac{|s^H|}{|s^G|} \cdot \frac{|G|}{|H|} = 1
\]
and thus $n=1$ if and only if $s^G \cap H = s^H$. 

Suppose $t \in H$ is $G$-conjugate to $s$, say $s = t^g$ for some $g \in G$. We need to show that $s$ and $t$ are $H$-conjugate. From the structure of $H$, it is easy to determine the $H$-classes of maximal tori of $H$; we see that there is a unique class of maximal tori of order $|s|$, so $T = \la s \ra$ and $\la t \ra$ are $H$-conjugate. Therefore, by replacing $t$ by an appropriate $H$-conjugate, if necessary, we may assume that $\la  s \ra = \la t \ra$.  But then $g$ normalizes $T$ and we have $N_G(T) < H$, so $s$ and $t$ are indeed $H$-conjugate and we conclude that $n=1$ as required. 
\end{proof}

\begin{lem}\label{l:e62}
Let $G = E_6^{\e}(q)$ and let $s \in G$ be an element of order $(q^6+\e q^3+1)/(3,q-\e)$. Then $P(G,s,2)>1-q^{-4}$ and $\gamma_u(G)=2$.
\end{lem}

\begin{proof}
By \cite[Section 4(g,h)]{Wei} and \cite[Proposition 6.2]{GK} we have $\mathcal{M}(G,s) = \{H\}$ with 
$H = {\rm L}_{3}^{\e}(q^3).3$. Now apply Lemma \ref{l:e6}.
\end{proof}

Finally, we prove that $\gamma_u(G) \geqs 3$ for the remaining two families of exceptional groups.

\begin{lem}\label{l:g2}
If $G = G_2(q)'$, then $\gamma_u(G) \geqs 3$.
\end{lem}

\begin{proof}
If $q=2$ then $G \cong {\rm U}_{3}(3)$ and with the aid of {\sc Magma} it is easy to check that $\gamma_u(G)=3$. For the remainder, we may assume $q \geqs 3$.

As in Lemma \ref{l:par}, let $\mathcal{P}$ be the union of the parabolic subgroups of $G$. In addition, let $\mathcal{H}^{\e}$ be the union of the maximal rank subgroups of $G$ of the form ${\rm SL}_{3}^{\e}(q){:}2$ for $\e=\pm$. We claim that 
\[
G = \mathcal{P} \cup \mathcal{H}^{+} \cup \mathcal{H}^{-}.
\]
Notice that this implies that every element in $G$ is contained in a maximal subgroup $H$  with $b(G,G/H) \geqs 3$, so Lemma \ref{l:udn} gives $\gamma_u(G) \geqs 3$ as required. (In fact, if $q \geqs 4$ is even, then a theorem of Bubboloni et al. \cite{BLW} implies that $G = \mathcal{H}^{+} \cup \mathcal{H}^{-}$.)

By Lemma \ref{l:par}, it suffices to show that every maximal torus of $G$ is either contained in a Levi factor of a maximal parabolic subgroup, or in a maximal subgroup of the form ${\rm SL}_{3}^{\e}(q){:}2$. There are six conjugacy classes of maximal tori in $G$ (corresponding to the six conjugacy classes in the Weyl group of $G$, which is isomorphic to $D_{12}$):
\[
C_{q-\e}^2,\; C_{q^2-\e q+1},\; C_{q^2-1} (\mbox{two classes}).
\]
There are two classes of maximal parabolic subgroups, both with Levi factor ${\rm GL}_{2}(q)$, so the maximal tori $C_{q-1}^2$ and $C_{q^2-1}$ (both classes) are contained in Levi factors. In addition, ${\rm SU}_{3}(q)$ contains $C_{q+1}^2$ and $C_{q^2-q+1}$, and ${\rm SL}_{3}(q)$ contains $C_{q^2+q+1}$. The result follows.
\end{proof}

\begin{lem}\label{l:f4}
If $G = F_4(q)$, then $\gamma_u(G) \geqs 3$.
\end{lem}

\begin{proof}
This is similar to the proof of the previous lemma. Define $\mathcal{P}$ as before and let $\mathcal{H}$ and $\mathcal{K}$ be the union of the maximal rank subgroups of the form $B_4(q)$ and ${}^3D_4(q)$, respectively. It suffices to show that 
\[
G = \mathcal{P} \cup \mathcal{H} \cup \mathcal{K}.
\]
The Weyl group of $G$ has $25$ conjugacy classes, so there are $25$ classes of maximal tori. Similarly, $B_4(q)$ has $20$ classes of maximal tori and by studying the embedding of these tori in $G$, we find that all but $7$ classes of maximal tori of $G$ have representatives  contained in a $B_4(q)$ subgroup (see \cite[pp.95--96]{Law99}). The exceptions are 
\[
C_{(q^3-\e)(q+\e)}, \; C_{q^3-\e}\times C_{q-\e},\; C_{q^2+\e q+1}^2,\; C_{q^4-q^2+1}.
\]
However, these are precisely the maximal tori of ${}^3D_4(q)$, so $\mathcal{H} \cup \mathcal{K}$ contains every maximal torus in $G$ and the result follows. (Note that if $q=3^f$ then $G = \mathcal{H} \cup \mathcal{K}$ by \cite{BLW}.)
\end{proof}

This completes the proof of Theorem \ref{t:exceptional_udn}. In particular, the proof of Theorem \ref{t:exmain} is complete.

\section{Two-dimensional linear groups} 

In this section we prove Theorem \ref{t:psl2main}. Set $G = {\rm L}_{2}(q)$ with $q \geqs 4$. 

\subsection{Spread}

We start by studying the spread and uniform spread of $G$. Define 
\[
f(q) = \left\{\begin{array}{ll}
q-1 & \mbox{if $q \equiv 1 \imod{4}$} \\
q-4 & \mbox{if $q \equiv 3 \imod{4}$} \\
q-2 & \mbox{if $q$ is even.}
\end{array}\right.
\]

\begin{lem}\label{l:psl2_1}
Let $G = {\rm L}_{2}(q)$ with $q \geqs 4$. Assume $q \geqs 11$ if $q$ is odd. Then 
$u(G) \geqs f(q)$.
\end{lem}

\begin{proof}
First assume $q$ is odd and let $s \in G$ be an element of order $(q+1)/2$, so $\mathcal{M}(G,s) = \{H\}$ with $H = D_{q+1}$. Let $x \in H$ be an element of prime order $r$. If $r=2$ then ${\rm fpr}(x,G/H) = i_2(H)/i_2(G)$ since $G$ has a unique class of involutions and we get
\[
{\rm fpr}(x,G/H) = \left\{\begin{array}{ll} \frac{1}{q} & \mbox{if $q \equiv 1 \imod{4}$} \\
\frac{q+3}{q(q-1)} & \mbox{if $q \equiv 3 \imod{4}$.}
\end{array}\right.
\]
Similarly, if $r$ is odd then $|x^G \cap H| = 2$, $|x^G| = q(q-1)$ and thus
\[
{\rm fpr}(x,G/H) = \frac{2}{q(q-1)}.
\] 
We conclude that if $x \in G$ has prime order, then
\[
{\rm fpr}(x,G/H) < \left\{\begin{array}{ll} \frac{1}{q-1} & \mbox{if $q \equiv 1 \imod{4}$} \\
\frac{1}{q-4} & \mbox{if $q \equiv 3 \imod{4}$} 
\end{array}\right.
\]
and the result follows by Corollary \ref{c:ug}. A very similar argument applies when $q$ is even, working with an element of order $q+1$.
\end{proof}

\begin{rem}
As noted in Remark \ref{r:psl2main}, the spread of $G = {\rm L}_{2}(q)$ is studied by Brenner and Wiegold in \cite{BW75} and the bound $s(G) \geqs f(q)$ is established in \cite[Theorem 4.02]{BW75}. In fact, this result states that $s(G) = f(q)$ for all $q \geqs 11$, but we will show below that this is false when $q \equiv 3 \imod{4}$.
\end{rem}

Fix subgroups $A$ and $B$ of $G$, where $A = D_{e(q+1)}$, $B$ is a Borel subgroup and $e = 2$ if $q$ is even, otherwise $e=1$. It will be useful to record that 
\begin{equation}\label{e:cup}
G = \bigcup_{g \in G}A^g \cup \bigcup_{g \in G}B^g
\end{equation}
(see \cite[Corollary 4.3]{BL}, for example).

\begin{thm}\label{t:l2qs}
Let $G = {\rm L}_{2}(q)$ with $q \geqs 4$ and $q \not\equiv 3 \imod{4}$. Assume that $q \geqs 13$ if $q$ is odd. Then $s(G) = u(G) = f(q)$.
\end{thm}

\begin{proof}
In view of Lemma \ref{l:psl2_1}, it suffices to show that $s(G) \leqs f(q)$. First assume $q \equiv 1 \imod{4}$. Fix a Borel subgroup $B$ of $G$ and let $x^B$ be the unique class of involutions in $B$. Note that $|x^B|=q$, say $x^B = \{x_1, \ldots, x_q\}$. Let $A=C_G(x_1) = D_{q+1}$, which is a maximal subgroup of $G$. We claim that there is no element $y \in G$ such that $G = \la x_i,y \ra$ for all $i$. In view of \eqref{e:cup}, it suffices to show that $\bigcup_{i}\mathcal{M}(G,x_i)$ contains every $G$-conjugate of $A$ and $B$. 

First note that $G$ has $q+1$ Borel subgroups, say $B,B_1, \ldots, B_q$. By considering fixed points, one checks that each involution in $G$ is contained in exactly two Borel subgroups. Moreover, we have $B \cap B_i = C_{(q-1)/2}$ for all $i$, so $B \cap B_i$ contains a unique involution. Therefore, we may assume that $x_i \in B \cap B_i$ and thus $\bigcup_{i}\mathcal{M}(G,x_i)$ contains every conjugate of $B$.

Now let us consider the conjugates of $A$. First note that $(|A^g|,|B|)=2$ and thus $|A^g \cap B| \leqs 2$ for all $g \in G$ (in fact, equality holds for all $g \in G$).  Now each $x_i$ is contained in $(q-1)/2$ conjugates of $A$; if $x_i$ and $x_j$ are both contained in $A^g$, then $\la x_i, x_j \ra \leqs A^g \cap B$ and thus $i=j$. This shows that $\bigcup_{i}\mathcal{M}(G,x_i)$ contains all $q(q-1)/2$ conjugates of $A$ and the result follows.

To complete the proof, let us assume $q$ is even. Let $B$ be a Borel subgroup of $G$ and fix maximal subgroups $A = D_{2(q+1)}$ and $C = D_{2(q-1)}$. Let $x^B = \{x_1, \ldots, x_{q-1}\}$ be the set of involutions in $B$. We claim that there is no $y \in G$ such that $G = \la x_i,y \ra$ for all $i$. To see this, let us first observe that each $x_i$ is contained in a unique Borel subgroup (namely $B$ itself), and also $q/2$ conjugates of both $A$ and $C$. Moreover, $|A^g \cap B| = 2$ and $|C^g \cap B| \in \{2,q-1\}$ for all $g \in G$, so $\bigcup_i\mathcal{M}(G,x_i)$ contains $q(q-1)/2$ subgroups of the form $D_{2(q\pm 1)}$. In particular, $\bigcup_i\mathcal{M}(G,x_i)$ contains every conjugate of $A$ and all but $q$ conjugates of $C$. Let $H_1, \ldots, H_q$ be the conjugates of $C$ that are not contained in $\bigcup_i\mathcal{M}(G,x_i)$

Seeking a contradiction, suppose there is an element $y \in G$ with $G = \la x_i,y \ra$ for all $i$. By considering \eqref{e:cup}, it follows that $y$ must be contained in a Borel subgroup and $|y|>1$ is a divisor of $q-1$. Without loss of generality, we may as well assume $y$ has order $q-1$. In particular, $y$ is contained in a unique conjugate of $C$, namely $N_G(\la y \ra)$. Since we are assuming $y \not\in \bigcup_i\mathcal{M}(G,x_i)$, it follows that $y \in H_i$ for some $i$. As noted above, we have $|H_i \cap B| \in \{2,q-1\}$. If $H_i \cap B = \la z \ra$ for an involution $z$, then $z=x_j$ for some $j$ and we get $H_i \in \mathcal{M}(G,x_j)$, which is a contradiction. Therefore, $|H_i \cap B|=q-1$ and thus $H_i \cap B = \la y \ra$ since $H_i$ has a unique subgroup of order $q-1$. But this implies that $y \in B$, which is a contradiction since $B \in \mathcal{M}(G,x_1)$. The result follows.
\end{proof}

Now assume $q \equiv 3 \imod{4}$ and $q \geqs 11$. By Lemma \ref{l:psl2_1} we have $u(G) \geqs f(q) = q-4$. In the proof of Theorem \ref{t:l2qs}, we worked with the set of involutions in a fixed Borel subgroup $B$ of $G$. However, a different  approach is needed when $q \equiv 3 \imod{4}$ since $|B| = q(q-1)/2$ is odd and it is more difficult to determine the exact spread of $G$. Indeed, this remains an open problem.

The next result gives a lower bound on $s(G)$ when $q$ is a prime. 
In particular, we see that the difference $s(G)-u(G)$ for a non-abelian simple group $G$ can be arbitrarily large.

\begin{prop}\label{p:psl2_2}
Let $G = {\rm L}_{2}(q)$, where $q \equiv 3 \imod{4}$ and $q \geqs 11$ is a prime. Then $s(G) \geqs (3q-7)/2$ and $s(G) - u(G) = (q+1)/2$.
\end{prop}

\begin{proof}
Fix a maximal subgroup $A = D_{q+1}$ and let $B$ be a Borel subgroup of $G$.  Consider a subset $\{x_1, \ldots, x_m\}$ of nontrivial elements in $G$.

If $s \in G$ has order $(q+1)/2$, then $s$ is contained in a unique maximal subgroup of $G$, namely $N_G(\la s \ra)$, which is a conjugate of $A$. Similarly, each element  of order $q$ is contained in a unique maximal subgroup (here we are using the hypothesis that $q$ is a prime), which is a conjugate of $B$. In view of \eqref{e:cup}, it follows that 
there is no element $y \in G$ such that $G = \la x_i,y\ra$ for all $i$ if and only if 
$\bigcup_{i}\mathcal{M}(G,x_i)$ contains every conjugate of $A$ and $B$. Since $(|A|,|B|)=1$, we deduce that 
\begin{equation}\label{e:sg}
s(G) = |S| + |T| - 1,
\end{equation}
where $S$ and $T$ are subsets of $G^{\#}$ of minimal size such that $\bigcup_{x \in S}\mathcal{M}(G,x)$ contains every conjugate of $A$ and 
$\bigcup_{x \in T}\mathcal{M}(G,x)$ contains every conjugate of $B$. Note that there are $q(q-1)/2$ conjugates of $A$ and $q+1$ conjugates of $B$.

First consider $|S|$. Let $x \in G$ be a nontrivial element. As observed in the proof of \cite[Theorem 4.02]{BW75}, if $|x|=2$ then $x$ is contained in precisely $(q+3)/2$ conjugates of $A$. On the other hand, if $|x|>2$ divides $q+1$, then $x$ is contained in a unique conjugate of $A$. Therefore, 
\[
|S| \geqs \frac{\frac{1}{2}q(q-1)}{\frac{1}{2}(q+3)} > q-4
\]
and thus $|S| \geqs q-3$. 

Now let us consider $|T|$. As noted above, each $x \in G$ of order $q$ is contained in a unique conjugate of $B$. Similarly, any nontrivial element of order dividing $(q-1)/2$ is contained in exactly two conjugates of $B$. This implies that $|T| \geqs (q+1)/2$. We claim that $|T| = (q+1)/2$. To see this, let $B_1, \ldots, B_{q+1}$ be the Borel subgroups of $G$ and note that $B_i \cap B_j = C_{(q-1)/2}$ and $B_i \cap B_j \cap B_k = 1$ for distinct $i,j,k$. In particular, if we write $B_i \cap B_j = \la x_{i,j} \ra$, then $T = \{x_{1,2}, x_{3,4}, \ldots, x_{q,q+1}\}$ has the desired property. This justifies the claim and we conclude that
\[
s(G) = |S| + |T| - 1 \geqs (q-3) + \frac{1}{2}(q+1) -1 = \frac{1}{2}(3q-7)
\]
as required.

Finally, let us consider $u(G)$. Fix a conjugacy class $x^G$ with $x \ne 1$ and note that $x$ is contained in a conjugate of $A$ or $B$ (see \eqref{e:cup}). If $x$ is contained in a conjugate of $A$, then there is no $g \in G$ such that $G = \la s, x^g \ra$ for all $s \in S$. Similarly, if $x$ is in a conjugate of $B$, then there is no $g \in G$ such that $G = \la t, x^g \ra$ for all $t \in T$. Therefore,
\[
u(G) = \max\{|S|,|T|\}-1 = |S|-1
\]
and thus $s(G) - u(G) = |T|=(q+1)/2$.
\end{proof}

\begin{rem}
Let $G = {\rm L}_{2}(q)$, where $q \geqs 11$ is a prime with $q \equiv 3 \imod{4}$. As noted in the proof of the previous proposition, in order to compute $s(G)$ we need to determine $|S|$ in \eqref{e:sg}. For $q=11$, we can use {\sc Magma} to show that $|S|=9$, which gives $s(G) = 9+6-1 = 14$ and $u(G)=8$.
\end{rem}

\begin{rem}
The case $G = {\rm L}_{2}(7)$ requires special attention. In \cite[Section 4]{BW75}, it is observed that $s(G) \geqs 3$. Using {\sc Magma}, we can identify $6$ elements in $G$ to show that $s(G) < 6$. For instance, if we take
\[
G = \la (3, 6, 7)(4, 5, 8), (1, 8, 2)(4, 5, 6) \ra < S_8
\]
then
\[
\mathcal{A}=\left\{
    \begin{array}{c}
    (2, 3, 8)(4, 5, 7),
    (1, 8, 5)(2, 7, 6),
    (2, 3, 4)(6, 8, 7), \\
    (2, 7, 8)(3, 5, 6),
    (1, 6, 4)(3, 5, 7),
    (1, 5, 2)(3, 8, 4)
\end{array}
\right\}
\]
has the desired property (that is, there is no $y \in G$ such that $G = \la x,y \ra$ for all $x \in \mathcal{A}$). In addition, an exhaustive search shows that $s(G) \geqs 5$, whence $s(G)=5$. Finally, one checks that $u(G)=3$. For example, the class $(1,7,2,3,8,5,6)^G$ is a witness to the bound $u(G) \geqs 3$. 
\end{rem}

The next result gives a lower bound on $s({\rm L}_{2}(q))$ which is valid for all $q \geqs 11$ with $q \equiv 3 \imod{4}$. In particular, this shows that the claim $s({\rm L}_{2}(q)) = q-4$ in \cite[Theorem 4.0.2]{BW75} is incorrect for all such $q$.

\begin{prop}\label{p:psl2_3}
Let $G = {\rm L}_{2}(q)$, where $q \equiv 3 \imod{4}$ and $q \geqs 11$. Then $s(G) \geqs q-3$.
\end{prop}
 
\begin{proof}
Define $A$ and $B$ as before (see \eqref{e:cup}) and set $C = D_{q-1}$. Note that $G$ contains $\frac{1}{2}q(q-1)$ conjugates of $A$. Let $\{x_1, \ldots, x_{q-3}\}$ be a set of nontrivial elements of $G$. 

First observe that each $y \in G$ of order $(q+1)/2$ is contained in a unique maximal subgroup of $G$, namely $N_G(\la y \ra)$, which is a conjugate of $A$. In addition, each involution is contained in $\frac{1}{2}(q+3)$ conjugates of $A$, whereas every other nontrivial element of $G$ is contained in at most one conjugate of $A$. This implies that if at least one $x_i$ is not an involution, then $\bigcup_{i}\mathcal{M}(G,x_i)$ does not contain every conjugate of $A$ and it follows that there exists an element $y \in G$ of order $(q+1)/2$ such that $G = \la x_i, y\ra$ for all $i$. Therefore, we may assume each $x_i$ is an involution. We claim that there is an element $y \in G$ of order $(q-1)/2$ such that $G = \la x_i, y\ra$ for all $i$. In particular, this shows that $s(G) \geqs q-3$, as required.

To justify the claim, first observe that $G$ contains $q(q+1)/2$ distinct conjugates of $C$. If $s \in G$ has order $(q-1)/2$, then $\mathcal{M}(G,y)$ comprises a unique conjugate of $C$ (namely, $N_G(\la s \ra)$) and two conjugates of $B$. Now each involution in $G$ is contained in $(q+1)/2$ conjugates of $C$ and no conjugates of $B$ (since $|B|$ is odd). In particular, 
$\bigcup_{i}\mathcal{M}(G,x_i)$ does not contain every conjugate of $C$. Moreover, if $N_G(\la y \ra)$ is such a conjugate of $C$, then $G = \la x_i, y\ra$ for all $i$ and the result follows.
\end{proof}

To conclude this section, we briefly consider the spread and uniform spread of $G = {\rm PGL}_{2}(q)$, where $q \geqs 5$ is odd. 

If $q=5$ then $G \cong S_5$ and thus $s(G) = 3$ and $u(G)=2$ by  Theorem~\ref{p:sn_odd}. If $q=7$, then a computation in \textsc{Magma} yields $s(G) \geqs u(G) = 4$. Now assume that $q \geqs 9$. In her PhD thesis \cite{Garion}, Garion states that $s(G) = q-4$ (see \cite[Proposition~6.2.4]{Garion}), but her argument only establishes the bound $u(G) \geqs q-4$. Indeed, the problem of determining the exact values of $s(G)$ and $u(G)$ is still open.

Our main result is the following.

\begin{prop}\label{p:pgl2}
Let $G = {\rm PGL}_{2}(q)$ with $q \geqs 9$ odd. Then
\[
q-4 \leqs u(G) \leqs s(G) \leqs q-1.
\]
\end{prop}

\begin{proof}
First we show that $u(G) \geqs q-4$. Let $s \in G$ be an element of order $q+1$ and note that $\mathcal{M}(G,s) = \{H\}$ with $H = N_G(\la s \ra) = D_{2(q+1)}$. Fix an element $x \in H$ of prime order $r$. 

Suppose $r=2$. Now $H$ has three classes of involutions (a central involution, plus two classes of size $(q+1)/2$) and $G$ contains two classes (represented by the elements $t_1$ and $t_1'$ in the notation of \cite[Table 4.5.1]{GLS}). We calculate that 
\[
{\rm fpr}(x,G/H) = \left\{ \begin{array}{ll}
\frac{1}{q} & \mbox{if $x$ is conjugate to $t_1$} \\
\frac{q+3}{q(q-1)} & \mbox{if $x$ is conjugate to $t_1'$.}
\end{array}\right.
\]
If $r>2$ then $r$ divides $(q+1)/2$, $|x^G| = q(q-1)$ and $|x^G \cap H| \leqs r-1 \leqs (q-1)/2$, so ${\rm fpr}(x,G/H) < 1/2q$. Therefore,
\[
{\rm fpr}(x,G/H) < \frac{1}{q-4}
\]
for all $x \in G$ of prime order, whence $u(G) \geqs q-4$ by Corollary \ref{c:ug}.

Finally, let us turn to the upper bound on $s(G)$; here we essentially repeat the argument in the proof of Theorem \ref{t:l2qs} with $q \equiv 1 \imod{4}$. Fix a maximal subgroup 
$A = D_{2(q+1)}$ and let $B$ be a Borel subgroup of $G$. Note that $G$ contains $q+1$ Borel subgroups, say $B, B_1, \ldots, B_q$. Let $\{x_1, \ldots, x_q\}$ be the 
unique class of involutions in $B$ and set $\mathcal{M} = \bigcup_{i}\mathcal{M}(G,x_i)$. It suffices to show that $\mathcal{M}$ contains every conjugate of $A$ and $B$.

Since $B \cap B_i = C_{q-1}$ contains a unique involution, we may assume that $x_i \in B_i$ for all $i$ and thus $\mathcal{M}$ contains every Borel subgroup. Now each $x_i$ is contained in $(q-1)/2$ conjugates of $A$, and there are $q(q-1)/2$ conjugates of $A$ in total. We have $|A^g \cap B|= 2$ for all $g \in G$, so no two of the $x_i$ are contained in the same conjugate of $A$. Therefore, $\mathcal{M}$ contains every conjugate of $A$ and the result follows.
\end{proof}

\subsection{Uniform domination}

By \cite[Proposition 6.4]{BH}, we have $\gamma_u({\rm L}_2(q)) \leqs 4$, with equality if and only if $q=9$. Our first aim is to determine the exact value of $\gamma_u({\rm L}_2(q))$ for all $q$. We begin by recording a preliminary lemma.

\begin{lem}\label{l:subdegrees}
Let $G = {\rm L}_{2}(q)$ where $q \geqs 11$ is odd and consider the action of $G$ on the set of cosets of $H = D_{q+1}$. The nontrivial subdegrees are as follows:
\begin{itemize}\addtolength{\itemsep}{0.2\baselineskip}
\item[{\rm (i)}] $q \equiv 1 \imod{4}${\rm :} $(q+1)/2$ and $q+1$, with multiplicities $(q-3)/2$ and $(q-1)/4$, respectively.
\item[{\rm (ii)}] $q \equiv 3 \imod{4}${\rm :} $(q+1)/4$, $(q+1)/2$ and $q+1$, with multiplicities $2$, $(q-3)/2$ and $(q-3)/4$, respectively.
\end{itemize}
In particular, $b(G,G/H) = 2$.
\end{lem}

\begin{proof}
First observe that $H \cap H^g \leqs C_2 \times C_2$ for all $g \in G \setminus H$ (see \cite[Lemma 2(b)]{Schmidt}, for example). Suppose $q \equiv 1 \imod{4}$, so $|H \cap H^g| \in \{1,2\}$ for all $g \in G \setminus H$. If $y \in H$ has order $2$ then $|y^G \cap H| = (q+1)/2$ and $|y^G| = q(q+1)/2$, so $y$ has $(q-1)/2$ fixed points on $\O=G/H$. Moreover, each $H$-orbit of length $(q+1)/2$ contains a unique fixed point of $y$, so $H$ has $(q-3)/2$ such orbits in total and we deduce that $H$ has 
\[
\frac{\frac{1}{2}q(q-1)-\frac{1}{4}(q+1)(q-3) - 1}{q+1} = \frac{1}{4}(q-1)
\]
regular orbits.
 
Now assume $q \equiv 3 \imod{4}$, so $H = C_G(x)$ for an involution $x \in G$. For $g \in G \setminus H$ we observe that $H \cap H^g = \la x, x^g \ra = C_2 \times C_2$ if and only if $[x,x^g]=1$. Similarly, $H \cap H^g=1$ if and only if there is no involution in $G$ that commutes with both $x$ and $x^g$. Now $x$ commutes with $(q+1)/2$ involutions (other than $x$ itself) and thus $H$ has
\[
\frac{\frac{1}{2}(q+1)}{\frac{1}{4}(q+1)} = 2
\]
orbits of length $(q+1)/4$. By \cite[Theorem 1.1(ii)]{BBPR}, there are precisely $(q+1)(q-3)/4$ elements $x^g$ for which there is no involution commuting with both $x$ and $x^g$. This implies that $H$ has
\[
\frac{\frac{1}{4}(q+1)(q-3)}{q+1} = \frac{1}{4}(q-3)
\]
regular orbits. Finally, we deduce that $H$ has 
\[
\frac{\frac{1}{2}q(q-1) - \frac{1}{2}(q+1) - \frac{1}{4}(q+1)(q-3) - 1}{\frac{1}{2}(q+1)} = \frac{1}{2}(q-3)
\]
orbits of length $(q+1)/2$. The result follows.
\end{proof}

\begin{prop}\label{p:psl2}
Let $G={\rm L}_{2}(q)$ with $q \geqs 4$. Then
\[
\gamma_u(G) = \left\{\begin{array}{ll}
4 & \mbox{if $q=9$} \\
3 & \mbox{if $q \in \{5,7\}$ or $q$ is even} \\
2 & \mbox{if $q \geqs 11$ is odd.}
\end{array}\right.
\]
\end{prop}

\begin{proof}
The cases with $q \leqs 9$ can be checked directly, so let us assume $q \geqs 11$. Note that $\gamma_u(G) \in \{2,3\}$ by \cite[Proposition 6.4]{BH}. 

First assume that $q$ is odd and let $s \in G$ be an element of order $(q+1)/2$. As noted in the proof of \cite[Proposition 6.4]{BH}, we have $\mathcal{M}(G,s) = \{H\}$ with $H = D_{q+1}$, and thus $\gamma_u(G) \leqs b(G,G/H)=2$, by Lemma~\ref{l:subdegrees}.

Now assume $q \geqs 16$ is even. We claim that each $g \in G$ is contained in a maximal subgroup $H$ of $G$ with $b(G,G/H) \geqs 3$, which implies that $\gamma_u(G) \geqs 3$ (and hence equality holds by \cite[Proposition 6.4]{BH}). To see this, we consider the action of $g$ on the natural module for $G$. If $g$ acts reducibly, then $g$ is contained in the stabiliser of a $1$-space (that is, a Borel subgroup of $G$) and the claim follows. On the other hand, if $g$ acts irreducibly then it is contained in a maximal dihedral subgroup $H= D_{2(q+1)}$ and the subdegrees for the action of $G$ on $G/H$ are presented in \cite[Table 2]{FI}. We see that there are no regular suborbits and thus $b(G,G/H) \geqs 3$. This justifies the claim and the proof of the proposition is complete.
\end{proof}

The next result completes the proof of Theorem \ref{t:psl2main}.

\begin{prop}\label{p:psl2p2}
Let $G = {\rm L}_{2}(q)$ with $q \geqs 11$ odd. Then  $P_2(G) = g(q)$, where
\[
g(q) = \left\{\begin{array}{ll}
\frac{1}{2}\left(1+\frac{1}{q}\right) & \mbox{if $q \equiv 1 \imod{4}$} \\
\frac{1}{2}\left(1 - \frac{q+3}{q(q-1)}\right) & \mbox{if $q \equiv 3 \imod{4}$.}
\end{array}\right.
\]
In particular, $P_2(G) \geqs \frac{24}{55}$, with equality if and only if $q=11$.
\end{prop}

\begin{proof}
Let $s \in G$ be an element of order $(q+1)/2$. Then, as noted above, $\mathcal{M}(G,s) = \{H\}$ with $H = N_G(\la s \ra) = D_{q+1}$, whence
\[
P(G,s,2) = \frac{r|H|^2}{|G|},
\]
where $r$ is the number of regular orbits of $H$ on $G/H$ (see Lemma \ref{l:udn2}). By applying Lemma \ref{l:subdegrees}, we deduce that $P(G,s,2) = g(q)$ and thus $P_2(G) \geqs g(q)$.

To complete the proof, we need to show that $P(G,t,2) \leqs P(G,s,2)$ for all $t \in G^{\#}$. If $t$ is contained in a Borel subgroup $B$, then $P(G,t,2) = 0$ since $b(G,G/B) \geqs 3$. Therefore, by replacing $t$ by a suitable conjugate, if necessary, we may assume that $t = s^m \in H$ for some positive integer $m$. Clearly, $P(G,t,2)=0$ if $|t|=2$, so assume $|t|>2$. Then the map $s^g \mapsto t^g$ is a bijection from $s^G$ to $t^G$ and we observe that if $\{t^{g_1}, t^{g_2}\}$ is a TDS then $\{s^{g_1}, s^{g_2}\}$ is a TDS. Therefore 
$P(G,t,2) \leqs P(G,s,2)$ and the result follows.
\end{proof}

\section{Classical groups}\label{s:cla}

In the previous section, we studied the two-dimensional linear groups ${\rm L}_{2}(q)$ and we now turn our attention to finite simple classical groups in general. We will focus on the uniform domination number and our aim is to prove the results stated in parts (ii) and (iii) of Theorem \ref{t:classmain} (recall that part (i) is established in \cite{BH}).

The main result of this section is the following, which makes substantial progress towards a complete classification of the finite simple classical groups $G$ with $\gamma_u(G)=2$. As in the introduction, write 
\begin{equation}\label{e:ccol}
\mathcal{C} = \{ {\rm PSp}_{2r}(q) \, : \, \text{$r \geqs 5$ odd, $q$ odd} \} \cup \{ {\rm P\O}^{\pm}_{2r}(q) \, : \, \text{$r \geqs 4$ even} \}.
\end{equation}

\begin{thm}\label{t:classmain2}
Let $G$ be a finite simple classical group. Then $\gamma_u(G)=2$ only if one of the following holds:
\begin{itemize}\addtolength{\itemsep}{0.2\baselineskip}
\item[{\rm (i)}] $G = {\rm L}_{2}(q)$ with $q \geqs 11$ odd;
\item[{\rm (ii)}] $G = {\rm L}_{n}^{\e}(q)$ with $n$ odd and 
$
(n,q,\e) \not\in \{(3,2,+), (3,4,+), (3,3,-), (3,5,-)\}
$;
\item[{\rm (iii)}] $G \in \mathcal{C}$.
\end{itemize}
Moreover, $\gamma_u(G)=2$ in cases {\rm (i)} and {\rm (ii)}. In addition, for the groups $G$ in part {\rm (ii)} we have  
\begin{itemize}\addtolength{\itemsep}{0.2\baselineskip}
\item[{\rm (a)}] $P_2(G)>\frac{1}{2}$, unless $G = {\rm U}_{5}(2)$ with $P_2(G) =  \frac{605}{1728}$; and
\item[{\rm (b)}] $P_2(G) \to 1$ as $|G| \to \infty$.
\end{itemize}
\end{thm}

\begin{rem}
We have been unable to determine if $\gamma_u(G)=2$ for the groups $G \in \mathcal{C}$ and we refer the reader to Remarks \ref{r:case3} and \ref{r:case4} for a brief discussion of the difficulties that arise in these special cases.
\end{rem}

We present the proof of Theorem \ref{t:classmain2} in a sequence of lemmas. We begin by recording some useful preliminary results.

\begin{lem}\label{l:reducible}
Let $G$ be a finite simple classical group and let $s \in G$ be an element that acts reducibly on the natural module $V$. Then one of the following holds:  
\begin{itemize}\addtolength{\itemsep}{0.2\baselineskip}
\item[{\rm (i)}] $s$ is contained in a proper subgroup $H$ of $G$ with $b(G,G/H) \geqs 3$.
\item[{\rm (ii)}] $G = {\rm U}_{2m}(q)$, $m$ is odd and $s$ is a regular semisimple element that fixes an orthogonal decomposition $V = U \perp U^{\perp}$ into nondegenerate $m$-spaces and acts irreducibly on both summands. 
\item[{\rm (iii)}] $G = {\rm P\O}_{2m}^{+}(q)$, $m$ is even and $s$ is a regular semisimple element that fixes an orthogonal decomposition $V = U \perp U^{\perp}$ into nondegenerate minus-type $m$-spaces and acts irreducibly on both summands. 
\end{itemize}
\end{lem}

\begin{proof}
Let $U$ be a proper nonzero subspace of $V$ fixed by $s$ of minimal possible dimension. In particular, note that $s$ acts irreducibly on $U$. Since $s$ also fixes the radical $U \cap U^{\perp}$, we deduce that $U$ is either totally singular or nondegenerate (recall that if $G$ is linear, then every subspace of $V$ is totally singular). Let $G_U$ be the stabiliser of $U$ in $G$ and set $H = N_G(G_U)$. Let $n$ be the dimension of $V$.

If $U$ is totally singular then $H$ is a maximal parabolic subgroup and it is easy to check that $|H|^2>|G|$ (see \cite{WM}, for example), whence
\[
b(G,G/H) \geqs \frac{\log |G|}{\log |G/H|} > 2 
\]
and (i) holds. Therefore, we may assume $U$ is nondegenerate and $s$ is not contained in a proper parabolic subgroup of $G$, so Lemma \ref{l:par} implies that $s$ is a regular semisimple element. If $G = {\rm PSp}_{n}(q)'$ then $|H|^2>|G|$ and thus (i) holds. If $G = \O_n(q)$ then $\dim U = 1$ (since every element of $G$ fixes a $1$-space) and once again we see that $|H|^2>|G|$. 

We have now reduced to the case where $U$ is a nondegenerate $\ell$-space and $G$ is a unitary or even-dimensional orthogonal group. Note that $s$ fixes the orthogonal decomposition $V = U \perp U^{\perp}$ and $\ell \leqs n/2$. 

First assume $G = {\rm U}_{n}(q)$. Since $s$ acts irreducibly on $U$, we see that $\ell$ is odd. If $n=2\ell$, then (ii) holds, so we may assume that $\ell < n/2$. Let $W$ be any nondegenerate $\ell$-dimensional subspace of $V$. We claim that there exists a nontrivial element of $G$ that fixes both $U$ and $W$. In particular, $b(G,G/H) \geqs 3$ and thus (i) holds. To see this, consider the sum $X=U+W$. Begin by assuming that $X$ is degenerate and fix a nonzero vector $v$ in the radical $X \cap X^{\perp}$. Then $X \subseteq \la v \ra^{\perp}$ and we can define a transvection $g \in G$ that acts trivially on the hyperplane $\la v \ra^{\perp}$. In particular, $g$ fixes $U$ and $W$. Now assume that $X$ is nondegenerate and note that $V = X \perp X^\perp$. Suppose that $\dim{X} = n-1$. Then, since $\ell < n/2$, we must have $n=2\ell+1$. The element $s$ stabilises $U^\perp$, a subspace of dimension $\ell+1$, which is even. Therefore, $s$ stabilises a nonzero subspace of $U^\perp$ of dimension strictly less than $\ell$, which is a contradiction. Therefore, $\dim{X} \leqs n-2$, and there is clearly an element $g \in G^{\#}$ that acts trivially on $X$, so $g$ fixes both $U$ and $W$. This implies that $b(G,G/H) \geqs 3$ as required. 

Finally, suppose $G = {\rm P\O}_{n}^{\e}(q)$ with $n=2m \geqs 8$. Let $(\, , \,)$ be the symmetric bilinear form on $V$ corresponding to the quadratic form defining $G$. Since $s$ acts irreducibly on $U$, we deduce that $\ell$ is even and $U$ is minus-type. If $m=\ell$, then $U$ and $U^{\perp}$ must both be minus-type spaces, so $\e=+$ and case (iii) holds, so we may assume that $\ell < m$. 

Let $W$ be any nondegenerate $\ell$-dimensional subspace of $V$ of the same type as $U$. Again we claim that there exists a nontrivial element of $G$ that fixes both $U$ and $W$. Write $X=U+W$. If $X$ is nondegenerate, then since $\dim X \leqs 2m-2$, there exists $g \in G^{\#}$ that acts trivially on $X$ and consequently stabilises $U$ and $W$. 

Now assume $X$ is degenerate. If there exist linearly independent vectors $u,v$ in $X \cap X^\perp$, then $X \subseteq \< u, v \>^\perp$ and the long root element defined as $x \mapsto x + (x,u)v - (x,v)u$ (see \cite[Section~3.7.3]{Wil}) acts trivially on $\<u,v\>^\perp$ and hence stabilises $U$ and $W$. Now assume $X \cap X^\perp = \<u\>$ is $1$-dimensional and write $X = Y \perp \< u \>$ where $Y$ is nondegenerate. Then $V = Y \perp Y^\perp$ and $u \in Y^\perp$. Since $Y^\perp$ is nondegenerate, there exists $v \in Y^\perp$ such that $(u,v)=1$, which implies that $\<u,v\> \subseteq Y^\perp$ is nondegenerate. Therefore, $Z = Y \perp \<u,v\>$ is a nondegenerate subspace of $V$ containing $X$. Since $\ell<m$, it follows that $\dim Z \leqs 2\ell+1<2m$ and thus $Z$ is a proper nondegenerate subspace of $V$ containing $X$. If $\dim Z \leqs 2m-2$, then there exists $g \in G^{\#}$ that acts trivially on $Z$, so we may assume $\dim Z = 2m-1$ (in which case, $q$ must be odd). Now $\dim Y^{\perp} = 3$ and we can define an element $g \in G^{\#}$ so that it acts trivially on $Y$ and as a regular unipotent element on $Y^{\perp}$. Moreover, we may choose $g$ so that it fixes the singular vector $u \in Y^{\perp}$. Then $g$ acts trivially on $X = Y \perp \< u \>$ and thus $g$ fixes $U$ and $W$. We conclude that $b(G,G/H) \geqs 3$ and (i) holds. This completes the proof.
\end{proof}

\begin{rem}
In the proof of Lemma \ref{l:u2m} we will show that if $G$ and $s$ are as in part (ii) of Lemma \ref{l:reducible}, then the conclusion stated in part (i) still holds (with $H$ the stabiliser of the given orthogonal decomposition of $V$).
\end{rem}

\begin{cor}\label{c:reducible}
Let $G$ be one of the following finite simple classical groups:
\[
\mbox{${\rm U}_{2m}(q)$ $(m$ even$)$, $\O_{n}(q)$, ${\rm P\O}_{2m}^{+}(q)$ $(m$ odd$)$.}
\]
Then $\gamma_u(G) \geqs 3$.
\end{cor}

\begin{proof}
Notice that every element $s \in G$ acts reducibly on the natural module for $G$ and thus Lemma \ref{l:reducible} implies that $s$ is contained in a proper subgroup $H$ with $b(G,G/H) \geqs 3$. Now apply Lemma \ref{l:udn} to conclude.
\end{proof}

Next we recall the definition of a Singer cycle.

\begin{defn}\label{d:singer}
Let $G$ be a finite simple classical group over $\mathbb{F}_q$ with natural module $V$. An element $s \in G$ is a \emph{Singer cycle} if $\la s \ra$ is an irreducible subgroup (with respect to the action on $V$) of maximal possible order. In particular, $s$ is a regular semisimple element and 
\renewcommand{\arraystretch}{1.2}
\[
|s| = \left\{\begin{array}{ll}
\frac{q^n-1}{(q-1)(n,q-1)} & \mbox{if $G = {\rm L}_{n}(q)$} \\
\frac{q^n+1}{(q+1)(n,q+1)} & \mbox{if $G = {\rm U}_{n}(q)$ and $n$ is odd} \\
\frac{q^{n/2}+1}{(2,q-1)} & \mbox{if $G = {\rm PSp}_{n}(q)$ or ${\rm P\O}_{n}^{-}(q)$}
\end{array}\right.
\]
where $n = \dim V$. Note that any element in $G$ that acts irreducibly on $V$ is equal to $s^k$ for some Singer cycle $s \in G$ and integer $k$. Let us also note that none of the groups ${\rm U}_{n}(q)$ ($n \geqs 4$ even), $\O_{n}(q)$ and ${\rm P\O}_{n}^{+}(q)$ contain elements that act irreducibly on $V$, so Singer cycles do not exist in these cases.
\end{defn}
\renewcommand{\arraystretch}{1}

In the statement of the next result, we say that $H$ is a \emph{field extension subgroup} if it is contained in Aschbacher's $\C_3$ collection of maximal subgroups of $G$ (see \cite[Section 4.3]{KL}). 

\begin{lem}\label{l:unique}
Let $G$ be a finite simple classical group over $\mathbb{F}_q$ with natural module $V$ and let $s \in G$ be a Singer cycle. Let $H$ be a field extension subgroup of $G$ containing $s$. Then $s$ is contained in a unique conjugate of $H$.
\end{lem}

\begin{proof}
First observe that $C_G(s) = C_H(s) = \<s\>$, so it suffices to show that $s^G \cap H = s^H$. By considering the structure of $H$, we see that $H$ contains a unique conjugacy class of maximal tori of order $|s|$ and we can complete the proof by repeating the argument in the proof of Lemma \ref{l:e72}.
\end{proof}

\begin{lem}\label{l:psp2m}
Suppose $G = {\rm PSp}_{2m}(q)'$, where $m \geqs 2$ and either $m$ or $q$ is even. Then $\gamma_u(G) \geqs 3$.
\end{lem}

\begin{proof}
First assume $m$ is even and fix an element $s \in G^{\#}$. If $s$ acts reducibly on the natural module for $G$, then Lemma \ref{l:reducible} implies that $P(G,s,2) = 0$. Now assume $s$ acts irreducibly, in which case $s = x^k$ for some Singer cycle $x \in G$ and integer $k$. Now $\M(G,x)$ (and thus $\M(G,s)$) contains a field extension subgroup $H$ of type ${\rm Sp}_{m}(q^2)$ and it is easy to check that $|H|^2>|G|$ (note that $|H| = 2|{\rm PSp}_{m}(q^2)|$ by  \cite[Proposition 4.3.10]{KL}). Therefore $b(G,G/H) \geqs 3$ and we conclude that $P(G,s,2) = 0$, so $\gamma_u(G) \geqs 3$. 

Finally, if $m \geqs 3$ is odd and $q$ is even, then \cite[Theorem 6.3(iii)]{BH} implies that $\gamma_u(G) \geqs m$ and the result follows.
\end{proof}

\begin{lem}\label{l:psp6}
If $G = {\rm PSp}_{6}(q)$, then $\gamma_u(G) \geqs 3$.
\end{lem}

\begin{proof}
In view of the previous lemma, we may assume $q$ is odd and it suffices to show that $P(G,s,2)=0$ when $s \in G$ is a Singer cycle. Since $\M(G,s)$ contains a field extension subgroup $H$ of type ${\rm Sp}_2(q^3)$, it is sufficient to show that $b(G,G/H) \geqs 3$. 
Write $G = \what{G}/Z$, where $\what{G} = {\rm Sp}_{6}(q)$ and $Z = Z(\what{G})$. By applying \cite[Lemmas 2.2 and 4.1]{GG}, 
we see that there is an element $x \in {\rm GL}_{6}(q)$ such that 
$\what{G} \cap \what{G}^x = {\rm Sp}_{2}(q^3)$. In particular, if $y \in \what{G}$  then
\[
{\rm Sp}_{2}(q^3) \cap {\rm Sp}_{2}(q^3)^y = \what{G} \cap \what{G}^x \cap \what{G}^{xy}
\]
and thus \cite[Lemma 5.7]{GG} implies that $Z$ is a proper subgroup of 
${\rm Sp}_{2}(q^{3}) \cap {\rm Sp}_{2}(q^3)^y$. Therefore, by passing to the quotient group $G = \what{G}/Z$, we deduce that the intersection of any two conjugates of $H$ in $G$ is nontrivial and thus $b(G,G/H) \geqs 3$. 
\end{proof}

\begin{lem}\label{l:l2m}
If $G = {\rm L}_{n}(q)$, where $n \geqs 4$ is even, then $\gamma_u(G) \geqs 3$.
\end{lem}

\begin{proof}
Set $n=2m$ with $m \geqs 2$. As in the proof of the previous lemma, it suffices to show that $P(G,s,2) = 0$ for a Singer cycle $s \in G$. To see this, first observe that $\M(G,s)$ contains a field  extension subgroup $H$ of the form ${\rm GL}_{m}(q^2)$. According to \cite[Proposition 4.3.6]{KL}, we have
\[
|H| = \frac{2(q+1)|{\rm PGL}_{m}(q^2)|}{(2m,q-1)}
\]
and one checks that $|H|^2 > |G|$ when $q=2$, so $b(G,G/H) \geqs 3$ and the desired result follows.

Now assume $q>2$. Set $\bar{G} = {\rm SL}_{2m}(K)$, where $K$ is the algebraic closure of $\mathbb{F}_{q}$, and let $\s$ be a Steinberg endomorphism of $\bar{G}$ such that $\bar{G}_{\s} = {\rm SL}_{2m}(q)$. Then there exists a maximal closed $\s$-stable subgroup $\bar{H}$ of $\bar{G}$ such that $\bar{H}_{\s}$ is of type ${\rm GL}_{m}(q^2)$. Here $\bar{H}$ is of type ${\rm GL}_m(K) \wr S_2$ and without loss of generality we may assume that $G = \bar{G}_{\s}/Z$ and $H = \bar{H}_{\s}/Z$ with  
$Z=Z(\bar{G}_{\s})$. By \cite[Corollary 3.15]{BGS2}, we have $\dim (\bar{H}\cap \bar{H}^g) \geqs m-1$ for all $g \in \bar{G}$ and thus \cite[Proposition 8.1]{GG} implies that the intersection of any two conjugates of $\bar{H}_{\s}$ in $\bar{G}_{\s}$ is nontrivial. In fact, the proof of \cite[Proposition 8.1]{GG} implies that any such intersection either contains a nontrivial unipotent element, or it has order at least $(q-1)^{m-1}$. Now $|Z| = (2m,q-1)$ and by excluding the cases $(m,q)=(2,3), (2,5)$ we deduce that $Z$ is a proper subgroup of every such intersection and it follows that the intersection of any two conjugates of $H$ in $G$ is nontrivial. It is straightforward to check directly that the same conclusion holds when $(m,q)=(2,3)$ or $(2,5)$. We conclude that 
$b(G,G/H) \geqs 3$ and the proof of the lemma is complete.
\end{proof}

\begin{lem}\label{l:u2m}
If $G = {\rm U}_{n}(q)$, where $n \geqs 4$ is even, then $\gamma_u(G) \geqs 3$.
\end{lem}

\begin{proof}
First recall that every element of $G$ acts reducibly on the natural module. Therefore, by Lemma \ref{l:reducible}, we may assume that $n = 2m$ with $m \geqs 3$ odd, and it suffices to show that $b(G,G/H) \geqs 3$ for a maximal subgroup $H$ of type ${\rm GU}_{m}(q) \wr S_2$. If $q=2$ then $|H|^2>|G|$ and the result follows, so we may assume $q>2$. We can now repeat the argument in the proof of the previous lemma, working with the algebraic group $\bar{G} = {\rm SL}_{2m}(K)$ and an appropriate Steinberg endomorphism. We omit the details.
\end{proof}

To complete the proof of Theorem \ref{t:classmain}, we may assume that $G = {\rm L}_{n}^{\e}(q)$ with $n$ odd. We start by studying the special case $n=3$.

\begin{lem}\label{l:lu3}
If $G = {\rm L}_{3}^{\e}(q)$, then 
\[
\gamma_u(G) = \left\{\begin{array}{ll}
4 & \mbox{if $(\e,q) = (+,4)$} \\
3 & \mbox{if $(\e,q) \in \{ (+,2), (-,3), (-,5)\}$} \\
2 & \mbox{otherwise.}
\end{array}\right.
\]
Moreover, if $\gamma_u(G) = 2$ then 
\[
\renewcommand{\arraystretch}{1.4}
P_2(G) = \left\{\begin{array}{ll}
\frac{(q^2+\e q +1)(q^2 - \e q -3)}{q^2(q^2-1)} & \mbox{if $q \equiv 0 \imod{3}$} \\
\frac{3q^5-5q^3+3q+\e 8}{3q^3(q^2-1)} & \mbox{if $q \equiv \e \imod{3}$} \\
\frac{(q^3-\e3q^2+q+\e2)(q^2+\e q +1)}{q^3(q-\e)^2} & \mbox{otherwise}
\end{array}\right.
\]
and thus $P_2(G) \geqs \frac{13}{24}$, with equality if and only if $G = {\rm L}_{3}(3)$. In particular, $P_2(G) \to 1$ as $q \to \infty$.
\end{lem}
\renewcommand{\arraystretch}{1}

\begin{proof}
The special cases with $(\e,q) \in \{ (+,2), (+,4), (-,3), (-,5)\}$ can be checked directly with the aid of {\sc Magma}. For the remainder, we may assume $G$ does not correspond to one of these cases.

Let $s \in G^{\#}$. If $s$ is reducible then $P(G,s,2) = 0$ by Lemma \ref{l:reducible}, so let us assume $s$ is irreducible, in which case $s=x^k$ for some Singer cycle $x \in G$ and integer $k$. Since $C_G(s) = C_G(x) = \la x \ra$, it follows that $P(G,s,2) \leqs P(G,x,2)$. Therefore, we may as well assume $s \in G$ is a Singer cycle, so $P_2(G) = P(G,s,2)$.
Set $a = (q^2 + \e q +1)/(3, q-\e)$. By applying the main theorem of \cite{Ber} and Lemma \ref{l:unique}, we deduce that $\M(G,s) = \{H\}$, where $H =  N_G(\la s \ra) = C_a{:}C_3$ is a field extension subgroup of type ${\rm GL}_{1}^{\e}(q^3)$. By Lemma \ref{l:udn2} we have $P_2(G) = r|H|^2/|G|$ and so it remains to determine the number $r$ of regular orbits of $H$ on $G/H$.

By arguing as in the proof of \cite[Proposition 3.2]{BGiu}, we see that
\begin{equation}\label{e:req}
r = \frac{|G:H| - a(b-1)-1}{|H|} \mbox{ and } b = \frac{2a|G:H|}{|y^G|}
\end{equation}
for any element $y \in H$ of order $3$. Since such an element $y$ is regular, it follows that  
\[
|y^G| = \left\{\begin{array}{ll}
q(q^2-1)(q^3-\e) & \mbox{if $q \equiv 0 \imod{3}$} \\
\frac{1}{3}q^3(q+\e)(q^2+\e q +1) & \mbox{if $q \equiv \e \imod{3}$} \\
q^3(q^3-\e) & \mbox{otherwise}
\end{array}\right.
\]
and one can check that this gives the stated expression for $P_2(G)$. 
\end{proof}

Next we handle the case where $n \geqs 5$ is a prime.

\begin{lem}\label{l:luprime}
If $G = {\rm L}_{n}^{\e}(q)$ with $n \geqs 5$ prime, then $\gamma_u(G) = 2$ and either 
$P_2(G) > \frac{1}{2}$, or $G = {\rm U}_{5}(2)$ and $P_2(G) = \frac{605}{1728}$. Moreover, $P_2(G) \to 1$ as $|G|$ tends to infinity.
\end{lem}

\begin{proof}
First assume $G = {\rm U}_{5}(2)$ and let $s \in G$ be a Singer cycle, so $s$ has order $11$ and $P(G,s,2) = P_2(G)$. Then $\M(G,s) = \{H\}$ with $H = {\rm L}_{2}(11)$ and using {\sc Magma} we calculate that $H$ has $11$ regular orbits on $G/H$. By applying Lemma \ref{l:udn2}, we deduce that $P_2(G) = \frac{605}{1728}$.

For the remainder, let us assume $G \ne {\rm U}_5(2)$. Fix a Singer cycle $s \in G$ and observe that $\M(G,s) = \{H\}$, where $H = N_G(\la s \ra) = C_a{:}C_n$ and
\[
a = |s| = \frac{q^n-\e}{(q-\e)(n,q-\e)}.
\]

Let $y \in H$ be any element of order $n$ and set $b = 2a|G:H|/|y^G|$. Then \eqref{e:req} holds, where $r$ denotes the number of regular orbits of $H$ on $G/H$. Since $y$ is a regular element of $G$, it follows that $|C_G(y)| \leqs (q+1)^{n-1}$ and thus $b \leqs (1-n^{-1})(q+1)^{n-1}$. Now, 
\[
P(G,s,2) = \frac{r|H|^2}{|G|} \geqs 1 - \frac{b|H|^2}{n|G|}
\]
and we observe that 
\[
|H| \leqs \left(\frac{q^n-1}{q-1}\right)n,\;\; |G| > (2n)^{-1}q^{n^2-1}.
\]
Therefore
\[
P(G,s,2) > 1 - 2n^2\left(\frac{q^n-1}{q-1}\right)^2(1-n^{-1})(q+1)^{n-1}\cdot q^{1-n^2} > 1 - 2n^2q^{4n-2-n^2}
\]
and the result follows.
\end{proof}

Finally, to complete the proof of Theorem \ref{t:classmain2} we address the general case with $n$ composite.

\begin{lem}\label{l:lucomp}
If $G = {\rm L}_{n}^{\e}(q)$ with $n \geqs 3$ composite and odd, then $\gamma_u(G)=2$ and $P_2(G) > \frac{1}{2}$. Moreover, $P_2(G) \to 1$ as $|G|$ tends to infinity.
\end{lem}

\begin{proof}
Let $s \in G$ be a Singer cycle and note that $n \geqs 9$. By \cite{Ber} and Lemma \ref{l:unique}, we have 
\[
\M(G,s) = \{H_k \, : \, k \in \pi(n)\}
\] 
where $H_k$ is a field extension subgroup of type ${\rm GL}_{n/k}^{\e}(q^k)$ and $\pi(n)$ is the set of prime divisors of $n$. More precisely, \cite[Proposition 4.3.10]{KL} gives
\begin{equation}\label{e:bk}
H_k  = B_k.k \leqs \left(\left(\frac{q^k-\e}{q-\e}\right).{\rm PGL}_{n/k}^{\e}(q^k)\right).\la \varphi \ra,  
\end{equation}
where $B_k$ is the image of ${\rm GL}_{n/k}^{\e}(q^k)$ in ${\rm L}_{n}^{\e}(q)$ and $\varphi$ is a field automorphism of order $k$. 

Let $\{x_1, \ldots, x_a\}$ be a set of representatives of the conjugacy classes in $G$ of elements of prime order and set 
\[
\what{Q}(G,s,2) := \sum_{i=1}^{a}|x_i^G|\left(\sum_{k \in \pi(n)}{\rm fpr}(x_i,G/H_k)\right)^2.
\]
In order to prove the lemma, it suffices to show that $\what{Q}(G,s,2)<\frac{1}{2}$ and $\what{Q}(G,s,2) \to 0$ as $|G| \to \infty$. To do this, it will be convenient to observe that 
$|\pi(n)| < \log n$, so 
\[
\what{Q}(G,s,2) < \a\log n
\]
where
\[
\a = 
\sum_{k \in \pi(n)}\left(\sum_{i=1}^{a}|x_i^G|\cdot {\rm fpr}(x_i,G/H_k)^2\right). 
\]
Let $V$ be the natural module for $G$ and let $p$ be the characteristic of $\mathbb{F}_q$. Let $\bar{G} = {\rm PSL}_{n}(K)$, where $K$ is the algebraic closure of $\mathbb{F}_q$, and let $\s$ be a Steinberg endomorphism of $\bar{G}$ such that $(\bar{G}_{\s})' = G$. For any element $x \in G$, write $x = \hat{x}Z$ with $\hat{x} \in {\rm GL}_{n}^{\e}(q)$ and $Z = Z({\rm GL}_{n}^{\e}(q))$, and let $\nu(x)$ be the codimension of the largest eigenspace of $\hat{x}$ as an element of ${\rm GL}_{n}(K)$ (note that this is independent of the choice of $\hat{x}$). Write
\[
\{x_1, \ldots, x_a \} = \{y_1, \ldots, y_b \} \cup \{z_1, \ldots, z_c \},
\]
where $\nu(y_i)<n/2$ and $\nu(z_i) \geqs n/2$ for all $i$.

It will be useful to write $\a = \a_1+\a_2+\a_3$, where the $\a_i$ are defined as follows. Firstly, if $3 \in \pi(n)$ then 
\[
\a_1 = \sum_{i=1}^{a}|x_i^G|\cdot {\rm fpr}(x_i,G/H_3)^2,
\]
otherwise $\a_1=0$. Similarly, 
\[
\a_2 = \sum_{k \in \pi(n),\, k \geqs 5}\left(\sum_{i=1}^{b}|y_i^G|\cdot {\rm fpr}(y_i,G/H_k)^2\right)
\]
and
\[
\a_3 = \sum_{k \in \pi(n),\, k \geqs 5}\left(\sum_{i=1}^{c}|z_i^G|\cdot {\rm fpr}(z_i,G/H_k)^2\right)
\]
if $\pi(n) \ne \{3\}$, otherwise $\a_2=\a_3=0$.

First consider $\a_3$. Fix $k \in \pi(n)$ with $k \geqs 5$ and set $H = H_k$. Let $x \in H$ be an element of prime order with $\nu(x) \geqs n/2$.  
Then 
\[
|H|< 2kq^{\frac{1}{k}n^2-1} \leqs 10q^{\frac{1}{5}n^2-1},\;\; |x^G| > \frac{1}{2}\left(\frac{q}{q+1}\right)^2q^{\frac{1}{2}n^2-1}
\]
(see \cite[Corollary 3.38]{Bur2}) and thus 
\[
\sum_{i=1}^{c}|z_i^G|\cdot {\rm fpr}(z_i,G/H)^2 < 2\left(\frac{q+1}{q}\right)^2q^{1-\frac{1}{2}n^2}\left(10q^{n^2/5-1}\right)^2 = 200\left(\frac{q+1}{q}\right)^2q^{-\frac{1}{10}n^2-1}.
\]
It follows that $\a_3 \leqs \delta_1$, where $\delta_1 = 0$ if $\pi(n) = \{3\}$, otherwise
\[
\delta_1 = \left(200\left(\frac{q+1}{q}\right)^2q^{-\frac{1}{10}n^2-1}\right)\log n.
\]

Now let us turn to $\a_1$ and $\a_2$. Fix $k \in \pi(n)$ and set $H = H_k$ and $B=B_k$ as in \eqref{e:bk}. In addition, let us define $m= \frac{n}{k}$ and set $\b = \a_1$ if $k = 3$, otherwise
\[
\b = \sum_{i=1}^{b}|y_i^G|\cdot {\rm fpr}(y_i,G/H)^2.
\]
Our aim is to determine a bound $\b < \gamma$ which is valid for all $k \geqs 3$, in which case $\a_1+\a_2 < \gamma\log n$. 

Let $x \in H$ be an element of prime order $r$ and assume $\nu(x)<n/2$ if $k \geqs 5$. There are several cases to consider and we will closely follow the proof of \cite[Proposition 3.1]{Bur3}.

First assume $x^G \cap (H \setminus B)$ is non-empty. Here $r=k$ and $\nu(x) = n(1-k^{-1}) > n/2$ (see \cite[(66)]{Bur3}), so we may assume $k=3$. Now
\[
|x^G|>\frac{1}{6}\left(\frac{q}{q+1}\right)^{2}q^{\frac{2}{3}n^2}
\]
and we calculate that $H \setminus B$ contains fewer than 
\[
2\left(\frac{q^3-\e}{q-\e}\right)\cdot \frac{|{\rm PGL}_{n/3}^{\e}(q^3)|}{|{\rm PGL}_{n/3}^{\e}(q)|} < 8q^{\frac{2}{9}n^2}
\]
such elements. In addition, there are at most $4q^{\frac{2}{9}n^2}$ in $B$ (there are none if $n$ is indivisible by $9$). By applying Lemma \ref{l:bd}, it follows that the contribution to $\b$  from the elements with $x^G \cap (H \setminus B) \ne \emptyset$ is less than
\[
6\left(\frac{q+1}{q}\right)^{2}q^{-\frac{2}{3}n^2}\left(12q^{\frac{2}{9}n^2}\right)^2 = 864\left(\frac{q+1}{q}\right)^{2}q^{-\frac{2}{9}n^2} = \delta_2.
\]
For the remainder, we may assume $x^G \cap H \subseteq B$. 

Next suppose $r=p>2$ and let $\l$ be the partition of $n$ corresponding to the Jordan form of $x$ on $V$ (this uniquely determines the $\bar{G}$-class of $x$). By considering the embedding of ${\rm GL}_{n/k}^{\e}(q^k)$ in ${\rm GL}_{n}^{\e}(q)$, it follows that $\l = (m^{ka_m}, \ldots, 1^{ka_1})$ for some non-negative integers $a_i$ (recall that this notation indicates that $x$ has $ka_i$ Jordan blocks of size $i$). Let $t \geqs 1$ be the number of non-zero $a_i$ in $\l$. Then as explained in the proof of \cite[Proposition 3.1]{Bur3}, we have 
\[
|x^G|>\frac{1}{2}\left(\frac{q}{q+1}\right)^tq^{\dim x^{\bar{G}}-1}
\]
and there are fewer than $2^tq^{\frac{1}{k}\dim x^{\bar{G}}}$ elements in $B$ that are $\bar{G}$-conjugate to $x$. Therefore, the contribution to $\b$ from unipotent elements when $p$ is odd is less than
\[
\sum 2^{2t+1}\left(\frac{q+1}{q}\right)^tq^{1+\left(\frac{2}{k}-1\right)\dim x^{\bar{G}}} < \sum q^{2t+2+\left(\frac{2}{k}-1\right)\dim x^{\bar{G}}},
\]
where the sum is over a set of $\bar{G}$-class representatives $x$ of order $p$ with the appropriate Jordan form on $V$. We claim that 
\[
2t+2+\left(\frac{2}{k}-1\right)\dim x^{\bar{G}} \leqs 12-2n.
\]
If $t=1$ then $\dim x^{\bar{G}} \geqs \frac{1}{2}n^2$ and the desired bound holds. For $t \geqs 2$ we have $n \geqs \frac{1}{2}kt(t+1)$,
\[
\dim x^{\bar{G}} \geqs k^2\left(m(t^2-t) - \frac{1}{4}t^4+\frac{1}{6}t^3+\frac{1}{4}t^2-\frac{1}{6}t\right)
\]
(see \cite[Lemma 3.25]{Bur2}) and the claim quickly follows. Since there are fewer than $2^{n/k}$ partitions of $n/k$, we conclude that the entire contribution to $\b$ from unipotent elements when $p$ is odd is less than
\[
2^{n/3}\cdot q^{12-2n} < q^{5-n} = \delta_3.
\]

Next we focus on the contribution from semisimple elements of odd order, so let us assume $r \ne p$ is odd. Let $i \geqs 1$ be minimal such that $r$ divides $q^i-1$. Similarly, let $i_0 \geqs 1$ be minimal such that $r$ divides $q^{ki_0}-1$ and note that
\[
i_0 = \left\{\begin{array}{ll}
i/k & \mbox{if $k$ divides $i$} \\
i & \mbox{otherwise.}
\end{array}\right.
\]
We define the integer $c = c(i,\e)$ as in \cite[Lemma 3.33]{Bur2}, so
\[
c = \left\{\begin{array}{ll}
2i & \mbox{if $\e=-$ and $i$ is odd} \\
i/2 & \mbox{if $\e=-$ and $i \equiv 2 \imod{4}$} \\
i & \mbox{otherwise.}
\end{array}\right.
\]

First assume $C_{\bar{G}}(x)$ is disconnected and $c=1$. By \cite[Lemma 3.34]{Bur2}, it follows that $r$ divides $n$ and $\nu(x) = n(1-r^{-1})$, so we may assume $k=3$. In addition, the condition $x^G \cap H \subseteq B$ implies that $r \geqs 5$, hence $n \geqs 15$ and $q \geqs 4$ (since $c=1$). Now
\[
|x^G| > \frac{1}{2r}\left(\frac{q}{q+1}\right)^{r-1}q^{n^2\left(1-\frac{1}{r}\right)} \geqs \frac{1}{10}\left(\frac{q}{q+1}\right)^{4}q^{\frac{4}{5}n^2}
\]
and $|B|<2q^{n^2/3-1}$, so the contribution to $\b$ from these elements is less than
\[
40\left(\frac{q+1}{q}\right)^{4}q^{-\frac{2}{15}n^2-2}.
\]

For the remainder of our analysis of semisimple elements of odd order, we may assume 
that either $C_{\bar{G}}(x)$ is connected or $c>1$. 

First assume that $k$ does not divide $i$, so $i_0=i$. As explained in \cite[Section 3.4]{Bur2}, the $G$-class of $x$ is determined by a tuple $(a_1, \ldots, a_t)$ of non-negative integers (where $t = (r-1)/c$) and we have 
\begin{equation}\label{e:dimxg}
\dim x^{\bar{G}} = n^2 - (n-v)^2 - c\sum_{j=1}^{t}a_j^2
\end{equation}
where $v = n - c\sum_{j}a_j$. In addition, if $c=1$ we may assume that $v \geqs a_j$ for all $j$. Since $i_0=i$, it follows that each $a_j$ is divisible by $k$, so $n \geqs kdc$ where $d \geqs 1$ is the number of non-zero $a_j$. As in the proof of \cite[Proposition 3.1]{Bur3}, we have
\[
|x^G \cap H|< 2^{d}q^{\frac{1}{k}\dim x^{\bar{G}}},\;\; |x^G|>\frac{1}{2}\left(\frac{q}{q+1}\right)^dq^{\dim x^{\bar{G}}}
\]
and it follows that the combined contribution to $\b$ from these semisimple elements is less than
\[
\sum 2\left(\frac{q+1}{q}\right)^dq^{-\dim x^{\bar{G}}}\left(2^{d}q^{\frac{1}{k}\dim x^{\bar{G}}}\right)^2 < \sum q^{3d+1-\left(1 - \frac{2}{k}\right)\dim x^{\bar{G}}},
\]
where the sum is over a set of representatives of the relevant $G$-classes. Now one can check that 
\[
\dim x^{\bar{G}} \geqs \left\{\begin{array}{ll}
4nkd - 4k^2d^2 - 2dk^2-2k^2 & \mbox{if $c \geqs 2$} \\
2nkd - k^2d^2 - dk^2 &\mbox{if $c=1$}
\end{array}\right.
\]
and by setting $k=3$ and $d=1$ we deduce that 
\[
3d+1-\left(1 - \frac{2}{k}\right)\dim x^{\bar{G}} \leqs 12-2n.
\]

Now assume $k$ divides $i$, so $i_0 = i/k$. As before, the $G$-class of $x$ is determined by a tuple $(a_1, \ldots, a_t)$ and we write $d$ for the number of non-zero $a_j$, so $n \geqs dc$ and 
\[
|x^G|>\frac{1}{2}\left(\frac{q}{q+1}\right)^dq^{\dim x^{\bar{G}}}.
\]
Note that in this case, the $a_j$ need not be divisible by $k$, in general. Now, if $k=3$ then the proof of \cite[Proposition 3.1]{Bur3} gives 
\begin{equation}\label{e:bd22}
|x^G \cap H|<2^{3d}\left(\frac{q^3}{q^3-1}\right)^dq^{\frac{1}{3}\dim x^{\bar{G}}}
\end{equation}
and we claim that the same bound holds for $k \geqs 5$. 

To see this, let $y \in x^G \cap H$ and write $\nu(y) = t$ and $\nu_0(y) = t_0$ with respect to the natural modules $V$ and $V_0$ for $G$ and ${\rm PGL}_{n/k}^{\e}(q^k)$, respectively. Recall that we may assume $t < n/2$ since $k \geqs 5$. Then $t_0 \leqs t/k$ (see the proof of \cite[Lemma 4.2]{LSh99}) and by appealing to the proof of \cite[Proposition 3.36]{Bur2} we deduce that 
\begin{align*}
|y^{B}| < 2\left(\frac{q^k}{q^k-1}\right)^{kd}q^{kt_0\left(\frac{2n}{k}-t_0-1\right)} & \leqs 2\left(\frac{q^5}{q^5-1}\right)^{5d}q^{t\left(\frac{2n}{k}-\frac{t}{k}-1\right)} \\
& < 4\left(\frac{q^3}{q^3-1}\right)^{d}q^{t\left(\frac{2n}{k}-\frac{t}{k}-1\right)}.
\end{align*}
From \cite[Proposition 3.40]{Bur2} we see that there are fewer than 
\[
\sum_{t_0=1}^{\lfloor t/k \rfloor}q^{kt_0} < 2q^{t}
\]
distinct $B$-classes in $x^G \cap H$, whence
\[
|x^G \cap H|< 8\left(\frac{q^3}{q^3-1}\right)^{d}q^{\frac{1}{k}t(2n-t)}
\]
Now $\dim x^{\bar{G}} \geqs 2t(n-t)$ by  \cite[Proposition 2.9]{Bur2004} and it is easy to check that $\frac{1}{5}t(2n-t) < \frac{2}{3}t(n-t)$. This shows that \eqref{e:bd22} holds for all $k$.

It follows that the contribution to $\b$ from these semisimple elements is less than 
\[
\sum 2^{6d+1}\left(\frac{q+1}{q}\right)^d\left(\frac{q^3}{q^3-1}\right)^{2d}q^{-\frac{1}{3}\dim x^{\bar{G}}} < \sum q^{7d+1-\frac{1}{3}\dim x^{\bar{G}}},
\]
where we sum over a set of $G$-class representatives. In view of \eqref{e:dimxg}, we calculate that 
\[
\dim x^{\bar{G}} \geqs 2ndc-d^2c^2-dc
\]
and thus
\[
7d+1-\frac{1}{3}\dim x^{\bar{G}} \leqs 12-2n.
\]

By combining the above estimates, noting that $G$ has at most $q^{n-1}$ semisimple conjugacy classes, we conclude that the entire contribution to $\b$ from semisimple elements $x \in G$ of odd order with $x^G \cap H \subseteq B$ is less than
\[
q^{n-1}\cdot q^{12-2n} + 40\left(\frac{q+1}{q}\right)^{4}q^{-\frac{2}{15}n^2-2}.
\]

Notice that this is larger than $1$ when $n=9$, so this case requires special attention. Here $k=3$, $i \in \{1,2,3,6,\frac{9}{2}(3-\e)\}$ and we can estimate the contribution to $\b$ by considering each possibility for $i$ in turn. For example, suppose $\e=+$ and $i=6$, so $i_0=2$ and $r$ divides $q^2-q+1$. Then 
\[
|x^G \cap H| \leqs 3\left( \frac{|{\rm GL}_{3}(q^3)|}{|{\rm GL}_{1}(q^3)||{\rm GL}_{1}(q^6)|}\right)<3q^{18},\;\;
|x^G| = \frac{|{\rm GL}_{9}(q)|}{|{\rm GL}_{3}(q)||{\rm GL}_{1}(q^6)|}>\frac{1}{2}q^{66}
\]
and there are fewer than $\frac{1}{6}q(q-1)$ such $G$-classes for a fixed value of $r$. Since $q^2-q+1$ has less than $\log(q^2-q+1)$ odd prime divisors, we deduce that the  total contribution to $\b$ from semisimple elements with $i=6$ is less than 
\[
\frac{1}{6}q(q-1)\log(q^2-q+1) \cdot 2q^{-66}(3q^{18})^2 < q^{-26}.
\]
In a similar fashion, we can estimate the contribution for the other values of $i$. Indeed, one can check that if $n=9$ then the total contribution to $\b$ from semisimple elements of odd order is less than $q^{-8}$ (this estimate is valid for $\e=\pm$). Set
\[
\delta_4 = \left\{\begin{array}{ll}
\displaystyle q^{11-n} + 40\left(\frac{q+1}{q}\right)^{4}q^{-\frac{2}{15}n^2-2} & \mbox{if $n \geqs 15$} \\
q^{-8} & \mbox{if $n=9$.}
\end{array}\right.
\]

To complete the proof of the lemma, it remains to estimate the contribution to $\b$ from involutions. Let $x \in H$ be an involution and first assume $p=2$. Here $x$ has Jordan form $[J_2^{k\ell},J_{1}^{n-2k\ell}]$ on $V$, for some $1 \leqs \ell \leqs \lfloor n/2k \rfloor$, and we get 
\begin{equation}\label{e:44}
|x^G \cap H| < 2q^{2\ell(n-k\ell)},\;\; |x^G|>\frac{1}{2}\left(\frac{q}{q+1}\right)q^{2k\ell(n-k\ell)}.
\end{equation}
Therefore, the contribution to $\b$ is less than 
\[
\sum_{\ell=1}^{\lfloor n/2k \rfloor} 8\left(\frac{q+1}{q}\right)q^{-2\ell(n-k\ell)} < \frac{n}{6}\cdot 8\left(\frac{q+1}{q}\right)q^{6-2n} < q^{2-n}.
\]
Now assume $p \ne 2$. Here $x$ has Jordan form $[-I_{\ell},I_{n-\ell}]$ on $V$ for some $1 \leqs \ell \leqs \lfloor n/2k \rfloor$ and we get the same bounds as in \eqref{e:44}. In particular, the contribution to $\b$ is less than $q^{2-n} = \delta_5$.

By bringing together the above estimates, we deduce that 
\[
\b < \delta_2 +  \delta_3 + \delta_4 + \delta_5 = \gamma
\]
and thus
\[
\what{Q}(G,s,2) < (\a_1+\a_2+\a_3)\log n < (\gamma\log n + \delta_1)\log n.
\]
This implies that $\what{Q}(G,s,2)< \frac{1}{2}$ for all possible values of $n$ and $q$. Moreover, we deduce that $\what{Q}(G,s,2) \to 0$ as $n$ or $q$ tends to infinity. This completes the proof of the lemma.
\end{proof}

This completes the proof of Theorem \ref{t:classmain}. We close this section by commenting on the classical groups arising in part (iii) of Theorem \ref{t:classmain2}; these are the groups that comprise the collection $\mathcal{C}$ (see \eqref{e:ccol}).

\begin{rem}\label{r:case3}
Suppose $G = {\rm PSp}_{2m}(q)$, where $mq$ is odd and $m \geqs 5$. Let $V$ be the natural module for $G$ and fix an element $s \in G^{\#}$. If $s$ acts reducibly on $V$, then Lemma \ref{l:reducible} implies that $s$ is contained in a proper subgroup $H$ with $b(G,G/H) \geqs 3$, so $P(G,s,2)=0$. Now assume $s$ is irreducible, so $s$ is a power of a Singer cycle $x \in G$. Clearly, if $g \in G$ then 
$\{s,s^g\}$ is a TDS only if $\{x,x^g\}$ is a TDS, so $\gamma_u(G)=2$ if and only if $P(G,x,2)>0$. Therefore, we may as well assume that $s$ is a Singer cycle.

By combining Lemma \ref{l:unique} with a theorem of Bereczky \cite{Ber}, we deduce that
\[
\M(G,s) = \{H_t, K \,:\, t \in \pi(m)\},
\]
where $H_t$ is a field extension subgroup of type ${\rm Sp}_{2m/t}(q^t)$, $K$ is a subgroup of type ${\rm GU}_{m}(q)$ and $\pi(m)$ is the set of prime divisors of $m$. The main theorem of \cite{B07} gives 
\[
b(G,G/H_t), b(G,G/K) \in \{2,3,4\},
\]
but the precise base sizes in these cases have not been determined. One can check that $|K|^2<|G|$ and $\what{Q}(G,K,2)>1$ (as defined in \eqref{e:qhat}), so our probabilistic methods do not yield $b(G,G/K) = 2$. In particular, we have $\what{Q}(G,s,2)>1$ and so a different approach is needed to determine whether or not $\gamma_u(G)=2$. For $G = {\rm PSp}_{10}(3)$, which is the smallest group satisfying the conditions in part (iii) of the theorem, we can use 
{\sc Magma} \cite{magma} to show that $b(G,G/K)=2$. Moreover, we can find an element $g \in G$ by random search such that 
\[
H_5 \cap H_5^g = H_5 \cap K^g = K \cap K^g = K \cap H_5^g = 1,
\]
which implies that $\{s, s^g\}$ is a total dominating set and thus $\gamma_u({\rm PSp}_{10}(3))=2$ in this case. We have not been able to estimate 
$P_2({\rm PSp}_{10}(3))$ and the general problem remains open. 
\end{rem}

\begin{rem}\label{r:case4}
Similar difficulties arise when $G = {\rm P\O}_{2m}^{\e}(q)$ and $m \geqs 4$ is even. First assume $\e=-$. As before, we may as well assume that $s \in G$ is a Singer cycle, in which case 
\[
\M(G,s) = \{H_t \, :\, t \in \pi(m)\}
\]
by \cite{Ber}, where $H_t$ is a field extension subgroup of type $O_{2m/t}^{-}(q^t)$ and $\pi(m)$ is the set of prime divisors $t$ of $m$ with $2m/t \geqs 4$. The main theorem of \cite{B07} gives $b(G,G/H_t) \leqs 4$, but the exact base size is not known. In particular, we have $|H_2|^2 < |G|$ and $\what{Q}(G,H_2,2)>1$, so our probabilistic methods will not determine $b(G,G/H_2)$ precisely. For $m=4$ we have $\gamma_u(G) = b(G,G/H_2)$ and we can use {\sc Magma} when $q$ is small. Indeed, if $m=4$ then 
$\gamma_u(G) = 3$ if $q=2$ and $\gamma_u(G) = 2$ if $q \in \{3,5\}$. Moreover, if $q=3$ we calculate that $H_2$ has exactly $10$ regular orbits on $G/H_2$, whence
\[
P_2({\rm P\O}_{8}^{-}(3)) = \frac{2050}{7371}
\]
by Lemma \ref{l:udn2}. A similar computation is out of reach when $q=5$ since the index $|G:H_2|$ is too large.

Now assume $\e=+$. Here every element of $G$ acts reducibly on the natural module $V$, so Lemma \ref{l:reducible} implies that $P(G,s,2)=0$, with the possible exception of the case where $s \in G$ is a regular semisimple element fixing an orthogonal decomposition $V = U \perp U^{\perp}$ into minus-type $m$-spaces, acting irreducibly on both summands. Clearly, in this case $\M(G,s)$ contains a maximal subgroup $H$ of type $O_{m}^{-}(q) \wr S_2$. For $m=4$, a computation with {\sc Magma} shows that $b(G,G/H) = 3$ if $q=2$ and $b(G,G/H)=2$ if $q \in \{3,5\}$, but the exact base size is not known in general (as before, \cite{B07} gives the bound $b(G,G/H) \leqs 4$). In addition, if $(m,q) \ne (4,2)$ then  $|H|^2<|G|$ and $\what{Q}(G,H,2) > 1$, so the probabilistic approach is inconclusive in this case. Finally, let us also observe that further work is needed to determine the complete set of maximal overgroups of $s$ (since $s$ is not a Singer cycle, we cannot appeal to \cite{Ber}).
\end{rem}

\section{Proofs of Theorems~\ref{t:prob2} and~\ref{t:higher}}\label{s:thm9}

In this final section we prove Theorems~\ref{t:prob2} and~\ref{t:higher}.

\begin{proof}[Proof of Theorem~\ref{t:prob2}]
Let $G$ be a finite simple group such that $\gamma_u(G)=2$ and $G \not\in \mathcal{C} \cup \mathcal{T}$. By applying Theorems \ref{t:altmain}, \ref{t:exmain}, 
\ref{t:psl2main} and \ref{t:classmain}, we immediately reduce to the case where $G$ is a sporadic group (see Remark \ref{r:a13} for the precise value of $P_2(A_{13})$).

Let $G$ be a sporadic group and recall that $G \not\in \mathcal{T}$. First assume that $G \not\in \{ \mathrm{Fi}_{23}, \mathbb{B}, \mathbb{M} \}$. As explained in \cite{BH_comp}, we can use the \textsf{GAP} Character Table Library \cite{CTblLib} to compute $\what{Q}(G,s,2)$ precisely for any element $s \in G$. We obtain the following results, where we adopt the {\sc Atlas} \cite{ATLAS} notation for conjugacy classes (we round up real numbers to 3 decimal places):
\[
\begin{array}{ccccccccc}
\hline
G                & {\rm M}_{23} & {\rm J}_1 & {\rm J}_4 & {\rm Ru}  & {\rm Ly}  & {\rm O'N} & {\rm Fi}_{24}' & {\rm Th}  \\
s                & {\tt 23A}    & {\tt 15A} & {\tt 29A} & {\tt 29A} & {\tt 37A} & {\tt 31A} & {\tt 29A}      & {\tt 27A} \\
\what{Q}(G,s,2)  & 0.030        & 0.364     & 0.001     & 0.168     & 0.001     & 0.337     & 0.001          & 0.060     \\
\hline
\end{array}
\]
In particular, we deduce that $P_2(G) > \frac{1}{2}$.

We will now consider the three remaining groups. First let $G = {\rm Fi}_{23}$ and $s \in G$. If $s$ is not in the class {\tt 35A},  then one can verify that there exists $H \in \M(G,s)$ such that $b(G,G/H) \geqs 3$ and consequently $P(G,s,2) = 0$ (the base size for every primitive action of a sporadic group is given in \cite{BOW}). Now assume that $s$ does belong to ${\tt 35A}$. In this case $\M(G,s) = \{ H \}$ with $H = S_{12}$. Therefore, $P_2(G) = P(G,s,2)$, which Lemma~\ref{l:udn2} implies is equal to $r|H|^2/|G|$, where $r$ is the number of regular orbits of $H$ on $G/H$. A computation in {\sf GAP} by Alexander Hulpke establishes $r=1$, so 
\[
P_2(G) = P(G,s,2) = \frac{|H|^2}{|G|} = \frac{7700}{137241}.
\]

Now let $G=\mathbb{B}$ and let $s$ be in {\tt 47A}. Then $\M(G,s) = \{H\}$ where $H = 47{:}23$ (see \cite[Table~IV]{GK}). Since $|x^G| \geqs 10^{10}$ for all prime order elements $x \in G$, we deduce that 
\[
P_2(G) \geqs P(G,s,2) \geqs 1 - \frac{1104^2}{10^{10}} > 1 - 10^{-3}.
\]

Finally, let $G=\mathbb{M}$ and let $s$ be in {\tt 59A}. Then $\M(G,s) = \{H\}$ where $H = {\rm L}_2(59)$ (see \cite[Table~IV]{GK}) and we proceed as in the previous case. Since $|H| = 102660$ and $|x^G| \geqs 10^{19}$ for all prime order elements $x \in G$, we conclude that
\[
P_2(G) \geqs P(G,s,2) \geqs 1 - \frac{102660^2}{10^{19}} > 1 - 10^{-9}.
\]
This completes the proof.
\end{proof}

\begin{rem}\label{r:spor_prob}
Let $G$ be a sporadic simple group with $\gamma_u(G)=2$ and $G \not\in \mathcal{T}$. If there is an element $s \in G$ such that $\M(G,s) = \{H\}$ and $b(G,G/H)=2$, then in some cases we can determine the probability $P_2(G)$ precisely. 

For example, let $G={\rm M}_{23}$ and $s \in G$. If $|s| \ne 23$, then $s$ is contained in a maximal subgroup $H$ with $b(G,G/H) \geqs 3$. Now assume $|s|=23$, in which case $\M(G,s) = \{H\}$ with $H = N_G(\<s\>) = 23{:}11$ and $P_2(G) = P(G,s,2)$. Using {\sc Magma}, we calculate that $H$ has $159$ regular orbits on $G/H$, so Lemma \ref{l:udn2} yields
\[
P_2(G) = \frac{159|H|^2}{|G|} = \frac{13409}{13440} > 0.997.
\]
\end{rem}

\begin{rem}\label{r:spor_T}
In Remarks~\ref{r:case3} and~\ref{r:case4}, we briefly discussed some of the special difficulties that arise when we try to determine if $\gamma_u(G)=2$ for the classical groups in $\mathcal{C}$. Here we discuss the groups in $\mathcal{T} = \{{\rm J}_{3}, {\rm He}, {\rm Co}_{1}, {\rm HN}\}$. For each $G \in \mathcal{T}$, there exists at least one class $s^G$ such that $b(G,G/H) = 2$ for all $H \in \M(G,s)$. Indeed, the relevant classes are as follows:
\[
\begin{array}{ll}
\hline
{\rm J}_3  & {\tt 3B}, \, {\tt 8A}, \, {\tt 9A}, \, {\tt 9B}, \, {\tt 9C}, \, {\tt 12A}, \, {\tt 19A}, \, {\tt 19B} \\
{\rm He}   & {\tt 7C}, \, {\tt 7D}, \, {\tt 7E}, \, {\tt 21A}, \, {\tt 21B} \\
{\rm Co}_1 & {\tt 35A} \\
{\rm HN}   & {\tt 5C}, \, {\tt 5D}, \, {\tt 10D}, \, {\tt 10E}, \, {\tt 15B}, \, {\tt 15C}, \, {\tt 20D}, \, {\tt 20E}, \, {\tt 25A}, \, {\tt 25B}, \, {\tt 30B}, \, {\tt 30C} \\
\hline
\end{array}
\] 
However, $\what{Q}(G,s,2) > 1$ in all cases, and we never have $\M(G,s) = \{H\}$ with $b(G,G/H)=2$. Therefore, our methods are inconclusive and we have not been able to determine if $\gamma_u(G)=2$ in these cases (in particular, we have been unable to apply computational methods to answer this question). However, we can verify the bound $\what{Q}(G,s,3) < 1$ for a suitable element $s$, which implies that $\gamma_u(G) \in \{2,3\}$ (see \cite[Theorem~4.2]{BH}).

Finally, let us remark that if $G={\rm Co}_1$ and $s$ is in {\tt 35A}, then $\M(G,s) = \{ H, K, L_1, L_2 \}$ where $H= (A_5 \times J_2){:}2$, $K = (A_7 \times {\rm L}_2(7)){:}2$ and $L_1 \cong L_2 = (A_6 \times {\rm U}_3(3)){:}2$.
\end{rem}

We now turn to the proof of Theorem~\ref{t:higher}.

\begin{proof}[Proof of Theorem~\ref{t:higher}]
Let $G$ be a finite simple group such that $\gamma_u(G)=2$ and $G \not\in \mathcal{C} \cup \mathcal{T}$. By Lemma~\ref{l:higher}, the claims in parts~(i) and~(ii) are immediate consequences of Theorems~\ref{t:altmain}(iv) and~\ref{t:exmain}(ii). Moreover, Theorem~\ref{t:prob2} implies that $\gamma_u^{(2)}(G) = 3$ unless $G \in \{ A_{13}, {\rm U}_5(2), {\rm Fi}_{23} \}$ or $G = {\rm L}_2(q)$ with $q \geqs 11$ and $q \equiv 3 \imod{4}$. In the first three cases, we can verify the claim in \textsc{Magma} by carrying out a random search. 

Now assume that $G = {\rm L}_2(q)$ for $q \geqs 11$ and $q \equiv 3 \imod{4}$. Let $s \in G$ be an element of order $(q+1)/2$. As noted in the proof of Proposition \ref{p:psl2_2}, we have $\mathcal{M}(G,s) = \{H\}$ with $H = D_{q+1}$. By Lemma~\ref{l:higher2}, it suffices to show that there exist $\a, \b, \g \in G/H$ such that $\{ \a, \b \}$, $\{ \a, \g \}$ and $\{ \b, \g \}$ are bases for the action of $G$ on $G/H$. To prove this, we will identify $G$ with ${\rm PSU}_{2}(q)$ and $G/H$ with the set $\Omega$ of orthogonal pairs of nondegenerate $1$-dimensional subspaces of the natural module $V$ for ${\rm PSU}_{2}(q)$. 

Fix an orthonormal basis $\{u,v\}$ for $V$ and write $\a = \{ \< u \>, \< v \> \}$. Then
\[
\Omega = \{ \a \} \cup \{ \omega_\l \, : \, \l \in \mathbb{F}_{q^2}^\times \ \text{and} \ \l^{q+1} \neq -1 \},
\]
where $\omega_\l = \{ \< u+\l v\>, \<u-\l^{-q} v\> \}$. Note that $\omega_\l = \omega_{-\l^{-q}}$ and the condition $\l^{q+1} \neq -1$ ensures that $\< u+\l v\> \neq \<u-\l^{-q} v\>$. 
We claim that if $\l \in \mathbb{F}_{q^2}^\times$ and $\l^{q+1} \neq -1$, then 
$\{\a, \omega_\l \}$ is a base if and only if $\l$ is a nonsquare in $\mathbb{F}_{q^2}$.

By Proposition~\ref{l:subdegrees}(ii), $H$ has exactly $(q-3)/4$ regular orbits on $\Omega$. Therefore, there are exactly $(q+1)(q-3)/4$ points $\omega \in \Omega$ such that $\{\a,\omega\}$ is a base. Now, there are exactly $(q^2-1)/2-(q+1) = (q+1)(q-3)/2$ nonsquare $\l \in \mathbb{F}_{q^2}^{\times}$ such that $\l^{q+1} \neq -1$. Therefore, it suffices to prove that $\{\a,\omega_\l\}$ is not a base if $\l$ is a square. To this end, suppose that $\l = \kappa^2$ for some $\kappa \in \mathbb{F}_{q^2}^{\times}$ and fix $g = \hat{g}Z({\rm SU}_2(q))$, where 
\[
\hat{g} = 
\left(
\begin{array}{cc}
0             & \kappa^{1-q} \\
-\kappa^{q-1} & 0            \\
\end{array}
\right),
\]
with respect to the ordered basis $(u,v)$. It is straightforward to check that $g$ fixes $\a$ and $\omega_{\l}$, which proves that $\{\a, \omega_\l \}$ is not a base, as claimed.

Now write $\mathbb{F}_{q^2}^\times = \< \mu \>$, $\b = \omega_\mu$ and $\g = \omega_{-\mu}$. Since $\mu$ has multiplicative order $q^2-1$, we see that $\{\a,\b\}$ and $\{\a,\g\}$ are bases. It now remains to prove that $\{\b,\g\}$ is a base. The norm of $u+\mu v$ is $1+\mu^{q+1}$, which is in $\mathbb{F}_q^\times$, so there exists $\nu \in \mathbb{F}_{q^2}^\times$ such that $\nu^{q+1}=1+\mu^{q+1}$. Now $(\nu\mu^{-1})^{q+1} = (1+\mu^{q+1})\mu^{-(q+1)} = 1+\mu^{-(q+1)}$, which is the norm of $u-\mu^{-q}v$, so 
$\{a,b\}$ is an orthonormal basis for $V$, where $a=\nu^{-1}(u + \mu v)$ and $b=\nu^{-1}\mu(u - \mu^{-q} v)$. Moreover, 
it is straightforward to check that
\[
\< u - \mu v \> = \< a + \delta b \>,\;\; \< u - \mu^{-q} v \> = \< a -\delta^{-q}b \>
\]
where $\delta=2(\mu^{-q}-\mu)^{-1}$. Therefore, 
\[
\b = \{\la a\ra, \la b \ra\},\;\; \gamma = \{\la a + \delta b \ra, \la a -\delta^{-q}b\ra\}
\]
and thus the argument in the previous paragraph implies that $\{\b,\g\}$ is a base if and only if $\delta$ is a nonsquare. Since $2$ is square and $\mu$ is nonsquare it remains to prove that $\mu^{-(q+1)}-1$ is square. However, this follows immediately from the fact that $\mu^{-(q+1)}-1 \in \mathbb{F}_q^{\times}$.  Therefore, $\{ \b, \g \}$ is a base for $G$ on $\Omega$, which completes the proof that $\gamma_u^{(2)}(G) = 3$.
\end{proof}

\end{document}